\documentclass[11pt]{amsart}

\usepackage[english]{babel}
\usepackage[utf8]{inputenc}
\usepackage[T1]{fontenc}

\usepackage[margin=1in, letterpaper]{geometry}

\usepackage{amssymb,amsmath,amsfonts,amsthm,enumerate,enumitem,bm,bbm,mathrsfs}
\usepackage[dvipsnames]{xcolor}
\usepackage{url}
\usepackage{enumitem}
\usepackage{mathtools}
\usepackage{placeins}
\usepackage{extarrows}
\usepackage{subfig}
\usepackage[font=small,labelfont=bf]{caption}
\usepackage{floatrow}
\usepackage{verbatim}
\usepackage{wrapfig}
\usepackage{constants}
\usepackage{tikz}
\usetikzlibrary{arrows,arrows.meta,cd,decorations.pathmorphing,backgrounds,positioning,fit,matrix}
\usepackage{xcolor,colortbl,graphics,graphicx}
\usepackage[colorlinks=true, allcolors=Blue, pdfencoding=auto]{hyperref}
\usepackage[noabbrev,capitalize]{cleveref}
\crefname{equation}{}{}

\newtheorem{theorem}{Theorem}[section]
\newtheorem{proposition}[theorem]{Proposition}
\newtheorem{lemma}[theorem]{Lemma}
\newtheorem{hypothesis}[theorem]{Hypothesis}
\newtheorem{corollary}[theorem]{Corollary}

\Crefname{problem}{Problem}{Problems}
\newtheorem*{question*}{Question}
\Crefname{question}{Question}{Questions}

\Crefname{knownlemma}{Lemma}{Lemmas}
\Crefname{knowntheorem}{Theorem}{Theorems}
\newtheorem{definition}[theorem]{Definition}

\Crefname{notation}{Notation}{Notations}
\newtheorem*{notation*}{Notation}

\newtheorem*{example*}{Example}

\theoremstyle{remark}
\newtheorem*{remark}{Remark}

\numberwithin{equation}{section}

\newcommand{\ka}{\mathfrak{a}}

\newcommand{\kc}{\mathfrak{c}}

\renewcommand{\d}{\mathfrak{d}}
\newcommand{\kd}{\mathfrak{d}}

\newcommand{\p}{\mathfrak{p}}
\newcommand{\kp}{\mathfrak{p}}

\newcommand{\m}{\mathfrak{m}}
\newcommand{\km}{\mathfrak{m}}
\newcommand{\n}{\mathfrak{n}}
\newcommand{\kn}{\mathfrak{n}}
\newcommand{\q}{\mathfrak{q}}
\newcommand{\kq}{\mathfrak{q}}
\newcommand{\F}{\mathfrak{F}}

\newcommand{\N}{\textnormal{N}}
\newcommand{\Z}{\mathbb{Z}}
\newcommand{\R}{\mathbb{R}}
\newcommand{\Q}{\mathbb{Q}}
\newcommand{\A}{\mathbb{A}}

\newcommand{\mD}{\mathcal{D}}

\newcommand{\mJ}{\mathcal{J}}

\newcommand{\mL}{\mathcal{L}}
\newcommand{\mO}{\mathcal{O}}

\newcommand{\mS}{\mathcal{S}}

\newcommand{\one}{\mathbbm{1}}

\renewcommand{\Re}{\textnormal{Re}}

\newcommand{\GL}{\textnormal{GL}}

\newcommand{\eps}{\varepsilon}

\renewcommand{\bar}{\overline}
\renewcommand{\hat}{\widehat}
\renewcommand{\tilde}{\widetilde}

\renewcommand{\pmod}[1]{\ (\textnormal{mod } #1)}
\newcommand{\lcm}{\textnormal{lcm}}

\newcommand{\Res}{\textnormal{Res}}

\definecolor{myBlue}{rgb}{0, 0, 0.6}

\newcommand{\re}{\mathop{\mathrm{Re}}}

\renewcommand{\Re}{\mathop{\mathrm{Re}}}
\renewcommand{\epsilon}{\varepsilon}

\newconstantfamily{abcon}{symbol=c}

\title{Large sieves for $\mathrm{GL}_n$ and applications}

\author{Alexandru Pascadi}
\address{Mathematical Institute, Radcliffe Observatory quarter, Woodstock Road, Oxford OX2 6GG, England}
\email{alexpascadi@gmail.com}

\author{Jesse Thorner}
\address{Department of Mathematics, University of Illinois Urbana-Champaign, 1409 West
Green Street, Urbana, IL 61801, USA}
\email{jesse.thorner@gmail.com}

\begin{document}

\begin{abstract}
Let $\mathfrak{F}_n$ be the set of unitary cuspidal automorphic representations of $\mathrm{GL}_n$ over a number field $F$, and let $\mathcal{S}\subseteq\mathfrak{F}_n$ be an arbitrary finite subset. Given $\pi_0\in\mathfrak{F}_{n_0}$, we establish large sieve inequalities for the families $\{L(s,\pi)\colon \pi\in\mathcal{S}\}$ and $\{L(s,\pi\times\pi_0)\colon \pi\in\mathcal{S}\}$ that, unlike previous results, are independent of progress towards the generalized Ramanujan conjecture, and simultaneously handle the Dirichlet coefficients of $L$, $L^{-1}$, and $\log L$. We also give the first such result that improves upon the trivial bound for short sums.  We present several applications, including:
\begin{enumerate}[leftmargin=*]
    \item the strongest bound for $\sum_{\pi\in\mathcal{S}}|L(\frac{1}{2},\pi)|^2$ that holds for arbitrary $\mathcal{S}$,
    \item significant improvements to zero density estimates for families of automorphic and Rankin--Selberg $L$-functions, counting violations to the generalized Riemann hypothesis near $\re(s)=1$,
    \item the removal of all unproven hypotheses in the conditional log-free zero density estimate for families of Rankin--Selberg $L$-functions proved by Brumley, Thorner, and Zaman, and
    \item an improvement of the density theorem for non-archimedean Langlands parameters due to Lichtman and Pascadi, counting violations to the generalized Ramanujan conjecture.
\end{enumerate}
\end{abstract}

\maketitle

\vspace{-0.3cm}
\section{Introduction and statement of main results}
\label{sec:intro}

\subsection{Large sieve inequalities}

The classical multiplicative large sieve inequality \cite[Theorem 7.13]{IK} states that
\begin{equation} \label{eqn:large-sieve-dirichlet}
    \sup_{\|a\|_2 = 1} \sum_{q \le Q}~\sideset{}{^*}\sum_{\chi \pmod{q}} \Bigg| \sum_{n\leq N} a_n \chi(n) \Bigg|^2 \leq N + Q^2 - 1,
\end{equation}
where $\chi$ varies among the primitive Dirichlet characters modulo $q$, and $(a_n)_{n=1}^{N}$ varies among complex sequences with $\|a\|_2 := (\sum_{n\leq N} |a_n|^2)^{1/2} = 1$. This result frequently makes decisive appearances in analytic number theory.  It is crucial in the proof of the Bombieri--Vinogradov theorem, zero density estimates, and bounds for moments of Dirichlet $L$-functions. It encodes the ``quasi-orthogonality'' of Dirichlet characters of varying moduli $q \le Q$, and it is optimal.

The bound \eqref{eqn:large-sieve-dirichlet} turns out to be a particular ($\GL_1$) case of a more general ($\GL_n$) phenomenon, which requires some setup. Let $F$ be a number field with ring of integers $\mathcal{O}_F$, absolute norm $\N = \N_{F/\Q}$, and absolute discriminant $D_F$. Throughout, let $\kp$ (resp.\ $\kn$, $\kd$, and $\kq$) denote nonzero prime ideals (resp.\ nonzero ideals) of $\mathcal{O}_F$. Let $\F_n$ denote the family of all cuspidal automorphic representations of $\GL_n(\A_F)$, normalized to have unitary central characters which are trivial on the diagonally-embedded positive reals. For $n, n_0 \in \mathbb{N}$, $\pi \in \F_n$, and $\pi_0 \in \F_{n_0}$, we recall the standard \cite{bump1998automorphic,goldfeld2011automorphic} and Rankin--Selberg $L$-functions \cite{JPSS,moeglin1989spectre,shahidi1981certain} given in $\Re(s) > 1$ by
\[
    L(s, \pi) = \sum_\n \frac{\lambda_\pi(\n)}{\N \n^{s}}, 
    \qquad\quad 
    L(s, \pi \times \tilde{\pi}_0) = \sum_\n \frac{\lambda_{\pi \times \tilde{\pi}_0}(\n)}{\N \n^{s}}.
\]
We denote by $\kq_{\pi}$ and $\mathfrak{C}_{\pi}$ the \emph{arithmetic} and \emph{analytic} conductors of $\pi$.  Here, $\mathfrak{C}_{\pi}$ follows Iwaniec and Sarnak \cite{IS}, and equals $\N \kq_{\pi}$ multiplied by a contribution from the spectral parameters (see \eqref{eqn:analytic_conductor_def}). Our normalization of the central characters ensures that for $Q\geq 1$, the truncated universal family
\[
\mathfrak{F}_n(Q):=\{\pi\in\mathfrak{F}_n\colon \mathfrak{C}_{\pi}\leq Q\}
\]
is finite \cite[Corollary 9]{Brumley}. In fact, it is expected that $|\mathfrak{F}_n(Q)|\asymp_{F,n}Q^{n+1}$. Brumley and Mili{\'c}evi{\'c} \cite{BM} proved that there exists a constant $\Cl[abcon]{BM}=\Cr{BM}(n,F)>0$ such that
\begin{equation}
\label{eqn:BM_asymptotic}
\#\{\pi\in\mathfrak{F}_n(Q)\colon \textup{$\pi_{\infty}$ spherical}\}=\Cr{BM} Q^{n+1}\Big(1+O_{n,F}\Big(\frac{1}{\log Q}\Big)\Big),
\end{equation}
which implies the sharp lower bound
\begin{equation}
\label{eqn:BM}
    |\mathfrak{F}_n(Q)|\gg_{F,n}Q^{n+1}.
\end{equation}

Now, let $\mathcal{S}\subseteq\mathfrak{F}_n$ be a finite set and $a(\kn)$ be a complex-valued function.  Given $N\geq 1$, we denote\footnote{Given orderings $(\pi_1,\pi_2,\ldots,\pi_i,\ldots)$ of $\mathcal{S}$ with monotonically increasing analytic conductor and $(\kn_1,\kn_2,\ldots,\kn_j,\ldots)$ of the nonzero ideals of $\mathcal{O}_F$ with monotonically increasing norm lying in $[1,x]$, let $A$ be the matrix $[\lambda_{\pi_i}(\kn_j)]$.  The quantity $C(N,\mathcal{S})$ equals the largest eigenvalue of the self-adjoint matrix $\overline{A}^t A$.}
\[
Q := \max_{\pi\in\mathcal{S}}\mathfrak{C}_{\pi}, 
\qquad 
\|a\|_2 := \Bigg(\sum_{\N\kn\leq N}|a(\kn)|^2\Bigg)^{1/2},\qquad C(N,\mathcal{S}):=\sup_{\|a\|_2=1}\sum_{\pi\in\mathcal{S}}\Bigg|\sum_{\N\kn\leq N}\lambda_{\pi}(\kn)a(\kn)\Bigg|^2.
\]
With this notation, the classical large sieve inequality from \eqref{eqn:large-sieve-dirichlet} reads (with $F = \Q$)
\begin{equation}
\label{eqn:GL1_large_sieve}
C(N,\mathfrak{F}_1(Q))\ll N+|\mathfrak{F}_1(Q)|.
\end{equation}
In general, the Cauchy--Schwarz inequality yields $C(N,\mathcal{S})\ll_{n,[F:\Q]} (NQ)^{o(1)}N|\mathcal{S}|$. A $\GL_n$ large sieve inequality is any improvement over this, possibly when $N$ is large with respect to $|\mathcal{S}|$ and $Q$, or vice versa. Up to factors of size $(NQ)^{o(1)}$, the best possible bound is
\begin{equation}
\label{eqn:large_sieve_optimal}
C(N,\mathcal{S})\ll_{n,F}(NQ)^{o(1)}(N+|\mathcal{S}|),
\end{equation}
which can be viewed as a ``quasi-orthogonality'' result for the $\pi\in\mathcal{S}$. There are some choices of $\mathcal{S}$ for which \eqref{eqn:large_sieve_optimal} is provably not attainable \cite{DunnRadziwill,IwaniecLi}. When $\mS = \mathfrak{F}_n(Q)$, establishing \eqref{eqn:large_sieve_optimal} remains a very difficult open problem for all $n \ge 2$.

Assuming the generalized Ramanujan conjecture (GRC) for all $\pi\in\mathcal{S}$, the first general bound on $C(N,\mathcal{S})$ is implicit in the work of Duke and Kowalski \cite[Section 4]{DK}, namely
\begin{equation} \label{eqn:DK_large_sieve_conj}
C(N,\mathcal{S})\ll_{n,[F:\Q]}(NQ)^{o(1)}(N+\sqrt{N}Q^{\frac{n}{2}}|\mathcal{S}|).
\end{equation}
Let
\[
\theta\in\Big[0,\frac{1}{2}-\frac{1}{n^2+1}\Big]
\]
be the best exponent towards GRC that holds for all $\pi\in\mathcal{S}$, so that $\theta=0$ if and only if each $\pi\in\mathcal{S}$ satisfies GRC. The first unconditional result for all $\mS$ follows from the work of Thorner and Zaman \cite{TZ_GLn}.  For a variant $C^{\mathrm{unram}}(N, \mS)$ of $C(N, \mS)$ which incorporates the coprimalty constraint $\gcd(\kn, \kq_{\pi}) = \mathcal{O}_F$ in the inner sum, they established the bound\footnote{This incorporates the correction from \cite[Section 4]{humphries2024zeros}.}
\begin{equation}
\label{eqn:TZ_large_sieve}
C^{\mathrm{unram}}(N, \mS) \ll_{n,[F:\Q]} (NQ)^{o(1)}(N+Q^{4\theta n^2+n}|\mathcal{S}|).
\end{equation}
The exponent $4\theta n^2$ arises from addressing the $\kn$ such that $\gcd(\kn,\kq_{\pi})\neq\mathcal{O}_F$. Recent work of Jiang \cite{Jiang} can also be used to establish the unconditional large sieve inequality
\begin{equation} \label{eqn:Jiang_large_sieve}
C(N,\mathcal{S})\ll_{n,[F:\Q]} (NQ)^{o(1)}(N+N^{\frac{1}{2}+\theta}Q^{n(\frac{1}{2}-\theta)}|\mathcal{S}|),
\end{equation}
which recovers the conditional bound from \eqref{eqn:DK_large_sieve_conj} under GRC.

Finally, a version of \eqref{eqn:TZ_large_sieve} also holds for the Dirichlet coefficients of Rankin--Selberg $L$-functions.  To state such results, we let $\pi_0\in\mathfrak{F}_{n_0}$ and write
\begin{equation*}
\begin{gathered}
\|a\|_{2,\pi_0}:=\Bigg(\sum_{\N\kn\leq N}\lambda_{\pi_0\times\widetilde{\pi}_0}(\kn)|a(\kn)|^2\Bigg)^{1/2}, \qquad 
C(N,\mathcal{S},\pi_0):=\sup_{\|a\|_{2,\pi_0}=1}\sum_{\pi\in\mathcal{S}}\Bigg|\sum_{\N\kn\leq N}\lambda_{\pi\times\pi_0}(\kn)a(\kn)\Bigg|^2.
\end{gathered}
\end{equation*}
Note that if $\mathbbm{1}\in\mathfrak{F}_1$ is the trivial representation (whose $L$-function is the Dedekind zeta function $\zeta_F(s)$), then $C(N,\mathcal{S},\mathbbm{1})=C(N,\mathcal{S})$. Extending the ideas in \cite{TZ_GLn}, Humphries and Thorner \cite{humphries2024zeros} unconditionally proved that for a suitable variant $C^{\mathrm{unram}}(N, \mS, \pi_0)$ of $C(N, \mS, \pi_0)$ which includes the constraint $\gcd(\kn,\kq_{\pi}\kq_{\pi_0})=\mathcal{O}_F$ in the inner sum, one has
\begin{equation}
\label{eqn:HT_large_sieve}
C^{\mathrm{unram}}(N, \mS, \pi_0) \ll_{n,[F:\Q]} (NQ)^{o(1)}(N+Q^{4\theta n^2+n}|\mathcal{S}|).
\end{equation}

If each $\pi\in\mathcal{S}$ has trivial conductor, or each $\pi\in\mathcal{S}$ satisfies GRC, then the $4\theta n^2$ in \eqref{eqn:TZ_large_sieve} and \eqref{eqn:HT_large_sieve} can be replaced with zero. Without hybrid-aspect subconvexity bounds for the $L$-functions in the family $\{L(s,\pi\times\widetilde{\pi}')\colon \pi,\pi'\in\mathcal{S}\}$ or a suitable trace formula\footnote{Large sieve inequalities based on trace formul{\ae} require that all $\pi \in \mS$ are automorphic under the same group. Aside from the proof of \eqref{eqn:large-sieve-dirichlet}, duality seems to be the only viable approach to allowing the group to vary.} for the family $\mathcal{S}$ (see \cite{blomer2023density,BlomerThorner,deshouillers1982kloosterman,jana2020applications}, for example), this seems to be the limit of the current methods.

In this paper, we build on the ideas of Lichtman and Pascadi \cite{lichtman2024density} and Thorner and Zaman \cite{TZ_GLn} to develop a new approach to large sieve inequalities that handles the contribution from ramified prime ideals more efficiently, leading to the {\it unconditional} elimination of the $Q^{4\theta n^2}$ term and the coprimality restrictions in \eqref{eqn:TZ_large_sieve} and \eqref{eqn:HT_large_sieve}.  Our first result is a corollary of our main technical result, \cref{thm:large-sieve-general}, which is also independent of progress towards GRC.

\begin{theorem}
\label{thm:large_sieve2}
If $\mathcal{S}\subseteq \mathfrak{F}_n$ is finite, $\pi_0\in\mathfrak{F}_{n_0}$, $N\geq 1$, and $Q=\max_{\pi\in\mathcal{S}}\mathfrak{C}_{\pi}$, then
\[
\begin{rcases*}
C(N,\mathcal{S})\\
C(N,\mathcal{S},\pi_0)
\end{rcases*}\ll_{n,[F:\Q]} (NQ)^{o(1)}(N+Q^n|\mathcal{S}|).
\]
The factor of $Q^n$ in the right-hand side can be replaced with $\max_{\pi,\pi'\in\mathcal{S}}\mathfrak{C}_{\pi\times\tilde{\pi}'}^{1/2}$.
\end{theorem}

\begin{remark}
\cref{thm:large_sieve2} also holds if, in the definition of $C(N, \mS)$, one replaces $\lambda_\pi(\kn)$ with the $\kn$-th Dirichlet coefficient of $L(s, \pi)^{-1}$ or $\log L(s, \pi)$ (and similarly for $C(N,\mathcal{S},\pi_0)$). See \cref{thm:large-sieve-general}.
\end{remark}

All of the $\GL_n$ large sieve inequalities mentioned above, including \cref{thm:large_sieve2} and the conditional bound from \eqref{eqn:DK_large_sieve_conj}, can only beat the trivial bound $C(N, \mS) \ll_{n,[F:\Q]} (NQ)^{o(1)} N|\mS|$ when $N > Q^n$. In this range, \cref{thm:large_sieve2} significantly improves upon or subsumes \eqref{eqn:DK_large_sieve_conj}--\eqref{eqn:HT_large_sieve}; in particular, this unconditionally establishes \eqref{eqn:DK_large_sieve_conj}. Consequently, \cref{thm:large_sieve2} is sharp (i.e., it matches \eqref{eqn:large_sieve_optimal}) in the widest range of $N$, as summarized in the table below.

\begin{center}
\begin{tabular}{c|c|c|c|c}
    \emph{Result:} & \eqref{eqn:Jiang_large_sieve} & \eqref{eqn:DK_large_sieve_conj} (assumes GRC) & \eqref{eqn:TZ_large_sieve} and \eqref{eqn:HT_large_sieve} & \cref{thm:large_sieve2} 
    \\
    \hline
    \emph{Sharp when:}
    & 
    $N \ge Q^n |\mS|^{\frac{2}{1-2\theta}}$
    & 
    $N \ge Q^n |\mS|^2$
    & 
    $N \ge Q^{4\theta n^2+n}|\mS|$
    & 
    $N \ge Q^n |\mS|$
\end{tabular}
\end{center}

If we replace the norm $\|a\|_2$ from the definition of $C(N, \mS)$ with $N^{1/2}\max_{\N\kn\leq N} |a(\kn)|$ (and similarly for $C(N, \mS, \pi_0)$), then \cref{thm:large_sieve2} undergoes a self-improvement in smaller ranges of $N$ via H{\"o}lder's inequality and a careful treatment of common divisors. To this end, let
\begin{equation*}
\begin{aligned}
\|a\|_{\infty}:=\max_{\N\kn\leq N} |a(\kn)|, \qquad 
&C^\infty(N,\mathcal{S}):=\sup_{N^{1/2} \|a\|_{\infty}=1}\sum_{\pi\in\mathcal{S}}\Bigg|\sum_{\N\kn\leq N}\lambda_{\pi}(\kn)a(\kn)\Bigg|^2,
\\
\|a\|_{\infty,\pi_0}:=\max_{\N\kn\leq N}\lambda_{\pi_0\times\widetilde{\pi}_0}(\kn)^{1/2} |a(\kn)|, \qquad 
&C^\infty(N,\mathcal{S},\pi_0):=\sup_{N^{1/2} \|a\|_{\infty,\pi_0}=1}\sum_{\pi\in\mathcal{S}}\Bigg|\sum_{\N\kn\leq N}\lambda_{\pi\times\pi_0}(\kn)a(\kn)\Bigg|^2.
\end{aligned}
\end{equation*}

\begin{corollary}
\label{cor:large_sieve2}
If $\mathcal{S}\subseteq \mathfrak{F}_n$ is finite, $\pi_0\in\mathfrak{F}_{n_0}$, $N\geq 1$, $Q=\max_{\pi\in\mathcal{S}}\mathfrak{C}_{\pi}$, and $k\geq 1$ is an integer, then
\begin{equation}
\label{eqn:cor-large-sieve2}
\begin{rcases*}
C^{\infty}(N,\mathcal{S})\\
C^{\infty}(N,\mathcal{S},\pi_0)
\end{rcases*}\ll_{k,n,[F:\Q]} (NQ)^{o(1)}(N|\mathcal{S}|^{\frac{k-1}{k}}+Q^{\frac{n}{k}}|\mathcal{S}|).
\end{equation}
The factor of $Q^{n/k}$ in the right-hand side can be replaced with $\max_{\pi,\pi'\in\mathcal{S}}\mathfrak{C}_{\pi\times\tilde{\pi}'}^{1/(2k)}$.
\end{corollary}

\begin{remark}
The bound $N|\mS|^{(k-1)/k} + Q^{n/k}|\mS|$ from \eqref{eqn:cor-large-sieve2} interpolates between the factors $N + Q^n|\mS|$ from \cref{thm:large_sieve2} and $N|\mS|$ from the trivial bound, and is smaller than both of them\footnote{Given values of $N, Q$ and $|\mS|$, the optimal choice of $k$ in \eqref{eqn:cor-large-sieve2} is either the floor or the ceiling of $\log(Q^n |\mS|)/\log N$.} if $N \in (Q^{n/k}, Q^n |\mS|^{1/k})$. When $k \asymp_{n,[F:\Q]} \eps^{-1}$, this yields a nontrivial large sieve inequality in the full range $N > Q^\eps$, showing that the duality approach to large sieves can also be useful for short sums.
\end{remark}

\begin{remark} 
By slightly modifying the argument in \cref{sec:holder}, \cref{cor:large_sieve2} still holds if, in $C^\infty(N,\mathcal{S})$, one replaces $\lambda_\pi(\kn)$ with the $\kn$-th Dirichlet coefficient of $L(s, \pi)^{-1}$ (and similarly for $C^\infty(N,\mathcal{S}, \pi_0)$).
\end{remark}

\subsection{Zero density estimates}

\cref{thm:large_sieve2} is a large sieve inequality for the Dirichlet coefficients of $L$-functions in families $\{L(s,\pi)\colon\pi\in\mathcal{S}\}$ and $\{L(s,\pi\times\pi_0)\colon\pi\in\mathcal{S}\}$.  In certain problems, it is important to consider instead the Dirichlet coefficients of $1/L$, or $\log L$.  The aforementioned technical extension of \cref{thm:large_sieve2}, given in \cref{thm:large-sieve-general}, establishes large sieve inequalities of the same strength as \cref{thm:large_sieve2} for these coefficients as well.  This quickly leads to zero density estimates for families of $L$-functions. For $T\geq 0$ and $\sigma\geq 0$, we define\footnote{In this paper, all zeros are counted with multiplicity.}
\[
N_{\pi}(\sigma,T):=\sum_{\substack{ \rho=\beta+i\gamma\\ \beta\geq \sigma,~|\gamma|\leq T \\ L(\rho,\pi)=0}}1,\qquad\quad N_{\pi\times\pi_0}(\sigma,T):=\sum_{\substack{ \rho=\beta+i\gamma \\ \beta\geq \sigma,~|\gamma|\leq T \\ L(\rho,\pi\times\pi_0)=0}}1.
\]
The generalized Riemann hypothesis (GRH) for $L(s,\pi)$ (resp.\ $L(s,\pi\times\pi_0)$) is equivalent to the statement that if $\sigma>\frac{1}{2}$, then $N_{\pi}(\sigma,T)=0$ (resp.\ $N_{\pi\times\pi_0}(\sigma,T)=0$).  In the absence of strong zero-free regions for $L$-functions, strong bounds for $N_{\pi}(\sigma,T)$ and $N_{\pi\times\pi_0}(\sigma,T)$, for individual $L$-functions or in families, can oftentimes suffice for arithmetic applications.  Perhaps the first notable example of this is Hoheisel's proof \cite{Hoheisel} that the prime number theorem holds in intervals of length $x^{1-1/33000}$.

The method of Humphries and Thorner \cite{humphries2024zeros} combined with a small refinement of \eqref{eqn:TZ_large_sieve} yields
\begin{equation}
\label{eqn:TZ_zde}
\sum_{\pi\in\mathcal{S}}N_{\pi}(\sigma,T)\ll_{n,[F:\Q]}(Q^{8\theta n^2+2n+1}T^{[F:\Q]n(n+1)+4}|\mathcal{S}|^4)^{\frac{1-\sigma}{3-2\sigma}+o(1)}.
\end{equation}
Humphries and Thorner \cite[Theorem 1.1]{humphries2024zeros} also proved that
\begin{equation}
\label{eqn:HT_zde}
\sum_{\pi\in\mathcal{S}}N_{\pi\times\pi_0}(\sigma,T)\ll_{n,n_0,[F:\Q]}(|\mathcal{S}|^4 (\mathfrak{C}_{\pi_0}QT^{[F:\Q]})^{6.15\max\{n^2,n_0 n\}})^{1-\sigma+o(1)}.
\end{equation}
Using \cref{thm:large-sieve-general}, we obtain the following substantial improvements over \eqref{eqn:TZ_zde} and \eqref{eqn:HT_zde}.
\begin{theorem}
\label{thm:ZDE}
If $\mathcal{S}\subseteq\mathfrak{F}_n$ is finite, $\pi_0\in\mathfrak{F}_{n_0}$, $Q=\max_{\pi\in\mathcal{S}}\mathfrak{C}_{\pi}$, and $T\geq 1$, then
\begin{align*}
\sum_{\pi\in\mathcal{S}}N_{\pi}(\sigma,T)&\ll_{n,[F:\Q]}(D_F^{-n(n+1)}Q^{2n+1}T^{[F:\Q]n(n+1)+4}|\mathcal{S}|^4)^{\frac{1-\sigma}{3-2\sigma}+o(1)},\\
\sum_{\pi\in\mathcal{S}}N_{\pi\times\pi_0}(\sigma,T)&\ll_{n,n_0,[F:\Q]}(D_F^{-n(n+n_0)}\mathfrak{C}_{\pi_0}^{n}Q^{2n+n_0}T^{[F:\Q]n(n+n_0)+4}|\mathcal{S}|^4)^{\frac{1-\sigma}{3-2\sigma}+o(1)}.
\end{align*}
\end{theorem}

In light of \eqref{eqn:BM}, we separately record the following immediate corollary of \cref{thm:ZDE}.

\begin{corollary}
\label{cor:full_family}
If $n,n_0,Q,T\geq 1$ and $\pi_0\in\mathfrak{F}_{n_0}$, then
\begin{align*}
\sum_{\pi\in\mathfrak{F}_n(Q)}N_{\pi}(\sigma,T)&\ll_{n,F}(|\mathfrak{F}_n(Q)|^{6-\frac{1}{n+1}}T^{[F:\Q]n(n+1)+4})^{\frac{1-\sigma}{3-2\sigma}+o(1)},\\
\sum_{\pi\in\mathfrak{F}_n(Q)}N_{\pi\times\pi_0}(\sigma,T)&\ll_{n,n_0,F}(\mathfrak{C}_{\pi_0}^{n}|\mathfrak{F}_n(Q)|^{6+\frac{n_0-2}{n+1}}T^{[F:\Q]n(n+n_0)+4})^{\frac{1-\sigma}{3-2\sigma}+o(1)}.
\end{align*}
\end{corollary}

\subsection{Log-free zero density estimates}
If each representation in  $\mathcal{S}\cup\{\pi_0\}$ satisfies GRC, then the $(QT|\mathcal{S}|)^{o(1)}$ and $(\mathfrak{C}_{\pi_0}QT|\mathcal{S}|)^{o(1)}$ terms in \cref{thm:ZDE} can be replaced by a power of $\log(QT|\mathcal{S}|)$ and $\log(\mathfrak{C}_{\pi_0}QT|\mathcal{S}|)$, respectively.  When the $o(1)$ terms can be eliminated altogether, the zero density estimates are called {\it log-free}. Log-free zero density estimates were first proved for Dirichlet $L$-functions by Linnik \cite{Linnik} as a crucial part of his proof that there exists an absolute constant $B>0$ such that if $\gcd(a,q)=1$, then the smallest prime $p\equiv a\pmod{q}$ is at most $q^B$.  (GRH for Dirichlet $L$-functions implies that any fixed $B>2$ is permissible.)  For discussion and examples of the importance of log-free zero density estimates, see \cite{BrumleyThornerZaman,HT_ZFR,IK,KanekoThorner,KM,Motohashi,SiuSound,soundararajan2019weak,TZ3,TZ_Sarkozy,Weiss}.

Conditional on certain progress towards GRC (which is unproved for $\pi\in\mathfrak{F}_n$ when $n\geq 3$ aside from certain special cases), the first log-free zero density estimate for general families was proved by Kowalski and Michel \cite{KM}.  Unconditionally, Thorner and Zaman \cite[Theorem 1.2]{TZ_GLn} proved that
\begin{equation}
\label{eqn:TZ_LFZDE}
\sum_{\pi\in\mathfrak{F}_n(Q)}N_{\pi}(\sigma,T)\ll_{n,[F:\Q]}(QT^{[F:\Q]})^{10^7 n^4 (1-\sigma)}.
\end{equation}
When $F=\Q$,  Brumley, Thorner, and Zaman \cite[Theorem 1.3]{BrumleyThornerZaman} proved that under a certain hypothesis for each $\pi\in\mathcal{S}\cup\{\pi_0\}$ (see \eqref{eqn:hyp_BTZ} below), one has the bound
\begin{equation}
\label{eqn:BTZ_LFZDE}
\sum_{\pi\in\mathcal{S}}N_{\pi\times\pi_0}(\sigma,T)\ll_{n,n_0}(\mathfrak{C}_{\pi_0}QT)^{10^8(n_0 n)^4(1-\sigma)}.
\end{equation}

Using \cref{thm:large-sieve-general} and a refinement to the process of zero detection in \cite{BrumleyThornerZaman,TZ_GLn}, we theoretically and numerically improve the exponent in \eqref{eqn:TZ_LFZDE}.  Such improvements also extend to \eqref{eqn:BTZ_LFZDE}, but more importantly, we remove all unproven hypotheses that the method of proof in \cite{BrumleyThornerZaman} required.
\begin{theorem}
\label{thm:LFZDE}
Let $F$ be a number field and $n,n_0\geq 1$.  Let $\mathcal{S}\subseteq\mathfrak{F}_n$ be finite and $Q=\max_{\pi\in\mathcal{S}}\mathfrak{C}_{\pi}$.  For all $\epsilon>0$, there exists an effectively computable constant $\Cl[abcon]{LFZDE_fam}=\Cr{LFZDE_fam}(n)>0$ such that if $Q\geq D_F^{\Cr{LFZDE_fam}}$ and $\pi_0\in\mathfrak{F}_{n_0}$, then
\begin{align*}
\sum_{\pi\in\mathcal{S}}N_{\pi}(\sigma,T)&\ll_{n,[F:\Q]}(QT^{[F:\Q]})^{543 n^3(1-\sigma)},\\
\sum_{\pi\in\mathcal{S}}N_{\pi\times\pi_0}(\sigma,T)&\ll_{n,n_0,[F:\Q]}(\max\{\mathfrak{C}_{\pi_0},Q\}T^{[F:\Q]})^{543\max\{n^3,n_0^3\}(1-\sigma)}.
\end{align*}
\end{theorem}

\subsection{Average bounds for $L(\frac{1}{2},\pi)$ and $L(\frac{1}{2},\pi\times\pi_0)$}

Let $\pi\in\mathfrak{F}_n$ and $\pi_0\in\mathfrak{F}_{n_0}$.  The generalized Lindel{\"o}f hypothesis (GLH) asserts that
\begin{equation}
\label{eqn:Lindelof}
|L(\tfrac{1}{2},\pi)|\ll_{n,[F:\Q]}\mathfrak{C}_{\pi}^{o(1)},\qquad |L(\tfrac{1}{2},\pi\times\pi_0)|\ll_{n,[F:\Q]}\mathfrak{C}_{\pi\times\pi_0}^{o(1)},
\end{equation}
while it follows from the Phragm{\'e}n--Lindel{\"o}f principle and Stirling's formula that
\begin{equation}
\label{eqn:convexity}
|L(\tfrac{1}{2},\pi)|\ll_{n,[F:\Q]}\mathfrak{C}_{\pi}^{\frac{1}{4}+o(1)},\qquad 
|L(\tfrac{1}{2},\pi\times\pi_0)|\ll_{n,n_0,[F:\Q]}\mathfrak{C}_{\pi\times\pi_0}^{\frac{1}{4}+o(1)}.
\end{equation}
Let $\mathcal{S}\subseteq\mathfrak{F}_n$ be a finite subset, and let $Q=\max_{\pi\in\mathcal{S}}\mathfrak{C}_{\pi}$. When combined with the approximate functional equation, a large sieve inequality of the form
\begin{equation}
\label{eqn:AFE_large_sieve}
C(Q^{\frac{1}{2}+o(1)},\mathcal{S})\ll_{n,F} |\mathcal{S}| Q^{o(1)}
\end{equation}
would imply that \eqref{eqn:Lindelof} holds on average over $\pi \in \mS$:
\begin{equation}
\label{eqn:ideal_2nd_moment}
\sum_{\pi\in\mathcal{S}}|L(\tfrac{1}{2},\pi)|^2\ll_{n,F} |\mS| Q^{o(1)},\qquad Q=\max_{\pi\in\mathcal{S}}\mathfrak{C}_{\pi}.
\end{equation}

When $n=1$ and $F=\Q$, \eqref{eqn:ideal_2nd_moment} follows from \eqref{eqn:large-sieve-dirichlet} (see Heath-Brown \cite{HB_real} for a similar result when the $\chi$ are restricted to be primitive quadratic).   There are many examples of $n\geq 2$ and $\mathcal{S}\subseteq\mathfrak{F}_n$ such that \eqref{eqn:large_sieve_optimal} (hence \eqref{eqn:ideal_2nd_moment}) holds; for example, see \cite{blomer2023density,deshouillers1982kloosterman,jana2020applications}.  Typically, in order to obtain \eqref{eqn:AFE_large_sieve}, a version of the Kuznetsov trace formula must be available. There are several interesting families $\mathcal{S}$ for which this amount of structure is not known, and it is not feasible to embed $\mathcal{S}$ into a larger, more structured family.  One important such family is $\mathfrak{F}_n(Q)$.

By combining \cref{cor:large_sieve2} with the approximate functional equation, we obtain the strongest known bound for the second moment of $L(\frac{1}{2},\pi)$ as $\pi$ varies in an arbitrary family $\mathcal{S}\subseteq\mathfrak{F}_n$.  In what follows, if $x\in\R$, then $\lceil x\rceil := \min\{m\in\Z\colon m\geq x\}$.

\begin{theorem}
\label{thm:2ndMoment}
Let $\mathcal{S}\subseteq\mathfrak{F}_n$ be a finite subset and $Q=\max_{\pi\in\mathcal{S}}\mathfrak{C}_{\pi}$.  If $\delta=\delta_{\mathcal{S}}\geq 0$ satisfies $|\mathcal{S}|\gg_{n,F} Q^{\delta}$, then
\[
\sum_{\pi\in\mathcal{S}}|L(\tfrac{1}{2},\pi)|^2\ll_{n,F} |\mathcal{S}|Q^{\frac{1}{2}-\frac{\delta}{\lceil 2(n+\delta)\rceil}+o(1)}.
\]
\end{theorem}

\begin{remark}
\cref{thm:2ndMoment} improves upon the convexity bound \cref{eqn:convexity} on average as soon as $\delta >0$.
For the full family $\mathcal{S}=\mathfrak{F}_n(Q)$, we may take $\delta = n+1$ per \eqref{eqn:BM}, in which case
\begin{equation}
\label{eqn:universal_moment}
\sum_{\pi\in\mathfrak{F}_n(Q)}|L(\tfrac{1}{2},\pi)|^2\ll_{n,F} |\mathfrak{F}_n(Q)| Q^{\frac{1}{4}-\frac{1}{8n+4}+o(1)}.
\end{equation}
This is new for all $n\geq 2$ and all $F$. Using a dyadic argument, we infer from Chebyshev's inequality, \eqref{eqn:BM_asymptotic}, and \eqref{eqn:universal_moment} that 
\begin{equation}
\label{eqn:Chebyshev_moment}
\Big|\Big\{\pi\in\mathfrak{F}_n(Q)\colon |L(\tfrac{1}{2},\pi)|\geq \mathfrak{C}_{\pi}^{\frac{1}{8}-\frac{1}{16n+8}+\epsilon}\Big\}\Big|\ll_{n,F,\epsilon} |\mathfrak{F}_n(Q)| Q^{-\epsilon}.
\end{equation}
In contrast, Nelson \cite[Theorem 1.2]{Nelson} proved that if $\mathfrak{C}_{\pi}$ does not exhibit conductor dropping, then $|L(\frac{1}{2},\pi)|\ll_{n,F,\N\kq_{\pi}}\mathfrak{C}_{\pi}^{1/4-\varpi_n}$, where $\varpi_n = \frac{1}{6n^5}+O(\frac{1}{n^6})$.  Yang \cite{Yang} showed that one may take $\varpi_n\gg\frac{1}{n^3}$.
\end{remark}

\begin{remark}
Like our other results, \cref{thm:2ndMoment} can improve in the presence of conductor dropping. A careful inspection of our proof shows that if
\[
Q = \max_{\pi\in\mathcal{S}}\mathfrak{C}_{\pi},\qquad R = \max_{\pi,\pi'\in\mathcal{S}}\mathfrak{C}_{\pi\times\tilde{\pi}'},\qquad |\mathcal{S}|\gg_{n,F}Q^{\delta},\qquad k=\Big\lceil\frac{\log(R|\mathcal{S}|^2)}{\log Q}\Big\rceil,
\]
then one can slightly refine \cref{thm:2ndMoment} to $\sum_{\pi\in\mathcal{S}}|L(\tfrac{1}{2},\pi)|^2\ll_{n,F} |\mathcal{S}| Q^{1/2 - \delta/k+o(1)}$.
\end{remark}

\begin{remark}
We prove a more general result, with $L(\frac{1}{2},\pi)$ replaced by $L(\frac{1}{2}+it,\pi)$ for any $t\in\mathbb{R}$.  Also, using the same ingredients, one can deduce bounds for higher moments $\sum_{\pi \in \mS} |L(\tfrac{1}{2}, \pi)|^{2r}$.
\end{remark}

Complementing \eqref{eqn:Chebyshev_moment}, one might want a small power-saving improvement over \eqref{eqn:convexity} for all $\pi\in\mathfrak{F}_n(Q)$ outside of a very small exceptional set (say, of size $|\mathfrak{F}_n(Q)|^{\epsilon}$).  As in \cite{humphries2024zeros,TZ_GLn}, such a result follows from \cref{thm:ZDE,cor:full_family}; see \cite[Section 2.1]{humphries2024zeros} for a discussion. By inserting \cref{cor:full_family} into the arguments in \cite[Section 6]{humphries2024zeros} instead of \eqref{eqn:HT_zde}, we immediately arrive at the following corollary. Unlike \cite{humphries2024zeros,TZ_GLn}, if $n_0\ll n$, the power saving over the convexity bound \eqref{eqn:convexity} does not decay to $0$ as $n$ grows.

\begin{corollary}
\label{cor:subconvex}
If $n,n_0,Q\geq 1$, $0<\epsilon\leq 1$, and $\pi_0\in\mathfrak{F}_{n_0}$, then
\begin{align*}
\begin{gathered}
\Big|\Big\{\pi\in\mathfrak{F}_n(Q)\colon |L(\tfrac{1}{2},\pi)|\geq \mathfrak{C}_{\pi}^{\frac{1}{4}-\frac{\epsilon}{10^{12}}}\Big\}\Big|\ll_{n,F,\epsilon}|\mathfrak{F}_n(Q)|^{\epsilon},\label{eqn:subconvex}\\
\Big|\Big\{\pi\in\mathfrak{F}_n(Q)\colon |L(\tfrac{1}{2},\pi\times\pi_0)|\geq \mathfrak{C}_{\pi\times\pi_0}^{\frac{1}{4}-\frac{\epsilon}{(1+\epsilon)10^9}\frac{n}{6n+n_0+4}}\Big\}\Big|\ll_{n,n_0,F,\epsilon}(\mathfrak{C}_{\pi_0}^n|\mathfrak{F}_n(Q)|)^{\epsilon}.
\end{gathered}
\end{align*}
\end{corollary}

\subsection{Density estimate for Langlands parameters and the size of $\mathfrak{F}_n(Q)$}

As mentioned above, our approach builds on the work of Lichtman and Pascadi \cite{lichtman2024density}.  Their main arithmetic result \cite[Theorem 1]{lichtman2024density}, restricted to the non-archimedean setting, is a density estimate for Satake parameters of $\pi\in\mathcal{S}$.  To describe their result, let $F=\Q$, $p$ be prime, and $\theta>0$.  Given $\pi\in\mathfrak{F}_n$, write $q_{\pi}$ for the arithmetic conductor and $\{\alpha_{j,\pi}(p)\colon 1\leq j\leq n\}$ for the set of non-archimedean Langlands parameters of $\pi$ at $p$.  When $p\nmid q_{\pi}$, these are the Satake parameters at $p$.  For a finite family $\mathcal{S}\subseteq \mathfrak{F}_n$, let $Q=\max_{\pi\in\mathcal{S}}\mathfrak{C}_{\pi}$.  Theorem 1 of \cite{lichtman2024density} states that if each $\pi\in\mathcal{S}$ satisfies $p\nmid q_{\pi}$, then
\begin{equation}
\label{eqn:LichtmanPascadi}
\#\{\pi\in\mathcal{S}\colon \max_{1\leq j\leq n}|\alpha_{j,\pi}(p)|\geq p^{\theta}\}\ll_{n,p,\theta}(Q^{2n})^{\frac{1-2\theta}{4\theta}+o(1)}.
\end{equation}
GRC would imply that if $\theta>0$, then the left-hand side vanishes.  In contrast with other density estimates for non-archimedean Langlands parameters (e.g., \cite{blomer2023density,jana2020applications}), the exponent on the right-hand side of \eqref{eqn:LichtmanPascadi} is $o_{n \to \infty}(1)$ near $\theta = \frac{1}{2}-\frac{1}{n^2+1}$, at which point the work of Luo, Rudnick, and Sarnak \cite{LRS} ensures that the left-hand side of \eqref{eqn:LichtmanPascadi} equals zero. By combining our main technical result, \cref{thm:large-sieve-general}, with the ideas in \cite{lichtman2024density}, we achieve several refinements.
\begin{theorem}
\label{thm:density}
Let $F$ be a number field, $\kp$ a prime ideal of $\mathcal{O}_F$, $\mathcal{S}\subseteq \mathfrak{F}_n$ a finite subset, and $Q=\max_{\pi\in\mathcal{S}}\mathfrak{C}_{\pi}$.  For all $\epsilon>0$, there exists an effectively computable constant $\Cl[abcon]{density_range}=\Cr{density_range}(n,\epsilon)\geq 1$ such that if $Q\geq D_F^{\Cr{density_range}}$ and $\theta\geq 0$, then
\begin{equation}
\label{eqn:density}
\#\{\pi\in\mathcal{S}\colon \max_{1\leq j\leq n}|\alpha_{j,\pi}(\kp)|\geq\N\kp^{\theta}\}\ll_{n,[F:\Q]}\N\kp^{n}(D_F^{-n^2}Q^{2n})^{\frac{1-2\theta}{\max\{1,4\theta\}}+\epsilon}.
\end{equation}
\end{theorem}
\begin{remark}
An inspection of the proof will show that one can replace $D_F^{-n^2}Q^{2n}$ with $\max_{\pi,\pi'\in\mathcal{S}}\mathfrak{C}_{\pi\times\tilde{\pi}'}$.
\end{remark}

Compared to \cite[Theorem 1]{lichtman2024density}, \cref{thm:density} removes the requirement that $F=\Q$, explicates the dependence on $\kp$ on the right-hand side, eliminates the need for each $\pi\in\mathcal{S}$ to be unramified at $\kp$, and improves the exponent on the right-hand side of \eqref{eqn:LichtmanPascadi} when $0\leq\theta<\frac{1}{4}$.  Surprisingly, several of these refinements come together to produce a nontrivial upper bound on $|\mathfrak{F}_n(Q)|$.

\begin{corollary}
\label{cor:family}
Let $F$ be a number field.  For all $\epsilon>0$, there exists an effectively computable constant $\Cl[abcon]{family}=\Cr{family}(n,\epsilon)>0$ such that if $Q\geq D_F^{\Cr{family}}$, then $|\mathfrak{F}_n(Q)|\ll_{n,[F:\Q],\epsilon}D_F^{-n^2}Q^{2n+\epsilon}$.
\end{corollary}
\begin{remark}
It follows as an immediate corollary that if $Q\geq 1$, then $|\mathfrak{F}_n(Q)|\ll_{n,F,\epsilon}Q^{2n+\epsilon}$.
\end{remark}

The bound on $|\mathfrak{F}_n(Q)|$ in \cref{cor:family} was first proved by Brumley, Thorner, and Zaman \cite[Theorem A.1]{BrumleyThornerZaman}.  Both proofs rely crucially on the detailed description of the Euler factors of Rankin--Selberg $L$-functions at ramified prime ideals in \cite[(A.8)]{soundararajan2019weak}, but the two proofs use this very differently.  In \cite{BrumleyThornerZaman}, $\mathfrak{F}_n(Q)$ is decomposed into several disjoint subfamilies based on conductors, central characters, and data that can be read off from the Bernstein--Zelevinsky description of the admissible dual.  The cardinalities of these subfamilies are estimated using Rankin--Selberg theory (including \cite[(A.8)]{soundararajan2019weak}) and sphere packing bounds in high dimensions.  In order to efficiently sum these cardinalities, one invokes a new bound on Rankin--Selberg conductors \cite[Theorem B.1]{BrumleyThornerZaman} due to Bushnell and Henniart.

Contrastingly, in this paper, the description of the ramified Euler factors goes into the proof of \cref{thm:large-sieve-general} (hence \cref{thm:density}).  While our unitary normalization of each $\pi\in\mathfrak{F}_n(Q)$ ensures that $\max_{1\leq j\leq n}|\alpha_{j,\pi}(\kp)|\geq 1$ when $\kp\nmid\kq_{\pi}$, classical analytic number theory shows that if $\pi\in\mathfrak{F}_n(Q)$, then the conductor $\kq_{\pi}$ must be indivisible by a prime ideal of norm $O_{n,F,\epsilon}(Q^{\epsilon})$. The bound in \cref{thm:density} with $\theta=0$ is now used for all $\kp$ with $\N\kp\ll_{n,F,\epsilon} Q^{\epsilon}$.  In particular, our proof does not rely on the aforementioned bound \cite[Theorem B.1]{BrumleyThornerZaman} for Rankin--Selberg conductors.

\begin{remark}
As in \cite{lichtman2024density}, it should be possible to obtain an analogue of \cref{thm:density} for archimedean places $v$, ramified or not, provided that one replaces the analytic conductor $\mathfrak{C}_{\pi}$ with the slightly larger \emph{total conductor} $\mathfrak{C}_{\pi}^{\mathrm{total}}$, as defined in \cite[(3.10)]{lichtman2024density}.  Let $C :=\max_{\pi\in\mathcal{S}}\mathfrak{C}_{\pi}^{\mathrm{total}}$.  Write the set of archimedean Langlands parameters of $\pi$ at $v$ as $\{\mu_{j,\pi}(v)\colon 1\leq j\leq n\}$, and define $\theta_\pi(v) := \max_{1\le j \le n} |\Re(\mu_{j,\pi}(v))|$. When $\theta \ge \tfrac{1}{4}$, an upper bound similar to \eqref{eqn:density} for $\#\{\pi \in \mS : \theta_{\pi}(v) \ge \theta\}$ would follow from the following modification of our large sieve inequality from \cref{thm:large_sieve2}:
\[
\sup_{\|a\|_2 = 1} \sum_{\pi \in \mS} X^{2\theta_{\pi}(v)} \Bigg| \sum_{\N\n \le N} \lambda_\pi(\kn) a(\kn) \Bigg|^2 \ll_{n,[F:\Q]} (NC)^{o(1)} (NX + C^n |\mS|),\qquad X\geq 1.
\]
Relative to \cref{thm:large_sieve2}, the incorporation of $X^{2\theta_{\pi}(v)}$ into the left-hand side is of no cost to the right-hand side when $X \le C^n|\mS| N^{-1}$; this is similar to \cite[Theorems 5 and 6]{deshouillers1982kloosterman}. Such a large sieve inequality should follow from the ideas in \cref{sec:non-def-covers,sec:large-sieve}, using the archimedean case of the local Langlands correspondence to describe the Rankin--Selberg factors at ramified archimedean places.
\end{remark}

\subsection*{Analytic notation}
\label{subsec:notation}
We use the standard notation $f \asymp_\eps g$, $f \ll_\eps g$, $f = O_\eps(g)$, $f = o(g)$ from analytic number theory, where the subscripts indicate that the implicit constants may depend on the parameter $\eps$.  In particular, the statement $f(x) \ll_{\nu} x^{o(1)} g(x)$ is equivalent to the statement that for all $\epsilon>0$, there exists an effectively computable constant $c(\nu,\epsilon)>0$ such that $|f(x)| \leq c(\nu,\epsilon) x^\eps |g(x)|$ in a range of $x$ that will be clear from context.  Throughout, we view $n$, $n_0$, and $[F:\Q]$ as being fixed, so we will allow implied constants to depend on them. We write $(\m, \n)$ and $[\m, \n]$ for the GCD and LCM of two integral ideals.  We write $\mathbb{N}$ for the set of positive integers.

\subsection*{Structure of paper} In \cref{sec:InformalOutline}, we give an informal overview of the main ideas of the paper.  In \cref{sec:L-functions}, we record some key properties of $L$-functions.  In \cref{sec:non-def-covers}, we introduce the notion of positive semi-definite covers for families of $L$-functions. This is  crucial for the statement of our main technical result, \cref{thm:large-sieve-general}, which we prove in \cref{sec:large-sieve}. \cref{thm:large_sieve2} follows quickly from \cref{thm:large-sieve-general}.  In \cref{sec:family_proofs}, we use \cref{thm:large-sieve-general} to prove \cref{thm:density,cor:family}. We prove \cref{thm:ZDE} in \cref{sec:zero-density}, \cref{cor:large_sieve2,thm:2ndMoment} in \cref{sec:holder}, and \cref{thm:LFZDE} in \cref{sec:LFZDE}.

\subsection*{Acknowledgements}
We thank Farrell Brumley, Yujiao Jiang, and James Maynard for helpful discussions. AP is supported by an EPSRC Scholarship. JT is partially supported by the Simons Foundation (MP-TSM-00002484) and the National Science Foundation (DMS-2401311).

\section{Informal outline}
\label{sec:InformalOutline}

Here we give an informal overview of our arguments, ignoring various technical details. Let $n, n_0 \in \mathbb{N}$, $N, Q \ge 1$, $\mS \subseteq \F_n$ be a finite set with $\max_{\pi \in \mS} \mathfrak{C}_{\pi} \le Q$, and $\pi_0 \in \F_{n_0}$.

\subsection{Previous work on $\GL_n$ large sieves}
By the duality principle (see \cite[p. 170]{IK}), one has
\begin{equation}
\label{eqn:dual-large-sieve}
    C(N, \mS) = \sup_{\|w\|_2 = 1}\sum_{\N \n \le N} \Bigg| \sum_{\pi \in \mS} w(\pi) \lambda_\pi(\n) \Bigg|^2,
\end{equation}
where $\|w\|_2 := ( \sum_{\pi \in \mS} |w(\pi)|^2 )^{1/2}$.  Duke and Kowalski \cite[Section 4]{DK} expanded the square in \eqref{eqn:dual-large-sieve}, swapped the order of summation, and applied Mellin inversion, thus obtaining
\begin{equation}
\label{eqn:approximation}
\begin{aligned}
\sum_{\N \n \le N} \Bigg| \sum_{\pi \in \mS} w(\pi) \lambda_\pi(\n) \Bigg|^2
=
    \sum_{\pi, \pi' \in \mS} w(\pi) \overline{w(\pi')} \frac{1}{2\pi i}\int_{3-i\infty}^{3+i\infty}\Bigg(\sum_{\kn}\frac{\lambda_{\pi}(\kn)\overline{\lambda_{\pi'}(\kn)}}{\N\kn^s}\Bigg)\frac{N^s}{s}ds.
\end{aligned}
\end{equation}
It is implicit in their work that if $\theta$ is the best bound towards GRC for each element of $\mathcal{S}$, then there exists a function $H(s,\pi,\pi')$, holomorphic in the region $\re(s)>\tfrac{1}{2}+2\theta$, such that
\[
\sum_{\kn}\frac{\lambda_{\pi}(\kn)\overline{\lambda_{\pi'}(\kn)}}{\N\kn^s} = H(s,\pi,\pi')L(s,\pi\times\widetilde{\pi}').
\]
The basis for the results in \cite{DK} is that this factorization holds inside of the critical strip when $\theta<\tfrac{1}{4}$ (and in particular, when GRC holds).  Now, one pushes the contour as far to the left as progress towards GRC permits, bounding $L(s,\pi\times\widetilde{\pi}')$ in the critical strip using convexity.
It follows from recent work of Jiang \cite{Jiang} that $H(s, \pi, \pi')$ is actually holomorphic in the region $\Re(s) > \tfrac{1}{2}+\theta$, and by inserting this into the approach from \cite{DK}, one arrives at \eqref{eqn:Jiang_large_sieve}.

On the other hand, it follows from the work of Thorner and Zaman \cite[Section 4.2]{TZ_GLn}, which relies on the orthogonality of Schur polynomials evaluated at the sets of local roots $\{\alpha_{j,\pi}(\kp)\}$ and $\{\alpha_{j',\pi'}(\kp)\}$, that for any given $\kn$, the bound
\begin{equation}
\label{eqn:approximation2}
    \Bigg| \sum_{\substack{\pi \in \mS \\ \gcd(\kq_\pi, \kn) = \mO_F}} w(\pi) \lambda_\pi(\n) \Bigg|^2 
   \leq 
    \sum_{\substack{\pi, \pi' \in \mS \\ \gcd(\kq_\pi \kq_{\pi'}, \kn) = \mO_F}} w(\pi) \bar w(\pi) \lambda_{\pi\times\widetilde{\pi}'}(\kn)
\end{equation}
holds without recourse to unproven progress towards GRC.  Summing over all $\kn$ with $\N\kn\leq N$ and applying Mellin inversion, one arrives at
\begin{equation}
\label{eqn:approximation3}
    \sum_{\N \n \le N} \Bigg| \sum_{\substack{\pi \in \mS \\ \gcd(\kq_{\pi}, \kn)=\mathcal{O}_F}} w(\pi) \lambda_\pi(\n) \Bigg|^2 
   \leq 
    \sum_{\pi, \pi' \in \mS} w(\pi) \bar w(\pi) \frac{1}{2\pi i}\int_{3-i\infty}^{3+i\infty}\frac{L(s,\pi\times\widetilde{\pi}')}{\prod_{\kp\mid\kq_{\pi}\kq_{\pi'}}L(s,\pi_{\kp}\times\widetilde{\pi}_{\kp}')}\frac{N^s}{s}ds.\hspace{-1mm}
\end{equation}
One then pushes the contour to the line $\re(s)=(\log N)^{-1}$, bounding $L(s,\pi\times\widetilde{\pi}')$ using \cref{lem:Li1} and the product over $\kp\nmid\kq_{\pi}\kq_{\pi'}$ using the existing progress towards GRC in \eqref{eqn:ramanujan_progress}.  It is this latter contribution that leads to the factor of $Q^{4\theta n^2}$ in \eqref{eqn:TZ_large_sieve}. Building on these arguments, Humphries and Thorner \cite[Theorem 4.1]{humphries2024zeros} proved large sieve inequalities with the same dependency on $\theta$ for the Dirichlet coefficients of $L(s, \pi \times \pi_0)$ and $L^{-1}(s, \pi \times \pi_0)$, the latter being relevant for zero detection. These are the key results behind the zero density estimates in \cref{eqn:TZ_zde,eqn:HT_zde}.

To prove a \emph{log-free} zero density estimate for $\{L(s,\pi)\colon \pi\in\mathcal{S}\}$, Thorner and Zaman \cite{TZ_GLn} inserted Selberg sieve weights into the large sieve (see \cite[Theorem 4.2]{TZ_GLn}). They used this to handle the contribution of unramified prime ideals $\p$, and trivially estimated the contribution from $\{\kp^k\colon \textup{$\kp$ ramified or $k\geq 2$}\}$ using the best progress towards GRC \cite{MS,luo1999generalized}. However, this was insufficient to obtain a log-free zero density estimate for the family $\{L(s,\pi\times\pi_0)\colon \pi\in\mathcal{S}\}$ when $\pi_0$ does not satisfy GRC. The need for GRC would be removed if each $\pi\in\mathcal{S}\cup\{\pi_0\}$ were to satisfy
\begin{equation}
\label{eqn:hyp_BTZ}
\prod_{\kp}\sum_{r=0}^{\infty}\frac{\max_j |\alpha_{j,\pi}(\kp)|^{2r}}{\N\kp^{r(1+\epsilon)}}\ll_{n,[F:\Q],\epsilon}\mathfrak{C}_{\pi}^{\epsilon},
\end{equation}
where $\alpha_{j,\pi}(\kp)$ ranges over the Satake parameters at $\kp$.  This is Hypothesis 1.1 in the work of Brumley, Thorner, and Zaman \cite{BrumleyThornerZaman}, and it is (as of now) only known when $\pi\in\mathfrak{F}_n$ with $n\leq 4$ \cite{Brumley_2}.

\subsection{Our innovations}
We now describe our key new ideas.

\subsubsection{A linear-algebraic approach to handle ramification}
One of our key observations is that the inequality from \eqref{eqn:approximation2} actually holds with no ramification constraints:
\begin{equation}
\label{eqn:key-inequality}
    \Bigg| \sum_{\pi \in \mS} w(\pi) \lambda_\pi(\kn) \Bigg|^2 
    =
    \sum_{\pi, \pi' \in \mS} w(\pi) \bar w(\pi') \lambda_\pi(\kn) \bar\lambda_{\pi}(\kn)
    \le 
    \sum_{\pi, \pi' \in \mS} w(\pi) \bar w(\pi') \lambda_{\pi \times \tilde \pi'}(\kn).
\end{equation}
This allows us to seamlessly pass from the products $\lambda_\pi(\kn) \bar{\lambda_{\pi'}(\kn)}$ to the actual Rankin--Selberg coefficients $\lambda_{\pi\times\tilde{\pi}'}(\kn)$, with no mention of the function $H(s, \pi, \pi')$ of Duke and Kowalski \cite{DK}.

To prove \eqref{eqn:key-inequality} in general, one needs a more careful treatment of the coefficients $\lambda_{\pi\times\tilde\pi'}(\kn)$ when $(\kn,\kq_{\pi}\kq_{\pi'})\neq\mathcal{O}_F$. While an explicit combinatorial argument as in \cite{TZ_GLn} becomes rather cumbersome and difficult to see, a linear-algebraic formulation inspired by \cite{lichtman2024density} will lead to a fairly short proof. In fact, we can reformulate \eqref{eqn:key-inequality} as follows:
\begin{equation} \label{eqn:posdef-claim}
    \lambda_{\pi \times \tilde\pi'}(\kn) - \lambda_\pi(\kn) \bar \lambda_{\pi'}(\kn)
    \qquad \text{ form a \emph{positive semi-definite matrix} in } \mathbb{C}^{\mS \times \mS}.
\end{equation}
Lichtman and Pascadi  \cite[Proposition 4]{lichtman2024density} have previously shown that $\lambda_{\pi \times \tilde\pi'}(\n)$ form a positive semi-definite matrix in $\mathbb{C}^{\mS \times \mS}$ (in the case $F = \Q$), which only means that the right-hand side of \cref{eqn:key-inequality} is nonnegative. Building on their work, we prove a general bound for bilinear sums over $\pi, \pi_0$ of $\lambda_{\pi \times \pi_0}(\kn)$; see \cref{prop:inequality-RS-covers}. When $\pi_0$ is fixed, this reads
\begin{equation} \label{eqn:inequality-RS}
    \Bigg| \sum_{\pi \in \mS} w(\pi) \lambda_{\pi \times \pi_0}(\n) \Bigg|^2
    \le 
    \lambda_{\pi_0 \times \tilde\pi_0}(\n) \sum_{\pi, \pi' \in \mS} w(\pi) \bar w(\pi') \lambda_{\pi \times \tilde\pi'}(\n).
\end{equation}
Taking $\pi_0\in\mathfrak{F}_1$ to be trivial, we recover \cref{eqn:key-inequality}. The dual form of our large sieve inequality from \cref{thm:large_sieve2} now follows by summing over $\kn$, Mellin inversion, and shifting the contour of integration near $\Re(s) = 0$. In particular, this overcomes the barrier at $\Re(s) > \tfrac{1}{2}$ in the approach of Duke and Kowalski \cite{DK}, with no dependency on progress towards GRC as in \cite{TZ_GLn, humphries2024zeros}. This is why \cref{thm:large_sieve2} is supersedes both \cref{eqn:DK_large_sieve_conj} (which assumed GRC) and \cref{eqn:TZ_large_sieve}.

\subsubsection{Handling $L^{-1}$ and $\log L$ via positive semi-definite covers}
A key further step, inspired by \cite{humphries2024zeros}, is to extend the left-hand side of \cref{eqn:inequality-RS} to include Dirichlet coefficients $\lambda^\circ_{\pi \times \pi_0}(\n)$ of other related Dirichlet series, such as $L(s, \pi \times \pi_0)^{-1}$ or $\log L(s, \pi \times \pi_0)$. Although these coefficients do not form a positive semi-definite matrix in $\pi, \pi_0$, they turn out to have $\lambda_{\pi \times \tilde\pi'}$ as a \emph{positive semi-definite cover}, in a sense that we introduce in \cref{sec:non-def-covers}. This leads to the general inequality\footnote{Such a bound is implicit in the work of Humphries--Thorner \cite{humphries2024zeros} when $\lambda^\circ_{\pi \times \pi_0}(\kn)$ are the Dirichlet coefficients $L(s, \pi \times \pi_0)^{-1}$ and $(\kn, \kq_\pi \kq_{\pi_0}) = \mO_F$, but failing to handle ramified primes incurs significant losses without GRC.}
\begin{equation} \label{eqn:inequality-RS-covers}
    \Bigg| \sum_{\pi \in \mS} w(\pi) \lambda^\circ_{\pi \times \pi_0}(\n) \Bigg|^2
    \le 
    \lambda_{\pi_0 \times \tilde\pi_0}(\n) \sum_{\pi, \pi' \in \mS} w(\pi) \bar w(\pi') \lambda_{\pi \times \tilde\pi'}(\n).
\end{equation}
The resulting large sieve inequality for $\lambda^\circ_{\pi \times \pi_0}(\n)$ is given in \eqref{eqn:large-sieve-general} of \cref{thm:large-sieve-general}. Our proof of \cref{thm:ZDE} relies on inserting this result, with $\lambda^{\circ}_{\pi\times\pi_0}(\kn)$ equal to the $\kn$-th Dirichlet coefficient of $L(s,\pi\times\pi_0)^{-1}$, into the zero detection arguments of Humphries and Thorner \cite[Section 5]{humphries2024zeros}.

Another application is the density estimate for Langlands parameters from \cref{thm:density}. Specifically, we apply \cref{thm:large-sieve-general} for the coefficients of $\log L(s, \pi \times \tilde\pi')$, and $a(\n)$ supported on a single $\n = \p^k$. This improves the approach from \cite{lichtman2024density} in two ways. Firstly, our framework of positive semi-definite covers removes the ramification constraints at the place $\p$.  Secondly, shifting the contour far to the left (see \cref{eqn:large-sieve-general_2}) yields better bounds for $\theta < \tfrac{1}{4}$. As mentioned before, our proof of \cref{thm:density} is intertwined with that of the upper bound for the size of the universal family from \cref{cor:family}.

\subsubsection{Handling $L'/L$: Small and large prime powers}
The application to the \emph{log-free} zero density estimates from \cref{thm:LFZDE} requires working with the Dirichlet coefficients of $-\frac{L'}{L}(s, \pi \times \pi_0)$. These are closely related to the coefficients of $\log L(s, \pi \times \pi_0)$ (for which the argument above applies), but differ by a logarithmic factor, which we cannot afford to lose. 

To handle the family $\{L(s,\pi\times\pi_0)\colon\pi\in\mathcal{S}\}$ unconditionally, we prove a version of \eqref{eqn:inequality-RS-covers} with Selberg sieve weights inserted (see \eqref{eqn:large-sieve-general-sifted}).  This is inspired by the work of Thorner and Zaman \cite{TZ_GLn}, but our work alleviates the need to handle the ramified primes differently. The pre-sifted large sieve from \eqref{eqn:large-sieve-general-sifted} is enough to handle small powers of primes, but we trivially estimate the contribution of powers $\kp^k$ with $k\gg n^2$ using pointwise bounds towards GRC. This choice is inspired by the work of Soundararajan and Thorner \cite[Proof of Theorem 2.4]{soundararajan2019weak}, which addresses the case where $|\mathcal{S}|=1$. These ideas lead to \cref{cor:ls-log-derivative}, which is our key ingredient towards \cref{thm:LFZDE}.

Once \cref{cor:ls-log-derivative} is proved, we apply it to the zero detection process for $L$-functions that uses Tur{\'a}n's lower bound for power sums.  So far, this seems to be the only method whereby one can prove log-free zero density estimates for automorphic or Rankin--Selberg $L$-functions without recourse to unproven progress towards GRC. Since the zero density estimates in \cref{thm:ZDE} (not log-free) have the best known exponents given their level of generality, we take the time here to numerically sharpen the rather large exponents that occur in \cite{BrumleyThornerZaman,soundararajan2019weak,TZ_GLn} and thus \cref{thm:LFZDE}.

\subsubsection{The H\"older step: Trouble with common divisors} 
As indicated before, the duality approach to the $\GL_n$ large sieve has mostly been useful for large ranges of $N$. One can sometimes handle smaller ranges of $N$ by combining large sieves with H{\"o}lder's inequality, as in \cite[Section 21]{FI} and \cite[Section 5.2]{LOSmith}. However, in those settings, the large sieve concerns characters that are known to satisfy GRC, or (as in \cite{LOSmith}) the coefficients $a(\kn)$ are supported on squarefree or prime ideals.

In \cref{cor:large_sieve2}, we refine \cref{thm:large_sieve2} using H{\"o}lder's inequality {\it without any hypotheses} towards GRC or the support of $a(\kn)$. The goal is to bound $C(N,\mS)$ (and similarly $C(N,\mS,\pi_0)$) via the higher moment
\[
    \sum_{\pi \in \mS} \Bigg\vert \sum_{\N\kn \le N} \lambda_\pi(\kn) a(\kn) \Bigg\vert^{2k}
    =
    \sum_{\pi \in \mS} \Bigg\vert \sum_{\N\kn_1, \ldots, \N\kn_k \le N} \prod_{j=1}^k \lambda_\pi(\kn_j) a(\kn_j) \Bigg\vert^2.
\]
Morally, one can join together $\lambda_\pi(\kn_j)$ for $1 \le j \le k$ via the multiplicativity of Dirichlet coefficients to produce a single sum of length $N^k$, and then apply \cref{thm:large_sieve2}. However, executing this argument unconditionally turns out to be surprisingly subtle. One needs to consider various common divisors $\kd$ of $\kn_1, \ldots, \kn_k$, which is problematic without assuming GRC due to the potentially-large values of $\lambda_{\pi}(\kd)$. 
When $k = 2$, it is natural to try to separate $\kd = (\kn_1, \kn_2)$, but this approach seems to fail when the ideals $\kn_j$ are not squarefree. Instead, we factor $\kn_1 = \kd_1 \km_1$ where $\km_1$ contains only the primes that occur with a greater power in $\kn_1$ than in $\kn_2$. Factoring $\kn_2 = \kd_2 \km_2$ similarly, we obtain $\km_1\km_2 = [\kn_1\kn_2]$ and thus $\kd_1, \kd_2 \mid \km_1\km_2$, which is crucial for absorbing $\kd_1, \kd_2$ via a divisor bound later on. Our proof requires a more general combinatorial decomposition of $\kn_1, \ldots, \kn_k$ (see \cref{lem:decompose-ideals}), coupled with some careful applications of the Cauchy--Schwarz inequality, and the $\ell^2$ control of the coefficients $\lambda_\pi(\kd)$ (respectively, $\lambda_{\pi_0 \times \tilde\pi_0}(\kd)^{-1/2} \lambda_{\pi \times \pi_0}(\kd)$). We ultimately need to replace the norm $\|a\|_2$ from $C(N,\mS)$ with the larger quantity $N^{1/2}\|a\|_\infty$.  Fortunately, in most applications, the values $a(\kn)$ have roughly constant size on dyadic intervals, so this is not an important loss. This is the case in \cref{thm:2ndMoment}, which follows from \cref{cor:large_sieve2} and an approximate functional equation.

\section{Properties of \texorpdfstring{$L$}{L}-functions}
\label{sec:L-functions}


We recall some facts about $L$-functions of automorphic and Rankin--Selberg $L$-functions.

\subsection{Standard \texorpdfstring{$L$}{L}-functions}

Given $\pi\in\mathfrak{F}_n$, let $\widetilde{\pi}\in\mathfrak{F}_n$ be the contragredient representation and $\kq_{\pi}$ be the conductor of $\pi$.  We express $\pi$ as a restricted tensor product $\bigotimes_v \pi_v$ of smooth admissible representations of $\GL_n(F_v)$, where $v$ varies over places of $F$.  When $v$ is a non-archimedean place corresponding with a prime ideal $\kp$, then the local $L$-function $L(s,\pi_{\kp})$ is defined in terms of the Langlands parameters $A_{\pi}(\kp)=\{\alpha_{1,\pi}(\kp),\ldots,\alpha_{n,\pi}(\kp)\}$ by
\begin{equation}
	\label{eqn:Euler_p_single}
	L(s,\pi_{\kp})=\prod_{j=1}^{n}(1-\alpha_{j,\pi}(\kp)\N\kp^{-s})^{-1}=\sum_{k=0}^{\infty}\frac{\lambda_{\pi}(\kp^k)}{\N\kp^{ks}}.
\end{equation}
If $\kp\nmid\kq_{\pi}$, then $\alpha_{j,\pi}(\kp)\neq0$ for all $j$.  If $\kp \mid \kq_{\pi}$, then there might exist $j$ such that $\alpha_{j,\pi}(\kp)=0$.  The standard $L$-function $L(s,\pi)$ associated to $\pi$ is
\[
L(s,\pi)=\prod_{\kp} L(s,\pi_{\kp})=\sum_{\kn}\frac{\lambda_{\pi}(\kn)}{\N\kn^s}.
\]
The Euler product and Dirichlet series converge absolutely when $\re(s)>1$.

At each archimedean place $ v$ of $F$, there are $n$ Langlands parameters $\mu_{j,\pi}(v)\in\mathbb{C}$ such that
\[
L(s,\pi_{\infty}) = \prod_{v|\infty}\prod_{j=1}^{n}\Gamma_{v}(s+\mu_{j,\pi}(v)),\qquad \Gamma_{v}(s)= \begin{cases}
	\pi^{-s/2}\Gamma(s/2)&\mbox{if $F_{v}=\R$,}\\
	2(2\pi)^{-s}\Gamma(s)&\mbox{if $F_{v}=\mathbb{C}$.}
\end{cases}
\]
By combining the work in \cite{blomer2011ramanujan,LRS,MS}, we know that there exists
\begin{equation}
\label{eqn:ramanujan_progress}
\theta_n \in\Big[0,\frac{1}{2}-\frac{1}{n^2+1}\Big]
\end{equation}
such that
\begin{equation}
\label{eqn:LRS_finite}
	|\alpha_{j,\pi}(\kp)|\leq  \N\kp^{\theta_n}\qquad\textup{ and }\qquad\re(\mu_{j,\pi}(v))\geq -\theta_n.
\end{equation}
GRC asserts that in \eqref{eqn:ramanujan_progress}, one may take $\theta_n=0$.  We have $\kq_{\pi}=\kq_{\widetilde{\pi}}$, and for each $\kp$ and each $ v$, we have the equalities of sets $\{\alpha_{j,\widetilde{\pi}}(\kp)\}=\{\overline{ \alpha_{j,\pi}(\kp)}\}$ and $\{\mu_{j,\widetilde{\pi}}(v)\}=\{\overline{\mu_{j,\pi}(v)}\}$.

Let $r_{\pi}$ be the order of the pole of $L(s,\pi)$ at $s=1$.  The completed $L$-function
\[
\Lambda(s,\pi) = (s(1-s))^{r_{\pi}}(D_F^n \N\kq_{\pi})^{s/2}L(s,\pi)L(s,\pi_{\infty})
\]
is entire of order 1.  There exists a complex number $W(\pi)$ of modulus 1 such that if $s\in\mathbb{C}$, then the functional equation $\Lambda(s,\pi)=W(\pi)\Lambda(1-s,\widetilde{\pi})$ holds.  Let $d(v)=1$ if $F_{v}=\R$ and $d(v)=2$ if $F_{v}=\mathbb{C}$.  The analytic conductor of $\pi$ \cite{IS} is given by
\begin{equation}
\label{eqn:analytic_conductor_def}
\mathfrak{C}_{\pi}(it)= D_F^n \N\kq_{\pi}\prod_{v|\infty}\prod_{j=1}^n(3+|it+\mu_{j,\pi}(v)|^{d(v)}),\qquad \mathfrak{C}_{\pi}= \mathfrak{C}_{\pi}(0).
\end{equation}
Since $\Lambda(s,\pi)$ is entire of order 1, there exist complex numbers $a_{\pi}$ and $b_{\pi}$ such that
\[
\Lambda(s,\pi)=e^{a_{\pi}+b_{\pi}s}\prod_{\Lambda(\rho,\pi)=0}\Big(1-\frac{s}{\rho}\Big)e^{s/\rho}.
\]
The zeros $\rho$ in the above Hadamard product are the nontrivial zeros of $L(s,\pi)$, and the zeros of $L(s,\pi)$ that arise as poles of $s^{r_{\pi}}L(s,\pi_{\infty})$ are the trivial zeros.

\subsection{Rankin--Selberg \texorpdfstring{$L$}{L}-functions}
\label{subsec:RS}

Let $\pi\in\mathfrak{F}_n$ and $\pi'\in\mathfrak{F}_{n'}$.  At each prime ideal $\kp$, Jacquet, Piatetski-Shapiro, and Shalika \cite{JPSS} associate to $\pi_{\kp}$ and $\pi_{\kp}'$ a local Rankin--Selberg $L$-function
\begin{equation}
\label{eqn:RS_Dirichlet_series}
L(s,\pi_{\kp}\times\pi_{\kp}')=\prod_{j=1}^{n}\prod_{j'=1}^{n'}(1-\alpha_{j,j',\pi\times\pi'}(\kp) \N\kp^{-s})^{-1}=\sum_{k=0}^{\infty}\frac{\lambda_{\pi\times\pi'}(\kp^k)}{\N\kp^{ks}}
\end{equation}
and a local conductor $\kq_{\pi_{\kp}\times\pi_{\kp}'}$.  If $\kp\nmid \kq_{\pi}\kq_{\pi'}$, then we have the equality of sets
\begin{equation}
\label{eqn:separate_dirichlet_coeffs}
\{\alpha_{j,j',\pi\times\pi'}(\kp)\}=\{\alpha_{j,\pi}(\kp)\alpha_{j',\pi'}(\kp)\}.
\end{equation}
The Rankin--Selberg $L$-function $L(s,\pi\times\pi')$ associated to $\pi$ and $\pi'$ and its arithmetic conductor are
\[
L(s,\pi\times\pi')=\prod_{\kp}L(s,\pi_{\kp}\times\pi_{\kp}')=\sum_{\kn}\frac{\lambda_{\pi\times\pi'}(\kn)}{\N\kn^s},\qquad \kq_{\pi\times\pi'}=\prod_{\kp}\kq_{\pi_{\kp}\times\pi_{\kp}'}.
\]
At an archimedean place $v$ of $F$, Jacquet, Piatetski-Shapiro, and Shalika associate $n'n$ complex Langlands parameters $\mu_{j,j',\pi\times\pi'}(v)$ to $\pi_v$ and $\pi_v'$, from which one defines
\begin{align*}
L(s,\pi_{\infty}\times\pi_{\infty}') = \prod_{v|\infty}\prod_{j=1}^{n}\prod_{j'=1}^{n'}\Gamma_{v}(s+\mu_{j,j',\pi\times\pi'}(v))=\prod_{j=1}^{n'n[F:\Q]}\Gamma_{\R}(s+\mu_{\pi\times\pi'}(j)).
\end{align*}
Using the explicit descriptions of $\alpha_{j,j',\pi\times\pi'}(\kp)$ and $\mu_{j,j',\pi\times\pi'}(v)$ in \cite{humphries2019standard,soundararajan2019weak}, one sees that
\begin{equation}
\label{eqn:LRS_2}
|\alpha_{j,j',\pi\times\pi'}(\kp)|\leq\N\kp^{\theta_n + \theta_{n'}},\qquad \re(\mu_{\pi\times\pi'}(j))\geq -\theta_n - \theta_{n'}.
\end{equation}

Let $r_{\pi\times\pi'} = -\mathrm{ord}_{s=1}L(s,\pi\times\pi')$.  By our normalization for the central characters of $\pi$ and $\pi'$, we have that $r_{\pi\times\pi'}=0$ if and only if $\pi\neq \widetilde{\pi}'$, and $r_{\pi\times\widetilde{\pi}}=1$ otherwise.  The completed $L$-function
\begin{equation}
\label{eqn:Lambdaspixpi'}
\Lambda(s,\pi\times\pi')=(s(s-1))^{r_{\pi\times\pi'}}(D_F^{n'n}\N\kq_{\pi\times\pi'})^{s/2}L(s,\pi\times\pi')L(s,\pi_{\infty}\times\pi_{\infty}')
\end{equation}
is entire of order 1, and there exists a complex number $W(\pi\times\pi')$ of modulus 1 such that $\Lambda(s,\pi\times\pi')$ satisfies the functional equation $\Lambda(s,\pi\times\pi')=W(\pi\times\pi')\Lambda(1-s,\widetilde{\pi}\times\widetilde{\pi}')$.  As with $L(s,\pi)$, the analytic conductor of $L(s,\pi\times\pi')$ is given by
\begin{equation}
\label{eqn:analytic_conductor_def_2}
\mathfrak{C}_{\pi\times\pi'}(it)= D_F^{n'n}\N\kq_{\pi\times\pi'}\prod_{v|\infty}\prod_{j=1}^n \prod_{j'=1}^{n'}(3+|it+\mu_{j,j',\pi\times\pi'}(v)|^{d(v)}),\qquad \mathfrak{C}_{\pi\times\pi'}= \mathfrak{C}_{\pi\times\pi'}(0).
\end{equation}
The combined work of Bushnell and Henniart \cite{BH} and Wattanawanichkul \cite[Lemma A.1]{Wattanawanichkul} (see also \cite[Appendix]{humphries2019standard}) yields
\begin{equation}
\label{eqn:BH}
\mathfrak{C}_{\pi\times\pi'}(it)\ll \mathfrak{C}_{\pi\times\pi'}(3+|t|)^{[F:\Q] n'n},\qquad \mathfrak{C}_{\pi\times\pi'}\ll D_F^{-n'n}\mathfrak{C}_{\pi}^{n'}\mathfrak{C}_{\pi'}^{n}.
\end{equation}
Since $\Lambda(s,\pi\times\pi')$ is entire of order 1, there exist complex numbers $a_{\pi\times\pi'}$ and $b_{\pi\times\pi'}$ such that the Hadamard factorization
\begin{equation}
\label{eqn:hadamard}
\Lambda(s,\pi\times\pi')=e^{a_{\pi\times\pi'}+b_{\pi\times\pi'}s}\prod_{\Lambda(\rho,\pi\times\pi')=0}\Big(1-\frac{s}{\rho}\Big)e^{s/\rho}
\end{equation}
holds.  The zeros $\rho$ in \eqref{eqn:hadamard} are the nontrivial zeros of $L(s,\pi\times\pi')$, and the zeros of $L(s,\pi\times\pi')$ that arise as poles of $s^{r_{\pi\times\pi'}}L(s,\pi_{\infty}\times\pi_{\infty}')$ are the trivial zeros.

\begin{lemma}[{\cite[Lemma 3.2]{HarcosThorner}}]
\label{lem:Li1}
For $(\pi,\pi')\in\mathfrak{F}_n\times\mathfrak{F}_{n'}$, consider the holomorphic function
\[
\mathfrak{L}(s,\pi\times\pi') := \lim_{w\to s}\Big(\frac{w-1}{w+1}\Big)^{r_{\pi\times\pi'}}L(w,\pi\times\pi'),\qquad\Re(s)>-1.
\]
If $j\geq 0$, $\sigma\geq 0$, $t\in\R$, and $\epsilon>0$, then
\begin{equation}
\label{eqn:Lbound1}
\mathfrak{L}^{(j)}(\sigma+it,\pi\times\pi')\ll_{j,\epsilon}\mathfrak{C}_{\pi\times\pi'}(it)^{\max(1-\sigma,0)/2+\epsilon}.
\end{equation}
If $\sigma< 0$, $t\in\R$, and $\epsilon>0$, then
\begin{equation}
\label{eqn:Lbound2}
\mathfrak{L}(\sigma+it,\pi\times\pi')\ll_{\sigma,\epsilon}\mathfrak{C}_{\pi\times\pi'}(it)^{\frac{1}{2}-\sigma+\epsilon}.
\end{equation}
\end{lemma}

\begin{lemma}[{\cite[Lemma 3.1]{humphries2024zeros}}]
\label{lem:mertens}
	If $\pi\in\mathfrak{F}_n$, $X\geq 3$, and $\epsilon>0$, then $\sum_{\N\kn\leq X}\frac{\lambda_{\pi\times\widetilde{\pi}}(\kn)}{\N\kn}\ll_{\epsilon} \mathfrak{C}_{\pi}^{\epsilon}\log X$.
\end{lemma}

Define the numbers $\mu_{\pi\times\pi'}(\kn)$ and $\Lambda_{\pi\times\pi'}(\kn)$ by the Dirichlet series identities
\[
\sum_{\kn}\frac{\mu_{\pi\times\pi'}(\kn)}{\N\kn^s}=\frac{1}{L(s,\pi\times\pi')}=\prod_{\kp}\prod_{j=1}^n\prod_{j'=1}^{n'}(1-\alpha_{j,j',\pi\times\pi'}(\kp)\N\kp^{-s}),\qquad \re(s)>1
\]
and
\[
\sum_{\ka}\frac{\Lambda_{\pi\times\pi'}(\ka)}{\N\ka^s} = -\frac{L'}{L}(s,\pi\times\pi')  = \sum_{\kp}\sum_{\ell=1}^{\infty}\frac{\sum_{j=1}^n \sum_{j'=1}^{n'}\alpha_{j,j',\pi\times\pi'}(\kp)^{\ell}\log\N\kp}{\N\kp^{\ell s}},\qquad \re(s)>1.
\]
From the latter, we obtain the identity
\begin{align*}
\log L(s,\pi\times\pi') = \sum_{\kn\neq\mathcal{O}_F}\frac{\Lambda_{\pi\times\pi'}(\kn)}{\N\kn^s \log\N\kn},\qquad \re(s)>1.
\end{align*}

\section{Positive semi-definite covers} \label{sec:non-def-covers}

In this section, we significantly extend the ideas in \cite{lichtman2024density} on positive semi-definite families of $L$-functions. Our goal is the inequality \eqref{eqn:inequality-RS-covers} for Rankin--Selberg coefficients, reiterated in \cref{prop:inequality-RS-covers}---recall that such a result is implicit in the proof of \cite[Theorem 4.1]{humphries2024zeros} when $\n$ has no ramified prime factors.  Here, we treat ramified and unramified places on the same footing.  To this end, we use a fairly general linear-algebraic language, although our arguments are ultimately combinatorial.

\subsection{Covers for matrices}

We first introduce a generalization of the classical positive semi-definiteness property for matrices.

\begin{definition}
\label{def:covers-bilinear-forms}
    Let $\mS$ be a finite set and $a, a^+\colon \mS \times \mS \to \mathbb{C}$ be functions (which can be viewed as matrices in $\mathbb{C}^{\mS \times \mS}$). We say that $a^+(x, y)$ is a \emph{positive semi-definite cover} for $a(x, y)$ if and only if there exist countably-many functions $u_j\colon \mS \to \mathbb{C}$ and complex numbers $d_j$ with $|d_j| \le 1$ such that
    \begin{equation} \label{eqn:covers-bilinear-forms}
        a(x, y) = \sum_{j=1}^\infty d_j u_j(x) \bar u_j(y),
        \qquad\qquad 
        a^+(x, y) = \sum_{j=1}^\infty u_j(x) \bar u_j(y),
    \end{equation}
    for all $x, y \in \mS$, where the convergence is absolute.
\end{definition}

\begin{remark}
In particular, $a^+(x, y)$ is positive semi-definite if and only if it is a positive semi-definite cover for itself (or for the identically-zero matrix). Any Hermitian matrix has a positive semi-definite cover, obtained by replacing all eigenvalues in its spectral decomposition with their absolute values.
\end{remark}

To motivate this definition, consider the following inequality.

\begin{lemma} \label{lem:inequality-bilinear-forms}
Let $\mS$ be a finite set and $a, a^+\colon \mS \times \mS \to \mathbb{C}$ be functions such that $a^+(x, y)$ is a positive semi-definite cover for $a(x, y)$.  If $\mS_1, \mS_2 \subseteq \mS$ are subsets and $v\colon \mS_1 \to \mathbb{C}$, $w\colon \mS_2 \to \mathbb{C}$ are functions, then
\[
    \Bigg| \sum_{\substack{x \in \mS_1 \\ y \in \mS_2}} v(x)w(y) a(x, y) \Bigg|^2
    \le 
    \Bigg(\sum_{x, x' \in \mS_1} v(x)\bar v(x')a^+(x, x')\Bigg) \Bigg(\sum_{y,y' \in \mS_2} w(y) \bar w(y') a^+(y', y)\Bigg).
\]
\end{lemma}

\begin{proof}
Recall the notation in \cref{def:covers-bilinear-forms}.  By expanding $a(x, y)$ and $a^+(x, y)$ as in \eqref{eqn:covers-bilinear-forms}, swapping sums, and applying the Cauchy--Schwarz inequality in $j$, we obtain
\[
\begin{aligned}
    \Bigg| \sum_{\substack{x \in \mS_1 \\ y \in \mS_2}} v(x)w(y)a(x, y) \Bigg|^2
    &=
    \Bigg| \sum_{j = 1}^\infty d_j \sum_{x \in \mS_1} v(x) u_j(x) \sum_{y \in \mS_2} w(y) \bar u_j(y)\Bigg|^2
    \\
    &\le 
    \Bigg(\sum_{j=1}^\infty
    \Bigg| \sum_{x \in \mS_1} v(x) u_j(x) \Bigg|^2
    \Bigg)
    \Bigg(\sum_{j=1}^\infty
    \Bigg| \sum_{y \in \mS_2} w(y) \bar u_j(y)\Bigg|^2
    \Bigg)
    \\
    &= 
    \Bigg(
    \sum_{x, x' \in \mS_1} v(x) \bar v(x') \sum_{j=1}^\infty u_j(x) \bar u_j(x')
    \Bigg)
    \Bigg(
    \sum_{y, y' \in \mS_2} w(y) \bar w(y') \sum_{j=1}^\infty \bar u_j(y) u_j(y')
    \Bigg),
\end{aligned}
\]
which simplifies to the desired bound.
\end{proof}

\cref{def:covers-bilinear-forms} is also well-behaved, in the sense that positive semi-definite covers are stable under several operations which also preserve positive semi-definiteness.

\begin{lemma} \label{lem:covers-basic-prop}
    Let $\lambda \in \mathbb{C}$. If $a^+(x, y)$ (resp.\ $b^+(x, y)$) is a positive semi-definite cover of $a(x, y)$ (resp.\ $b(x, y)$), then $|\lambda| a^+(x, y)$ is a positive-definite cover of $\lambda a(x, y)$, $(a^+ + b^+)(x, y)$ is a positive-definite cover of $(a + b)(x, y)$, and  $(a^+ \cdot b^+)(x, y)$ is a positive-definite cover of $(a^+ \cdot b^+)(x, y)$.
\end{lemma}
\begin{proof}
Write $a(x, y) = \sum_{j=1}^\infty c_j u_j(x) \bar u_j(y)$ and $a^+(x, y) = \sum_{j=1}^\infty u_j(x) \bar u_j(y)$, according to \cref{def:covers-bilinear-forms}.  Similarly, write $b(x, y) = \sum_{k=1}^\infty d_k v_k(x) \bar v_k(y)$ and $b^+(x, y) = \sum_{k=1}^\infty v_k(x) \bar v_k(y)$.  Per \cref{def:covers-bilinear-forms}, the bounds $|c_j|, |d_k| \le 1$ hold. The claims for scaling and addition are immediate.  For multiplication, we have the expansion
\[
    a(x, y)\,b(x, y) = 
    \sum_{j=1}^\infty \sum_{k=1}^\infty c_j d_k\ (u_j \cdot v_k)(x)\ \bar{(u_j \cdot v_k)(y)},
\]
where $|c_jd_k| \le 1$,
and $a^+(x,y)\,b^+(x, y)$ similarly.  We can then re-index to a single countable sum.
\end{proof}

\subsection{Covers for families of Dirichlet series} 
We now pass from matrices $a(x, y)$ to families of Dirichlet series $A(s; x, y)$. The reader should keep in mind the case when $\mS$ is a family of automorphic representations, and the Dirichlet series $A(s; x, y)$ come from Rankin--Selberg convolutions.

\begin{definition} \label{def:covers-l-functions}
Let $F$ be a number field, $\mS$ be a finite set, and $(A(s; x, y))_{x, y \in \mS}$, $(A^+(s; x, y))_{x, y \in \mS}$ be families of formal Dirichlet series with expansions
    \[
        A(s; x, y) = \sum_\n \frac{a(\n; x, y)}{\N \n^s},
        \qquad\quad 
        A^+(s; x, y) = \sum_\n \frac{a^+(\n; x, y)}{\N \n^s}.
    \]
    We say that $A^+(s; x, y)$ is a {\bf positive semi-definite cover} for $A(s; x, y)$ if for each $\n$, $a^+(\n; x, y)$ is a positive semi-definite cover for $a(\n; x, y)$ (indexing over $x, y \in \mS$). In other words, there exist functions $u_j : \mS \to \mathbb{C}$, complex numbers $d_j$ with $|d_j| \le 1$, and integral ideals $\n_j$ such that
    \[
        A(s; x, y) = \sum_{j=1}^\infty d_j \frac{u_j(x) \bar u_j(y)}{\N \n_j^s},
        \qquad\quad 
        A^+(s; x, y) = \sum_{j=1}^\infty \frac{u_j(x) \bar u_j(y)}{\N \n_j^s},
    \]
    where the convergence of the Dirichlet coefficient of each $\N \n^{-s}$ is absolute, pointwise in $x, y$. We say that $A^+(s; x, y)$ is {\bf positive semi-definite} if it is a positive semi-definite cover for $0$.
\end{definition}

Fortunately, several natural operations on $L$-functions preserve positive semi-definite covers.  The following lemma is a generalization of \cite[Lemma 3]{lichtman2024density}.

\begin{lemma} \label{lem:nondef-l-functions}
    Let $z \in \mathbb{C}$. 
    If $A^+(s; x, y)$ (resp.\ $B^+(s; x, y)$) is a positive semi-definite cover of $A(s; x, y)$ (resp.\ $B^+(s; x, y)$), then the families
    \[
        |z| A^+(s; x, y),
        \quad 
        (A^+ + B^+)(s; x, y),
        \quad 
        (A^+ \cdot B^+)(s; x, y),
        \quad
        \exp A^+(s; x, y)
    \]
    are positive semi-definite covers, respectively, of
    \[
        |z| A(s; x, y),
        \quad 
        (A + B)(s; x, y),
        \quad 
        (A \cdot B)(s; x, y),
        \quad
        \exp A(s; x, y)
    \]
    If $C(s; x, y)$ is positive semi-definite, then $\exp C(s; x, y)$ is a positive semi-definite cover for $C(s; x, y)$.
\end{lemma}

\begin{proof}
The claims about the covers of $z A(s; x, y)$, $(A+B)(s; x, y)$, and $(A \cdot B)(s; x, y)$ follow from \cref{lem:covers-basic-prop}. For the exponentiation property, we use the expansions
\[
    \exp A(s; x, y) = \sum_{k = 0}^\infty \frac{1}{k!} A(s; x, y)^k,
    \qquad\quad 
    \exp A^+(s; x, y) = \sum_{k = 0}^\infty \frac{1}{k!} A^+(s; x, y)^k
\]
with the previous properties (note that $A^+(s; x, y)^k$ is a positive semi-definite cover for $A(s; x, y)^k$ for each $k$). For the final claim, expand $\exp C(s; x, y)$ as above and consider the term $k = 1$.
\end{proof}

Finally, we apply these notions to Rankin--Selberg $L$-functions.

\begin{lemma}
If $\mathcal{S}\subseteq \bigcup_{n=1}^{\infty}\mathfrak{F}_n$ is finite, then $(\log L(s, \pi \times \tilde\pi'))_{\pi, \pi' \in \mS}$ is positive semi-definite.
\end{lemma}

\begin{proof}
This is a purely formal generalization of \cite[Proposition 4]{lichtman2024density} to variable ranks $n$ and general number fields. 
When $F = \Q$, the computations of Brumley \cite[Appendix A.2, (A.8)]{soundararajan2019weak} express each factor in the Euler product $L(s, \pi \times \tilde\pi') = \prod_{p\text{ prime}} L(s, \pi_p \times \tilde\pi'_p)$, with minor changes of notation, as
\[
    \log L(s, \pi_p \times \tilde \pi'_p) =
    \sum_{[\varrho_\ell] \text{ for } \Q_p}\ 
    \sum_{f, \nu \ge 1} 
    \frac{1}{f p^{e_\ell f(s-\nu)}}
    \sum_{\substack{\rho_j(\pi_p) \in [\varrho_\ell] \\ n_j(\pi_p) \ge \nu}}\ 
    z_j(\pi_p, \ell)^{e_\ell f}
    \bar{ \sum_{\substack{\rho_k(\pi'_p) \in [\varrho_\ell] \\ n_k(\pi'_p) \ge \nu}} 
    z_k(\pi'_p, \ell)^{e_\ell f}
    },
\]
where $[\varrho_\ell]$ ranges over all the twist-equivalence classes of unitary supercuspidal representations of a general linear group over $\Q_p$, $e_\ell$ is the torsion number of $\varrho_\ell$, $\rho_j(\pi_p)$ (resp., $\rho_k(\pi'_p)$) ranges over the supercuspidals in a deconstruction of the $p$-adic components $\pi_p$ (resp., $\pi'_p$) that are twist-equivalent to $\varrho_\ell$, $n_j(\pi_p)$ are positive integers depending only on $\pi_p$ and $j$, and $z_j(\pi_p, \ell)$ are complex numbers depending only on $\pi_p, \ell, j$. Summing over primes $p$ and collecting terms, one can ultimately write
\[
    \log L(s, \pi \times \tilde\pi') = \sum_j \frac{S_j(\pi) \bar S_j(\pi')}{n_j^s},
\]
where $j = (p, \ell, f, \nu)$ is a tuple of parameters with countably many values, $n_j = p^{e_\ell f}$, $S_j(\pi)$ are complex numbers depending only on $\pi$ and $j$, and the Dirichlet coefficient of each $n_j^{-s}$ is eventually constant in the series above. For a general number field $F$, an analogous argument yields
\begin{equation} \label{eqn:log-non-def}
    \log L(s, \pi \times \tilde \pi') = \sum_j \frac{S_j(\pi) \bar S_j(\pi')}{\N \n_j^s},
\end{equation}
with all parameters depending implicitly on $F$. This is the form required by \cref{def:covers-l-functions}.
\end{proof}

Combining our ingredients in this section, we obtain the following result.
\begin{proposition}
\label{prop:inequality-RS-covers}
Let $\mS_1, \mS_2$ be finite subsets of $\bigcup_{n=1}^{\infty}\mathfrak{F}_n$. With the understanding that families are indexed over $\mS_1 \cup \{\tilde\pi : \pi \in \mS_2\}$, the following results hold:
\begin{itemize}
    \item[$(i)$.] $L(s, \pi \times \tilde\pi')$ is a positive semi-definite cover for itself, $L(s, \pi \times \tilde\pi')^{-1}$, and $\log L(s, \pi \times \tilde\pi')$.
    \item[$(ii)$.] Let $L^+(s, \pi \times \tilde\pi') = \sum_\n \lambda^+_{\pi \times \tilde\pi'}(\n) \N\n^{-s}$ be a positive semi-definite cover for $L^\circ(s, \pi \times \tilde\pi') = \sum_\n \lambda^\circ_{\pi \times \tilde\pi'}(\n) \N\n^{-s}$.  If $v\colon \mS_1 \to \mathbb{C}$ and $w\colon \mS_2 \to \mathbb{C}$ are functions, then
\[
    \Bigg| \sum_{\substack{\pi_1 \in \mS_1 \\ \pi_2 \in \mS_2}} v(\pi_1) w(\pi_2) \lambda^\circ_{\pi_1 
    \times \pi_2}(\n) \Bigg|^2 \le \Bigg( \sum_{\pi_1, \pi_1' \in \mS_1} v(\pi_1) \bar v(\pi_1') \lambda^+_{\pi_1 \times \tilde\pi_1'}(\n) \Bigg) \Bigg( \sum_{\pi_2, \pi_2' \in \mS_2} w(\pi_2) \bar w(\pi_2') \lambda^+_{\tilde\pi_2' \times \pi_2}(\n) \Bigg).
\]
In particular, if $\mS_1 = \{\pi_1\}$ consists of only one element, then
\begin{equation}
\label{eqn:average-RS-bound}
    \Bigg| \sum_{\pi_2 \in \mS_2} w(\pi_2) \lambda^\circ_{\pi_1 
    \times \pi_2}(\n) \Bigg|^2 \le \lambda^+_{\pi_1 \times \tilde\pi_1}(\n) \sum_{\pi_2, \pi_2' \in \mS_2} w(\pi_2) \bar w(\pi_2') \lambda^+_{\tilde\pi_2' \times \pi_2}(\n).
\end{equation}
If $\mS_2 = \{\pi_2\}$, then \eqref{eqn:average-RS-bound} further simplifies to 
\begin{equation}
\label{eqn:pointwise-RS-bound}
    |\lambda^\circ_{\pi_1 \times \pi_2}(\n)|^2 \le \lambda^+_{\pi_1 \times \tilde\pi_1}(\n) \lambda^+_{\tilde\pi_2 \times \pi_2}(\n).
\end{equation}
\end{itemize}    
\end{proposition}

\begin{proof}
    For the first two claims in (i), note that $\log L(s, \pi \times \tilde\pi')$ is a positive semi-definite cover for $\pm \log L(s, \pi \times \tilde\pi')$, and exponentiate both sides. The last claim in (i) follows directly from the last part of \cref{lem:nondef-l-functions}. Part (ii) is simply an application of \cref{lem:inequality-bilinear-forms} for $\mS_1$ and $\{\tilde\pi : \pi \in \mS_2\}$.
\end{proof}

\begin{remark}
\cref{prop:inequality-RS-covers} generalizes and refines several results found in literature.
\begin{itemize}
\item[(1).]
Take $L^\circ = L^+ = L$, i.e., $\lambda^\circ_{\pi \times \pi'}(\n) = \lambda^+_{\pi \times \tilde\pi'}(\n) = \lambda_{\pi \times \tilde\pi'}(\n)$. Then \eqref{eqn:average-RS-bound} matches our claim from \eqref{eqn:inequality-RS}, which implies \eqref{eqn:key-inequality} and \eqref{eqn:posdef-claim} by taking $\pi_1$ to be the trivial representation.  Recall that \eqref{eqn:key-inequality} extends the bound \eqref{eqn:approximation2} used by Thorner--Zaman \cite{TZ_GLn} in the unramified case. Meanwhile, \eqref{eqn:pointwise-RS-bound} recovers the bound $|\lambda_{\pi_1\times\pi_2}(\kn)|^2\leq \lambda_{\pi_1\times\tilde{\pi}_1}(\kn)\lambda_{\pi_2\times\tilde{\pi}_2}(\kn)$, which extends \cite[Lemma 3.1]{JLW} to number fields.
\item[(2).] Take $\lambda^\circ_{\pi \times \tilde\pi'}(\n) = \lambda^+_{\pi \times \tilde\pi'}(\n) = \lambda_{\pi \times \tilde\pi'}(\n) - \lambda_\pi(\n) \bar\lambda_{\pi'}(\n)$, which is positive semi-definite in $\pi, \pi'$ by \eqref{eqn:posdef-claim}. Then \eqref{eqn:pointwise-RS-bound} becomes the bound
\[
    |\lambda_{\pi_1 \times \pi_2}(\n) - \lambda_{\pi_1}(\n) \lambda_{\pi_2}(\n)|^2 \le (\lambda_{\pi_1 \times \tilde\pi_1}(\n) - |\lambda_{\pi_1}(\n)|^2) (\lambda_{\pi_2 \times \tilde\pi_2}(\n) - |\lambda_{\pi_2}(\n)|^2),
\]
which removes the real part and the coprimality constraints from \cite[Proposition 3.1]{TZ_GLn}.
\item[(3).] Take $L^\circ = L^{-1}$ and $L^+ = L$, i.e., $\lambda^\circ_{\pi \times \tilde\pi'}(\n)$ are the coefficients $\mu_{\pi\times\pi'}(\kn)$ of $L(s, \pi \times \tilde\pi')^{-1}$, and $\lambda^+_{\pi \times \tilde\pi'}(\n) = \lambda_{\pi \times \tilde\pi'}(\n)$.  Then, we recover from \eqref{eqn:average-RS-bound} and \eqref{eqn:pointwise-RS-bound} the bounds
\begin{equation}
\label{eqn:mu_bound}
\begin{aligned}
 \Bigg| \sum_{\pi_2 \in \mS_2} w(\pi_2) \mu_{\pi_1 
    \times \pi_2}(\n) \Bigg|^2 &\le \lambda_{\pi_1 \times \tilde\pi_1}(\n) \sum_{\pi_2, \pi_2' \in \mS_2} w(\pi_2) \bar w(\pi_2') \lambda_{\pi_2 \times \tilde\pi_2'}(\n),\\
    |\mu_{\pi\times\pi'}(\kn)|^2&\leq \lambda_{\pi\times\tilde{\pi}}(\kn)\lambda_{\pi'\times\tilde{\pi}'}(\kn),
\end{aligned}
\end{equation}
respectively.  These were proved in \cite[Section 4]{humphries2024zeros} when $\gcd(\kn,\kq_{\pi}\kq_{\pi'})=\mathcal{O}_F$.  The bound \eqref{eqn:mu_bound} was independently proved by Jiang using other means in \cite[Lemma~3.1]{Jiang}.
\item[(4).] When $L^+ = L^\circ = \log L$, \eqref{eqn:pointwise-RS-bound} recovers the inequality $|\Lambda_{\pi\times\pi'}(\kn)|^2\leq \Lambda_{\pi\times\tilde{\pi}}(\kn)\Lambda_{\pi'\times\tilde{\pi}'}(\kn)$ from \cite[Proposition A.1]{soundararajan2019weak}. When $L^\circ = \log L$, $L^+ = L$, and $\pi_1 = \pi_2$, \eqref{eqn:pointwise-RS-bound} recovers the bound $\Lambda_{\pi\times\tilde{\pi}}(\kp^k)\leq k\lambda_{\pi\times\tilde{\pi}}(\kp^k) \log \N\kp$, which is \cite[(6.10)]{soundararajan2019weak}. These were the key inequalities for Rankin--Selberg coefficients used in the work of Soundararajan and Thorner \cite{soundararajan2019weak}.
\end{itemize}
\end{remark}

\section{A large sieve for covered coefficients} \label{sec:large-sieve}

We now state a general large sieve inequality for the Dirichlet coefficients $\lambda_{\pi \times \tilde\pi'}$ of Rankin--Selberg $L$-functions, and in fact for any coefficients $\lambda^\circ_{\pi \times \tilde\pi'}$ which have $\lambda_{\pi \times \tilde\pi'}$ as a positive semi-definite cover. This improves on \cite[Theorem 4.2]{TZ_GLn} and \cite[Theorem 4.1]{humphries2024zeros} by factors of $Q^{4\theta_n n^2}$ in the off-diagonal terms, while also weakening the required lower bound on the sifting parameter $z$.

\begin{theorem}
\label{thm:large-sieve-general}
Let $\eps > 0$ and $x, T \ge 1$.  If $a(\n)$ is a complex-valued function, $\pi_0 \in \F_{n_0}$, $\mS \subseteq \F_n$ is finite, $\max_{\pi \in \mS} \mathfrak{C}_{\pi} = Q$, and $L(s, \pi \times \tilde\pi')$ is a positive semi-definite cover for $L^\circ(s, \pi \times \tilde\pi') = \sum_\n \lambda^\circ_{\pi \times \tilde\pi'}(\n) \N \n^{-s}$ (indexing by $\pi, \pi' \in \mS \cup \{\tilde\pi_0\}$), then
{\small\begin{equation}
\label{eqn:large-sieve-general}
    \sum_{\pi \in \mS}
    \Bigg| \sum_{\N \n \in (x, e^{1/T}x]} a(\n) \lambda^\circ_{\pi \times \pi_0}(\n)   \Bigg|^2
    \ll_{\epsilon}
    Q^\eps
  \Big(\frac{x}{T} + D_F^{-\frac{n^2}{2}}Q^n T^{\frac{[F:\Q]n^2}{2}+\epsilon} |\mS|\Big) \sum_{\N \n \in (x, e^{1/T}x]} |a(\n)|^2 \lambda_{\pi_0 \times \tilde\pi_0}(\n).
\end{equation}}%
Also, if $x\geq D_F^{-n^2}Q^{2n}|\mathcal{S}|^{\epsilon}$, then
\begin{equation}
\label{eqn:large-sieve-general_2}
    \sum_{\pi \in \mS}
    \Bigg| \sum_{\N \n \in (x, ex]} a(\n) \lambda^\circ_{\pi \times \pi_0}(\n)   \Bigg|^2
    \ll_{\epsilon}
  Q^{\epsilon} x
    \sum_{\N \n \in (x, ex]} |a(\n)|^2 \lambda_{\pi_0 \times \tilde\pi_0}(\n).
\end{equation}
Finally, there exists an effectively computable constant $\Cl[abcon]{large_sieve_z}=\Cr{large_sieve_z}(n,[F:\Q],\epsilon)\geq 1$ such that if $z \geq \Cr{large_sieve_z} D_F^{-n^2/2}Q^{n+\eps}$, then 
\begin{equation} \label{eqn:large-sieve-general-sifted}
\begin{aligned}
    &\sum_{\pi \in \mS}
    \Bigg| \sum_{\substack{\N \n \in (x, e^{1/T} x] \\ \p \mid \n\, \Rightarrow\, \N \p > z}} a(\n) \lambda^\circ_{\pi \times \pi_0}(\n) \Bigg|^2
    \\
    &\ll_{\epsilon}
    \frac{1}{\log z}\Big(\frac{x}{T} + D_F^{-\frac{n^2}{2}}Q^{n+\eps} z^{2n^2+2+\eps} T^{\frac{[F:\Q]n^2}{2}+\epsilon} 
     |\mS|\Big)
    \sum_{\substack{\N \n \in (x, e^{1/T} x] \\ \p \mid \n\, \Rightarrow\, \N \p > z}} |a(\n)|^2 \lambda_{\pi_0 \times \tilde\pi_0}(\n).
\end{aligned}
\end{equation}
\end{theorem}

\begin{proof}[Proof of \cref{thm:large_sieve2} assuming \cref{thm:large-sieve-general}]
Take $T=1$.  If $\lambda^\circ = \lambda$, then the second bound in \cref{thm:large_sieve2} follows from \eqref{eqn:large-sieve-general} and a standard dyadic decomposition argument.  The first bound in \cref{thm:large_sieve2} follows from the second, with $\pi_0=\mathbbm{1}$.
\end{proof}

To prove \cref{thm:large-sieve-general}, we require a version of \cite[Lemma 4.1]{TZ_GLn} without a coprimality constraint.  Given a compactly-supported smooth function $\phi\colon \R \to \mathbb{C}$, we define the entire function $\hat\phi(s) := \int_\R \phi(y) e^{sy} dy$.  For each integer $k\geq 0$, the bound $\hat\phi(s) \ll_{\phi,k,\Re(s)} e^{2|\mathrm{Re}(s)|} (1 + |s|)^{-k}$ holds.

\begin{lemma}
\label{lem:rs-sums}
    Fix a smooth test function $\phi$ with compact support in $[-2, 2]$.  Let $T, x \ge 1$; $\pi, \pi' \in \F_n$; $\mathfrak{C}_{\pi}, \mathfrak{C}_{\pi'} \le Q$; and $\epsilon>0$. For $\d$ squarefree, define
    \begin{equation} \label{eqn:g-pi-notation}
        g_\pi(\d)
        := 
        \prod_{\p \mid \d} (1 - L(1, \pi_\p \times \tilde \pi_\p)^{-1}).
    \end{equation}
   There holds
    \[
        \Bigg|\sum_{\substack{\n \\ \d \mid \n}} \phi\Big(T \log \frac{\N \n}{x}\Big)
        \lambda_{\pi \times \tilde\pi'}(\n)
        -
        g_\pi(\d) x
        \frac{\hat\phi(T^{-1})}{T}
        \mathop{\Res}_{s = 1} L(s, \pi \times \tilde \pi')\Bigg|\ll_{\epsilon}D_F^{-\frac{n^2}{2}}Q^{n+\epsilon} \N\d^{n^2+\epsilon} T^{\frac{[F:\Q]n^2}{2}+\epsilon}.
    \]
    Also, if $\sigma<0$, then
    \[
        \Bigg|\sum_{\substack{\n}} \phi\Big(T \log \frac{\N \n}{x}\Big)
        \lambda_{\pi \times \tilde\pi'}(\n)
        -
        x
        \frac{\hat\phi(T^{-1})}{T}
        \mathop{\Res}_{s = 1} L(s, \pi \times \tilde \pi')\Bigg|\ll_{\sigma,\epsilon}
        x^{\sigma}(D_F^{-n^2}Q^{2n} T^{[F:\Q]n^2})^{\frac{1}{2}-\sigma+\frac{\epsilon}{2n^2}}.
    \]
\end{lemma}

\begin{proof}
The first result is identical to the proof of \cite[Lemma 4.1]{TZ_GLn} except that we do not need to bound the contribution of the prime ideals $\p \mid \q_\pi \q_{\pi'}$ separately.  This removes a factor of $Q^{4\theta_n n^2}$ from our final bound compared to \cite{TZ_GLn}. Note that the main term vanishes unless $\pi = \pi'$.

The basis for the proof in \cite{TZ_GLn} is a contour integration argument, where the contour is pushed slightly to the right of the line $\re(s)=0$ and \eqref{eqn:Lbound1} is applied.  The second result deviates from the first only in that we push the contour even further to the left, in which case we apply \eqref{eqn:Lbound2}.
\end{proof}

\begin{proof}[Proof of \cref{thm:large-sieve-general}]
We take $\eps < 1/10$ without loss of generality. Like \cite[Theorem 4.2]{TZ_GLn}, our proof will use duality and Selberg's sieve; our key new input is \cref{prop:inequality-RS-covers}, which allows us to handle the ramified and unramified primes on the same footing.  Let $g_{\pi}(\kd)$ be as in \eqref{eqn:g-pi-notation}.  Let $z\geq 1$, and let $\rho_{\pi,z}(\kd)$ be a real-valued function such that
\begin{equation} \label{eqn:sieve-weights}
    \begin{cases} 
    \rho_{\pi,z}(\mO_F) = 1, \\ 
    \rho_{\pi,z}(\d) = 0\
    \text{ unless }\ 
    \kd \in \mD_{\pi,z} := \left\{\d : \N \d \le z,\ \d \mid P_{\pi,z}\right\},\text{ where }\ 
    P_{\pi,z} := \prod_{\substack{\N \p \le z \\ g_\pi(\p) \neq 0}} \p,\\
    |\rho_{\pi,z}(\kd)|\leq 1.
    \end{cases}
\end{equation}
We consider the general sifted sum
\begin{equation}\label{eqn:general-large-sieve-sum}
    \mL(z) := \max_{\|b\|_2 = 1} \sum_{\pi \in \mS}
    \Bigg| \sum_{\substack{\N \n \in (x, e^{1/T}x] \\ \lambda_{\pi_0 \times \tilde\pi_0}(\n) \neq 0}} b(\n) \frac{\lambda^\circ_{\pi \times \pi_0}(\n)}{\sqrt{\lambda_{\pi_0 \times \tilde\pi_0}(\n)}} \sum_{\d \mid \n} \rho_{\pi,z}(\d) \Bigg|^2,
\end{equation}
observing from \eqref{eqn:sieve-weights} that
\begin{equation} \label{eqn:key-sieve-property}
    \sum_{\d \mid \n} \rho_{\pi,z}(\d) = 1
    \qquad 
    \text{ if all prime factors } \p \mid \n \text{ satisfy } \N \p > z.
\end{equation}
We deduce \eqref{eqn:large-sieve-general}--\eqref{eqn:large-sieve-general-sifted} by taking different values of $z$.

By duality, there exists a function $w : \mS \to \mathbb{C}$ with $\sum_{\pi \in \mS} |w(\pi)|^2 = 1$ such that
\[
    \mL(z) = \sum_{\substack{\N \n \in (x, e^{1/T}x] \\ \lambda_{\pi_0 \times \tilde\pi_0}(\n) \neq 0}} 
    \Bigg| \sum_{\pi \in \mS} w(\pi) \frac{\lambda^\circ_{\pi \times \pi_0}(\n)}{\sqrt{\lambda_{\pi_0 \times \tilde\pi_0}(\n)}} \sum_{\d \mid \n} \rho_{\pi,z}(\d) \Bigg|^2.
\]
We apply \eqref{eqn:average-RS-bound} with $\pi_1 = \pi_0$, $\mS_2 = \mS$, and the weights $w(\pi) \sum_{\d \mid \n} \rho_{\pi,z}(\d)$, thus obtaining the bound
\[
    \mL(z) \le \sum_{\substack{\N \n \in (x, e^{1/T}x] \\ \lambda_{\pi_0 \times \tilde\pi_0}(\n) \neq 0}} 
    \sum_{\pi, \pi' \in \mS} \Bigg(w(\pi) \sum_{\d \mid \n} \rho_{\pi,z}(\d)\Bigg) \bar{\Bigg(w(\pi') \sum_{\d' \mid \n} \rho_{\pi',z}(\d')\Bigg)} \lambda_{\pi \times \tilde\pi'}(\n).
\]
Each summand in the $\kn$-sum is nonnegative.  Therefore, we can drop the condition $\lambda_{\pi_0 \times \tilde\pi_0}(\n) \neq 0$ and insert a smooth majorant for the indicator function of $(x,e^{1/T}x]$.  Let $\phi : \R\to[0,1]$ be an infinitely differentiable function,  supported in $[-2, 2]$ and bounding from above the indicator function of $[0,1]$. Upon swapping sums, we find that
\[
\begin{aligned}
    \mL(z) 
    &\le 
    \sum_{\N\n \in (x, e^{1/T} x]} \phi\Big(T \log \frac{\N \n}{x}\Big)
    \sum_{\pi, \pi' \in \mS} \Bigg(w(\pi) \sum_{\d \mid \n} \rho_{\pi,z}(\d)\Bigg) \bar{\Bigg(w(\pi') \sum_{\d' \mid \n} \rho_{\pi',z}(\d')\Bigg)} \lambda_{\pi \times \tilde\pi'}(\n)
    \\
    &= 
    \sum_{\pi, \pi' \in \mS} w(\pi) \bar w(\pi') \sum_{\substack{\d \in \mD_{\pi,z} \\ \d' \in \mD_{\pi',z}}} \rho_{\pi,z}(\d) \rho_{\pi',z}(\d')
    \sum_{\substack{\n \\ [\d, \d'] \mid \n}} 
    \phi\Big(T \log \frac{\N \n}{x}\Big) \lambda_{\pi \times \tilde\pi'}(\n).
\end{aligned}
\]
Using \cref{lem:rs-sums} and \eqref{eqn:sieve-weights} to evaluate the inner sums, we obtain
\[
\begin{aligned}
    \mL(z) 
    &\le 
    x \frac{\hat\phi(T^{-1})}{T} \sum_{\pi \in \mS} |w(\pi)|^2 \mathop{\Res}_{s = 1} L(s, \pi \times \tilde\pi)
    \sum_{\d, \d' \in \mD_{\pi,z}} \rho_{\pi,z}(\d) \rho_{\pi,z}(\d') g_\pi([\d,\d'])
    \\
    &+
    O_\eps\Bigg(D_F^{-\frac{n^2}{2}}Q^{n+\eps} T^{\frac{[F:\Q]n^2}{2}+\epsilon} \sum_{\pi, \pi' \in \mS} |w(\pi) w(\pi')| \sum_{\substack{\d \in \mD_{\pi,z} \\ \d' \in \mD_{\pi',z}}} \N [\d, \d']^{n^2+\eps} \Bigg),
\end{aligned}
\]
where the main term comes from the diagonal terms $\pi = \pi'$.  Since $\sum_{\N \d \le z} 1 \ll_\eps z^{1+\eps}$ and $\sum_\pi |w(\pi)|^2 = 1$, the error term is $O_\eps(D_F^{-n^2/2}Q^{n+\eps} z^{2n^2+2+\eps} T^{(1/2)[F:\Q]n^2+\epsilon} |\mS| )$, as in \cite[(4.16)]{TZ_GLn}.

To evaluate the main term, we note (as in \cite[p.\,15]{TZ_GLn}) that \cite[Theorem 7.1]{friedlander2010opera} implies the existence of Selberg sieve weights $\rho_\pi(\d)$ satisfying both \eqref{eqn:sieve-weights} and
\[
\begin{aligned}
    \sum_{\d, \d' \in \mD_{\pi,z}} \rho_{\pi,z}(\d) \rho_{\pi,z}(\d') g_\pi([\d,\d'])=
    \Bigg(\sum_{\substack{\N \d \le z \\ \d \text{ square-free}}} \prod_{\p \mid \d} \sum_{j = 1}^\infty \frac{\lambda_{\pi \times \tilde\pi}(\p^j)}{\N \p^j}\Bigg)^{-1}
    \le 
    \Bigg(\sum_{\N \n \le z} \frac{\lambda_{\pi \times \tilde\pi}(\n)}{\N \n}\Bigg)^{-1}.
\end{aligned}
\]
Using $\hat\phi(T^{-1}) \ll 1$ and $\sum_{\pi \in \mS} |w(\pi)|^2 = 1$, we conclude that
\begin{equation}
\label{eqn:general-large-sieve-bound}
    \mL(z) \ll_{\epsilon} \frac{x}{T} \max_{\pi \in \mS} \frac{\mathop{\Res}_{s = 1} L(s, \pi \times \tilde\pi)}{\sum_{\N \n \le z} \frac{\lambda_{\pi \times \tilde\pi}(\n)}{\N \n}}
    +
    D_F^{-\frac{n^2}{2}}Q^{n+\eps} z^{2n^2+2+\eps} T^{\frac{[F:\Q]n^2}{2}+\epsilon} |\mS|.
\end{equation}

To deduce \eqref{eqn:large-sieve-general}, we first note that if $\sum_{\N \n \in (x, e^{1/T}x]} |a(\n)|^2 \lambda_{\pi_0 \times \tilde\pi_0}(\n)=0$, then the result is trivial.  When this sum is nonzero, we take $z=1$. Combining \cref{lem:Li1}, \eqref{eqn:general-large-sieve-sum}, \eqref{eqn:key-sieve-property}, and \eqref{eqn:general-large-sieve-bound}, we obtain
\begin{equation}
\label{eqn:pre_b_choice}
    \max_{\|b\|_2 = 1} \sum_{\pi \in \mS}
    \Bigg| \sum_{\substack{\N \n \in (x, e^{1/T}x] \\ \lambda_{\pi_0 \times \tilde\pi_0}(\n) \neq 0}} b(\n) \frac{\lambda^\circ_{\pi \times \pi_0}(\n)}{\sqrt{\lambda_{\pi_0 \times \tilde\pi_0}(\n)}} \Bigg|^2
    \ll_{\epsilon} Q^\eps (x + D_F^{-\frac{n^2}{2}}Q^n T^{\frac{[F:\Q]n^2}{2}+\epsilon}|\mS|).
\end{equation}
If we choose
\[
    b(\n) := \frac{a(\n) \sqrt{\lambda_{\pi_0 \times \tilde\pi_0}(\n)}}{\sum_{\N \ka \in (x, e^{1/T}x]} |a(\ka)|^2 \lambda_{\pi_0 \times \tilde\pi_0}(\ka)},
\]
then \eqref{eqn:pre_b_choice} reduces to
\[
   \sum_{\pi \in \mS}
    \Bigg| \sum_{\substack{\N \n \in (x, e^{1/T}x] \\ \lambda_{\pi_0 \times \tilde\pi_0}(\n) \neq 0}} a(\n) \lambda^\circ_{\pi \times \pi_0}(\n) \Bigg|^2
    \ll_{\epsilon} Q^\eps (x + D_F^{-\frac{n^2}{2}}Q^n T^{\frac{[F:\Q]n^2}{2}+\epsilon}|\mS|)\sum_{\N \n \in (x, e^{1/T}x]} |a(\n)|^2 \lambda_{\pi_0 \times \tilde\pi_0}(\n).
\]
To finish, note that if $\lambda_{\pi_0 \times \tilde\pi_0}(\kn) = 0$, then $\lambda^\circ_{\pi \times \pi_0}(\n) = 0$ (per \eqref{eqn:pointwise-RS-bound} with $\lambda^+ = \lambda$).  Therefore, the condition that $\lambda_{\pi_0\times\tilde{\pi}_0}(\kn)\neq 0$ is unnecessary.  Removing this condition, we arrive at \eqref{eqn:large-sieve-general}.

To deduce \eqref{eqn:large-sieve-general_2}, we take $T=z=1$.  We proceed as in the proof of \eqref{eqn:large-sieve-general}, except that we apply the second estimate in \cref{lem:rs-sums} with $\sigma = 1-2\epsilon^{-1}$. This choice of $\sigma$ and the assumption that $x \ge D_F^{-n^2} Q^{2n} |\mS|^\eps$ ensure that $
    |\mS| x^{\sigma}(D_F^{-n^2}Q^{2n})^{\frac{1}{2}-\sigma+\frac{\epsilon}{2n^2}}
    \ll_{\epsilon} Q^{\epsilon}x$.  This allows us to absorb the contribution of the off-diagonal terms with $\pi \neq \pi'$ into that of the diagonal terms.

To deduce \eqref{eqn:large-sieve-general-sifted}, we choose
\[
    b(\n) := 
    \frac{a(\n) \sqrt{\lambda_{\pi_0 \times \tilde\pi_0}(\n)}}{\sum_{\substack{\N \ka \in (x, e^{1/T}x] \\ \p \mid \ka \implies \N\p > z}} |a(\ka)|^2 \lambda_{\pi_0 \times \tilde\pi_0}(\ka)}
\]
when $\N\n \in (x, e^{1/T}x]$ and all prime factors $\p \mid \n$ have $\N \p > z$, and $b(\kn)=0$ otherwise.  If there exists an effectively computable constant $\Cr{large_sieve_z}=\Cr{large_sieve_z}(n,[F:\Q],\eps)\geq 1$ such that
\begin{equation}
\label{eqn:diagonal-rs-lower-bound}
    \sum_{\N \n \le z} \frac{\lambda_{\pi \times \tilde\pi}(\n)}{\N \n}
    \gg_{\eps}
    (\log z) \mathop{\Res}_{s = 1} L(s, \pi \times \tilde\pi),\qquad z \geq \Cr{large_sieve_z} D_F^{-n^2/2}Q^{n+\eps},
\end{equation}
then \eqref{eqn:large-sieve-general-sifted} would now follow from \eqref{eqn:general-large-sieve-sum}, \eqref{eqn:key-sieve-property}.  It remains to show \eqref{eqn:diagonal-rs-lower-bound}.  By inserting a smooth minorant and applying \cref{lem:rs-sums} with $T = 1$ and $\d = \mO_F$, we infer that if $y\geq 1$, then
\[
    \sum_{\N \n \in (y, ey]} \frac{\lambda_{\pi \times \tilde\pi}(\n)}{\N \n}
    \gg
    \mathop{\Res}_{s=1} L(s, \pi \times \tilde\pi) + O_\eps\Bigg(\frac{D_F^{-\frac{n^2}{2}}Q^{n+\frac{\eps}{3}}}{y}\Bigg).
\]
Decomposing the interval $(z^{1-\frac{\eps}{3n}}, z]$ into $\asymp_\eps \log z$ dyadic subintervals, we deduce that
\[
    \sum_{\N \n \le z} \frac{\lambda_{\pi \times \tilde\pi}(\n)}{\N \n}
    \gg_\eps
    1 + (\log z) \mathop{\Res}_{s=1} L(s, \pi \times \tilde\pi) + O_\eps\Bigg(\frac{D_F^{-\frac{n^2}{2}}Q^{n+\frac{\eps}{3}}}{z^{1-\frac{\eps}{3n}}} \log z\Bigg).
\]
For $z$ as in \eqref{eqn:diagonal-rs-lower-bound}, the error term has absolute value at most $\frac{1}{2}$.  This finishes our proof of \eqref{eqn:diagonal-rs-lower-bound}.
\end{proof}

\begin{corollary}
\label{cor:sifted-bound-single-rep}
Let $\eps > 0$, $n_0 \in \mathbb{N}$, $x\ge 1$, and $\pi_0\in\mathfrak{F}_{n_0}$.  If $T \ge 1$ and $z \geq \Cr{large_sieve_z} D_F^{-n_0^2/2}\mathfrak{C}_{\pi_0}^{n_0+\eps}$, then
\begin{equation}
\label{eqn:sifted-bound-single-rep}
    \sum_{\substack{\N \n \in (x, e^{1/T} x] \\ \p \mid \n \Rightarrow \N \p > z}}  \lambda_{\pi_0 \times \tilde{\pi}_0}(\n)
    \ll_{\epsilon}
    \frac{x}{T \log z} + D_F^{-\frac{n_0^2}{2}}\mathfrak{C}_{\pi_0}^{n_0+\eps} z^{2n_0^2+2+\eps} T^{\frac{[F:\Q]n_0^2}{2}+\epsilon}.
\end{equation}
\end{corollary}
\begin{proof}
We apply \eqref{eqn:large-sieve-general-sifted} with $\mS = \{\pi_0\}$, $\lambda^\circ = \lambda$, and $a(\n) = 1$.
\end{proof}

\section{Proofs of \cref{thm:density,cor:family}}
\label{sec:family_proofs}

Let $\kp$ be a prime ideal of $\mathcal{O}_F$, $n\geq 2$, $\mathcal{S}\subseteq \mathfrak{F}_n$ a finite subset, and $Q=\max_{\pi\in\mathcal{S}}\mathfrak{C}_{\pi}$.  For $\theta\geq 0$, we define $\mathcal{S}(\kp,\theta) = \{ \pi \in \mathcal{S} : \max_{1\leq j\leq n} |\alpha_{j,\pi}(\kp)| \ge \N \kp^\theta \}$.  Let
\begin{equation}
\label{eqn:family_defs}
\begin{gathered}
\mathcal{M}=\begin{cases}\lceil D_F^{-n^2}Q^{2n}|\mathcal{S}(\kp,\theta)|^{\epsilon}\rceil&\mbox{if $\theta\leq 1/4$,}\\
\lceil D_F^{-n^2/2}Q^{n}|\mathcal{S}(\kp,\theta)|^{1+\epsilon}\rceil&\mbox{if $\theta>1/4$,}
\end{cases}\qquad \N\kp\leq \mathcal{M}^{\frac{1}{n+1}},\qquad M = \Big\lfloor\frac{\log \mathcal{M}}{\log\N\kp}-n\Big\rfloor\geq 1.
\end{gathered}
\end{equation}
We use a lower bound for power sums, originally due to S{\'o}s and Tur{\'a}n \cite{SosTuran}.

\begin{proposition}[{\cite{KolesnikStraus}}]
\label{prop:Turan}
Let $M,N\in \mathbb{N}$.  For any $N$ complex numbers $z_1,\ldots,z_N$, there exists an integer $k\in[M+1,M+N]$ such that $|z_1^k+\cdots+z_N^k|\geq 1.007|z_1|^k (4e(1+M/N))^{-N}$. 
\end{proposition}
\begin{remark}
Makai \cite{Makai} showed that the constant $4e$ is essentially optimal.
\end{remark}

By \cref{prop:Turan} and our definition of $\mathcal{S}(\kp,\theta)$, we find that for all $\pi\in\mathcal{S}(\kp,\theta)$, there exists an integer $k_{\pi}\in[M+1,M+n]$ such that
\[
\Bigg|\sum_{j=1}^n \frac{\alpha_{j,\pi}(\kp)^{k_{\pi}}}{k_{\pi}}\Bigg|^2=\frac{1}{k_{\pi}^2}\Bigg|\sum_{j=1}^n \alpha_{j,\pi}(\kp)^{k_{\pi}}\Bigg|^2\geq \frac{1.007^2}{k_{\pi}^2}\Bigg|\frac{\displaystyle\max_{1\leq j\leq n}|\alpha_{j,\pi}(\kp)|^{k_{\pi}}}{(4e(1+\frac{M+n}{n}))^n}\Bigg|^2\gg \frac{\N\kp^{2k_{\pi} \theta}}{M^{2n+2}}\gg\frac{(\mathcal{M}\N\kp^{-n})^{2\theta}}{\mathcal{M}^{\epsilon}}.
\]
We now sum over $\pi\in\mathcal{S}(\kp,\theta)$, thus obtaining
\begin{equation}
\label{eqn:power_sum_alpha}
\frac{(\mathcal{M}\N\kp^{-n})^{2\theta}}{\mathcal{M}^{\epsilon}}|\mathcal{S}(\kp,\theta)|\ll\sum_{k=M+1}^{M+n}\,\sum_{\pi\in\mathcal{S}(\kp,\theta)}\Bigg|\sum_{j=1}^n \frac{\alpha_{j,\pi}(\kp)^{k}}{k}\Bigg|^2\ll \max_{k\in[M+1,M+n]}\sum_{\pi\in\mathcal{S}(\kp,\theta)}\Bigg|\sum_{j=1}^n \frac{\alpha_{j,\pi}(\kp)^{k}}{k}\Bigg|^2.
\hspace{-2.3mm}
\end{equation}

We estimate the sum on the far right-hand side using the large sieve inequalities in \cref{thm:large-sieve-general} with $\pi_0=\mathbbm{1}$, $L^\circ(s, \pi \times \tilde\pi') = \log L(s, \pi \times \tilde\pi')$ (which is permissible by \cref{prop:inequality-RS-covers}(i)), $x=\frac{1}{2}\N\kp^k$, and $a(\kn)$ the indicator function for the condition $\kn=\kp^k$.  Specifically, we apply \eqref{eqn:large-sieve-general} if $\theta>\frac{1}{4}$ and \eqref{eqn:large-sieve-general_2} if $\theta\leq \frac{1}{4}$.  Our choices of $\kp$, $k$, $x$, and $\mathcal{M}$ ensure that $x\asymp \mathcal{M}$ and
\begin{equation}
\label{eqn:power_sum_alpha_2}
\max_{k\in[M+1,M+n]}\sum_{\pi\in\mathcal{S}(\kp,\theta)}\Bigg|\sum_{j=1}^n \frac{\alpha_{j,\pi}(\kp)^{k}}{k}\Bigg|^2\ll_{\epsilon}Q^{\epsilon}\mathcal{M}.
\end{equation}
We infer from \eqref{eqn:power_sum_alpha} and \eqref{eqn:power_sum_alpha_2} that
\begin{equation}
\label{eqn:nice_S_bound_pre}
|\mathcal{S}(\kp,\theta)|\ll_{\epsilon}\N\kp^{2n\theta}\mathcal{M}^{1-2\theta+\epsilon}Q^{\epsilon},\qquad \text{ for } \N\kp\leq \mathcal{M}^{\frac{1}{n+1}}.
\end{equation}

\begin{proof}[Proof of \cref{cor:family}]
We apply \eqref{eqn:nice_S_bound_pre} to the case where $\mathcal{S}=\mathfrak{F}_n(Q)$ and $\theta=0$.  We thus obtain, for $0<\epsilon<1$ and $\kp$ a prime ideal of $\mathcal{O}_F$ such that $\N\kp\leq (D_F^{-n^2}Q^{2n})^{1/(n+1)}$, the bound
\[
|\mathcal{S}(\kp,0)|\ll_{\epsilon}(D_F^{-n^2}Q^{2n}|\mathcal{S}(\kp,0)|^{\epsilon})^{1+\epsilon}Q^{\epsilon}\leq (D_F^{-n^2}Q^{2n}|\mathfrak{F}_n(Q)|^{\epsilon})^{1+\epsilon}Q^{\epsilon}.
\]

Suppose to the contrary that $\pi\in\mathfrak{F}_n(Q)$ has satisfies $\#\{\kp\mid\kq_{\pi}\}=\#\{\kp\colon \N\kp\leq Q^{\epsilon}\}$.  Since $\N\kq_{\pi}\leq Q$, it follows from \cite[Lemma 7.3]{ThornerZhuo} and the bound $[F:\Q]\ll \log Q$ that $\#\{\kp\mid\kq_{\pi}\}\ll (\log Q)^2$.  However, it follows from work of Zaman (see \cite[Theorem 1.3.1]{Zaman_Thesis}) that there exists an effectively computable constant $\Cl[abcon]{family_prime}=\Cr{family_prime}(\epsilon)>0$ such that if $Q\geq D_F^{\Cr{family_prime}}$, then $\#\{\kp\colon \N\kp\leq Q^{\epsilon}\}\gg_{\epsilon} Q^{\epsilon/3}$.  This is a contradiction, so we must conclude that if $\epsilon>0$, $Q\geq  D_F^{\Cr{family_prime}}$, and $\pi\in\mathfrak{F}_n(Q)$, then there exists a prime ideal $\kp$ with $\N\kp\leq Q^{\epsilon}$ such that $\kp\nmid \kq_{\pi}$.

Our unitary normalization of $\pi\in\mathfrak{F}_n(Q)$ ensures that if $\kp\nmid\kq_{\pi}$, then $\max_{1\leq j\leq n}|\alpha_{j,\pi}(\kp)|\geq 1$.  Therefore, it follows from the preceding paragraph that if $Q\geq D_F^{\Cr{family_prime}}$, then for all $\pi\in\mathfrak{F}_n(Q)$, there exists a prime ideal $\kp$ with $\N\kp\leq Q^{\epsilon}$ such that $\max_{1\leq j\leq n}|\alpha_{j,\pi}(\kp)|\geq 1$.  Consequently, we have that
\begin{align*}
|\mathfrak{F}_n(Q)|&\leq \sum_{\N\kp\leq Q^{\epsilon}}|\mathcal{S}(\kp,0)|\ll_{\epsilon}(D_F^{-n^2}Q^{2n}|\mathfrak{F}_n(Q)|^{\epsilon})^{1+\epsilon}Q^{2\epsilon}.
\end{align*}
Replacing $\epsilon$ with $\epsilon/(7n)$ and solving for $|\mathfrak{F}_n(Q)|$, we prove the theorem.
\end{proof}

\begin{proof}[Proof of \cref{thm:density}]

Let $Q\geq D_F^{\Cr{family}}$.  Since $|\mathcal{S}(\kp,\theta)|\leq |\mathfrak{F}_n(Q)|$, it follows from \cref{cor:family}, \eqref{eqn:family_defs}, \eqref{eqn:nice_S_bound_pre}, and a rescaling of $\epsilon$ that if
\[
\mathscr{M}=\begin{cases}
	D_F^{-n^2}Q^{2n}&\mbox{if $\theta\leq 1/4$,}\\
	D_F^{-n^2/2}Q^n |\mathcal{S}(\kp,\theta)|&\mbox{if $\theta>1/4$,}
\end{cases}
\]
then $|\mathcal{S}(\kp,\theta)|\ll_{\epsilon}\N\kp^{2n\theta}\mathscr{M}^{1-2\theta}Q^{\epsilon}$ when $\N\kp\leq \mathscr{M}^{\frac{1}{n+1}}$, establishing \cref{thm:density} in this range of $\N\kp$.  (One must solve for $|\mathcal{S}(\kp,\theta)|$ when $\theta>\frac{1}{4}$.)  If $\N\kp>\mathscr{M}^{1/(n+1)}$, then \cref{thm:density} is trivial compared to \cref{cor:family}.
\end{proof}

\section{Proof of \cref{thm:ZDE}}
\label{sec:zero-density}

When the bounds in \cref{thm:large-sieve-general} are combined with Plancherel's theorem, we can deduce several ``hybrid'' large sieve inequalities.  Here are a few examples.

\begin{corollary}
\label{cor:MVT_mu}
	Let $\pi_0\in\mathfrak{F}_{n_0}$, $Q,T\geq 1$, and $\epsilon>0$.  If $X\geq D_F^{-\frac{n^2}{2}}Q^{n+\epsilon}T^{\frac{n^2 [F:\Q]}{2}+1+\epsilon}|\mathcal{S}|$, $Y\geq e$, and $\log Y\asymp_{\epsilon} \log X$, then
	{\small\begin{align*}
    \begin{gathered}
    \sum_{\pi\in\mathcal{S}}\int_{-T}^{T}\Bigg|\sum_{\substack{\N\kn > X}}\frac{\mu_{\pi}(\kn)}{\N\kn^{1+\frac{1}{\log Y}+iv}}\Bigg|^2 dv \ll_{\epsilon} Q^{\epsilon}T^{\epsilon}\log X,\quad \sum_{\pi\in\mathcal{S}}\int_{-T}^{T}\Bigg|\sum_{\substack{\N\kn > X}}\frac{\mu_{\pi\times\pi_0}(\kn)}{\N\kn^{1+\frac{1}{\log Y}+iv}}\Bigg|^2 dv \ll_{\epsilon} \mathfrak{C}_{\pi_0}^{\epsilon}Q^{\epsilon}T^{\epsilon}\log X,\\
	\sum_{\pi\in\mathcal{S}}\int_{-T}^{T}\Bigg|\sum_{\substack{\N\kn\leq X}}\frac{\mu_{\pi}(\kn)}{\N\kn^{\frac{1}{2}+iv}}\Bigg|^2 dv\ll_{\epsilon} Q^{\epsilon}T^{\epsilon}X\log X,\quad 
	\sum_{\pi\in\mathcal{S}}\int_{-T}^{T}\Bigg|\sum_{\substack{\N\kn\leq X}}\frac{\mu_{\pi\times\pi_0}(\kn)}{\N\kn^{\frac{1}{2}+iv}}\Bigg|^2 dv\ll_{\epsilon} \mathfrak{C}_{\pi_0}^{\epsilon}Q^{\epsilon}T^{\epsilon}X\log X.
    \end{gathered}
	\end{align*}}%
\end{corollary}
\begin{proof}
The proof is the same as that of \cite[Corollary 4.4]{humphries2024zeros}, with \cref{thm:large-sieve-general} in the role of \cite[Theorem 4.1]{humphries2024zeros}.  For the first two estimates, we apply \cref{thm:large-sieve-general} with $\pi_0=\mathbbm{1}$ and $\lambda_{\pi}^{\circ}(\kn)=\mu_{\pi}(\kn)$.  For the second two estimates, we apply \cref{thm:large-sieve-general} with $\pi_0$ arbitrary and $\lambda_{\pi\times\pi_0}^{\circ}(\kn)=\mu_{\pi\times\pi_0}(\kn)$.
\end{proof}

\begin{proof}[Proof of \cref{thm:ZDE}]
We prove the second zero density estimate in \cref{thm:ZDE}, since the first is a special case ($\pi_0=\mathbbm{1}$).  We summarize as follows.  Using the ideas of Montgomery \cite{Monty}, one proceeds as in the proof of \cite[Theorem 1.1]{humphries2024zeros} (without the coprimality conditions), concluding that if
\[
X=D_F^{-\frac{n^2}{2}}Q^n T^{\frac{n^2[F:\Q]}{2}+1}|\mathcal{S}|(\mathfrak{C}_{\pi_0}QT)^{\epsilon},\quad Y=\Big(D_F^{-\frac{n_0 n}{2}}\mathfrak{C}_{\pi_0}^{\frac{n}{2}}Q^{\frac{n_0}{2}}T^{\frac{n_0 n [F:\Q]}{2}+1}|\mathcal{S}|X(\mathfrak{C}_{\pi_0}QT)^{\epsilon}\Big)^{\frac{1}{3-2\sigma}},
\]
then 
\begin{multline*}
\sum_{\pi\in\mathcal{S}}N_{\pi\times\pi_0}(\sigma,T)\ll_{\epsilon} X^{\epsilon} \Big[Y^{2(1-\sigma)}\int_{-T}^{T}\sum_{\pi\in\mathcal{S}}\Bigg|\sum_{\N\kn>X}\frac{\mu_{\pi\times\pi_0}(\kn)}{\N\kn^{1+\frac{1}{\log Y}+iv}}\Bigg|^2 dv\\
+ Y^{\frac{1}{2}-\sigma}\Bigg(\int_{-T}^{T}\sum_{\pi\in\mathcal{S}}|L(\tfrac{1}{2}+iv,\pi\times\pi_0)|^2 dv\Bigg)^{\frac{1}{2}}\Bigg(\int_{-T}^{T}\sum_{\pi\in\mathcal{S}}\Bigg|\sum_{\N\kn\leq X}\frac{\mu_{\pi\times\pi_0}(\kn)}{\N\kn^{\frac{1}{2}+iv}}\Bigg|^2 dv\Bigg)^{\frac{1}{2}}+ Y^{1-\sigma}\Big].
\end{multline*}
It should be noted that one of the intermediate steps in the proof of \cite[Theorem 1.1]{humphries2024zeros} uses the bound in \cite[Lemma 3.3]{humphries2024zeros}:  If $\gcd(\kn,\kq_{\pi}\kq_{\pi_0})=\mathcal{O}_F$, then $2|\mu_{\pi\times\pi_0}(\kn)|\leq \lambda_{\pi\times\tilde{\pi}}(\kn)+\lambda_{\pi_0\times\tilde{\pi}_0}(\kn)$.  Fortunately, per \eqref{eqn:mu_bound}, this bound holds independently of the coprimality condition.

We estimate the two $v$-integrals involving $\mu_{\pi\times\pi_0}(\kn)$ using \cref{cor:MVT_mu}.  The remaining integral is estimated trivially using \cref{lem:Li1}.  These estimates readily yield \cref{thm:ZDE}.
\end{proof}

\section{Proofs of Corollaries \ref{cor:large_sieve2} and \ref{thm:2ndMoment}}
\label{sec:holder}

We denote by $\mJ_F$ the set of nonzero integral ideals of $F$. To avoid confusion with $k$-tuples $(\kn_1, \ldots, \kn_k) \in \mJ_F^k$, we write $\gcd(\kn_1, \ldots, \kn_k)$ and $\lcm(\kn_1, \ldots, \kn_k)$ for the greatest common divisor and lowest common multiple of these ideals, respectively. We begin with the following lemma.

\begin{lemma}
\label{lem:decompose-ideals}
    For each positive integer $k$, there exists a map $ \mathscr{M}\colon \mJ_F^k \to \mJ_F^k$ such that for all $(\kn_1, \ldots, \kn_k) \in \mJ_F^k$, the tuple $(\km_1, \ldots, \km_k) := \mathscr{M}(\kn_1, \ldots, \kn_k)$ satisfies the following:
    \begin{itemize}
        \item[$(i)$] For all $1 \le j \le k$, one has $\km_j \mid \kn_j$ and $\gcd(\km_j, \km_j^{-1} \kn_j) = \mO_F$;
        \item[$(ii)$] $\km_1 \cdots \km_k = \lcm(\kn_1, \ldots, \kn_k)$, and the ideals $(\km_j)_{1 \le j \le k}$ are pairwise relatively prime.
    \end{itemize}
\end{lemma}

\begin{proof}
Let $(\kn_1, \ldots, \kn_k) \in \mJ_F^k$.  There exists an integer $L\geq 1$ and positive integers $a_1, \ldots, a_L$ such that $\kp_1^{a_1} \cdots \kp_L^{a_L}$ is the factorization of $\lcm(\kn_1, \ldots, \kn_k)$.  We construct the tuple $(\km_1, \ldots, \km_k) = \mathscr{M}(\kn_1, \ldots, \kn_k)$ by a greedy process: 
\begin{itemize} 
\item[1.] Initially, set $\km_1 = \cdots = \km_k := \mO_F$. 
\item[2.] At step $\ell \in \{1, \ldots, L\}$, append $\kp_\ell^{a_\ell}$ to some arbitrarily-chosen $\km_j$ (i.e., set $\km_j \mapsto \km_j \kp_\ell^{a_\ell}$) such that $\kp_\ell^{a_\ell} \mid \kn_j$. Such a $j$ must exist since $\kp_\ell^{a_\ell} \mid \lcm(\kn_1, \ldots, \kn_k)$.
\end{itemize}
Throughout the process, for all $1 \le j \le k$, we have $\km_j \mid \kn_j$. In fact, $\km_j$ is a product of maximal prime powers dividing $\kn_j$ (i.e., $\kp^a \| \kn_j$), so $(\km_j, \km_j^{-1}\kn_j) = \mO_F$.  At the end of the process, we also have $\km_1 \cdots \km_k = \kp_1^{a_1}\cdots \kp_L^{a_L} = \lcm(\kn_1, \ldots, \kn_k)$. Moreover, each prime power $\kp_\ell^{a_\ell}$ appears as a factor of exactly one of $\km_1, \ldots, \km_k$, so the ideals $\km_1,\ldots,\km_k$ are pairwise relatively prime.
\end{proof}


\begin{proof}[Proof of \cref{cor:large_sieve2}]
Let $\epsilon>0$.  We prove the result for $C^\infty(N, \mS, \pi_0)$; the bound for $C^\infty(N, \mS)$ then follows by taking $\pi_0 = \one$. We need to show that if $a(\kn)$ is a function satisfying $N^{1/2} \|a\|_{\infty,\pi_0} = 1$, then
\[
    \sum_{\pi\in\mathcal{S}}\Bigg|\sum_{\N\kn\leq N}\lambda_{\pi\times\pi_0}(\kn)a(\kn)\Bigg|^2
    \ll_{k,\epsilon} (NQ)^{\epsilon} (N|\mS|^{\frac{k-1}{k}} + Q^{\frac{n}{k}}|\mS|).
\]
If $\lambda_{\pi_0\times\tilde{\pi}_0}(\kn)=0$, then by \eqref{eqn:pointwise-RS-bound}, we have that $\lambda_{\pi\times\pi_0}(\kn)=0$.  Therefore, we can assume without loss of generality that $a(\kn)$ is supported on ideals $\kn$ with $\lambda_{\pi_0 \times \tilde\pi_0}(\kn) \neq 0$ and $\N \kn \le N$. For $\pi \in \mS$, we denote
\begin{equation} \label{eqn:twist-pi0}
    \lambda^{\pi_0}_\pi(\kn) := 
    \begin{cases}
        0, &\lambda_{\pi_0 \times \tilde\pi_0}(\kn) = 0, \\ 
        \lambda_{\pi_0 \times \tilde\pi_0}(\kn)^{-1/2} \lambda_{\pi \times \pi_0}(\kn), & 
        \text{otherwise.}
    \end{cases}
\end{equation}
Note that this is a multiplicative function of $\kn$.  By changing variables $\lambda_{\pi_0 \times \tilde\pi_0}(\kn)^{1/2} a(\kn) \mapsto a(\kn)$, it remains to show that
\begin{equation} \label{eqn:holder-cor-to-show}
    \sum_{\pi\in\mathcal{S}}\Bigg|\sum_\kn \lambda^{\pi_0}_\pi(\kn)a(\kn)\Bigg|^2 \ll_{k,\epsilon} (NQ)^{\epsilon} (N|\mS|^{\frac{k-1}{k}} + Q^{\frac{n}{k}}|\mS|),
\end{equation}
whenever $a(\kn)$ is supported on $\N \kn \le N$ and $\max_\kn |a(\kn)| = N^{-1/2}$. By H\"older's inequality, we have
\begin{equation} \label{eqn:holder-step}
\begin{aligned}
    \sum_{\pi\in\mathcal{S}} 1 \cdot \Bigg|\sum_{\kn}\lambda^{\pi_0}_{\pi}(\kn)a(\kn)\Bigg|^2
    &\le 
    |\mS|^{\frac{k-1}{k}}
    \left( \sum_{\pi\in\mathcal{S}} \Bigg|\sum_{\kn}\lambda^{\pi_0}_{\pi}(\kn)a(\kn)\Bigg|^{2k} \right)^{\frac{1}{k}}
    \\
    &= 
    |\mS|^{\frac{k-1}{k}}
    \left( \sum_{\pi\in\mathcal{S}} \Bigg|\sum_{\kn_1, \ldots, \kn_k}\lambda^{\pi_0}_{\pi}(\kn_1) \cdots \lambda^{\pi_0}_{\pi}(\kn_k) a(\kn_1) \cdots a(\kn_k)\Bigg|^2 \right)^{\frac{1}{k}}.
\end{aligned}
\end{equation}

Let $\mathscr{M}$ be as in \cref{lem:decompose-ideals}.  Given $(\kn_1,\ldots,\kn_k)\in \mJ_F^k$, write $(\km_1,\ldots,\km_k)=\mathscr{M}(\kn_1,\ldots,\kn_k)$ and $(\kd_1,\ldots,\kd_k)=(\km_1^{-1}\kn_1,\ldots,\km_k^{-1}\kn_k)$.  We now have
\begin{equation}
\label{eqn:expanded_after_M}
\begin{aligned}
   & \sum_{\kn_1, \ldots, \kn_k} \lambda^{\pi_0}_{\pi}(\kn_1) \cdots \lambda^{\pi_0}_{\pi}(\kn_k) a(\kn_1) \cdots a(\kn_k)\\
&= \sum_{\substack{\kd_1, \ldots, \kd_k \\ \km_1, \ldots, \km_k \\  \mathscr{M}(\kd_1\km_1, \ldots, \kd_k\km_k) = (\km_1, \ldots, \km_k)}}
    \lambda^{\pi_0}_{\pi}(\kd_1\km_1) \cdots \lambda^{\pi_0}_{\pi}(\kd_k\km_k) a(\kd_1\km_1) \cdots a(\kd_k\km_k).
\end{aligned}
\end{equation}
Indeed, each tuple $(\kn_1, \ldots, \kn_k) \in \mJ_F^k$ appears exactly once as $(\kd_1\km_1, \ldots, \kd_k\km_k)$ in the final sum, since $\km_1, \ldots, \km_k$ are uniquely determined by $\mathscr{M}$. By \cref{lem:decompose-ideals} and the multiplicativity of $\lambda^{\pi_0}_\pi$, we have that
\[
\lambda^{\pi_0}_{\pi}(\kd_1\km_1) \cdots \lambda^{\pi_0}_{\pi}(\kd_k\km_k) = \lambda^{\pi_0}_\pi(\kd_1) \cdots \lambda^{\pi_0}_\pi(\kd_k) \lambda^{\pi_0}_\pi(\km_1\cdots\km_k).
\]
If we define
\begin{equation} \label{eqn:ak-def}
    a_{\kd_1,\ldots,\kd_k}(\kc) := \sum_{\substack{\km_1 \cdots \km_k = \kc \\ \mathscr{M}(\kd_1\km_1, \ldots, \kd_k\km_k) = (\km_1, \ldots, \km_k)}} a(\kd_1\km_1) \cdots a(\kd_k\km_k),
\end{equation}
then \eqref{eqn:expanded_after_M} equals
\[
\sum_{\kd_1, \ldots, \kd_k}
    \lambda^{\pi_0}_\pi(\kd_1) \cdots \lambda^{\pi_0}_\pi(\kd_k)
    \sum_{\N\kc \le N^k}
    \lambda^{\pi_0}_\pi(\kc) 
    a_{\kd_1,\ldots,\kd_k}(\kc).
\]
Taking a square and applying Cauchy--Schwarz in the sum over $\kd_1, \ldots, \kd_k$, we deduce that
\begin{equation} \label{eqn:squared-holder-sum}
\begin{aligned}
    &\Big\vert \sum_{\kn_1, \ldots, \kn_k} \lambda^{\pi_0}_{\pi}(\kn_1) \cdots \lambda^{\pi_0}_{\pi}(\kn_k) a(\kn_1) \cdots a(\kn_k) \Big\vert^2
    \\
    =\
    &\Big\vert \sum_{\kd_1, \ldots, \kd_k}
    \lambda^{\pi_0}_\pi(\kd_1) \cdots \lambda^{\pi_0}_\pi(\kd_k)
    \sum_{\N\kc \le N^k}
    \lambda^{\pi_0}_\pi(\kc) 
    a_{\kd_1,\ldots,\kd_k}(\kc) \Big\vert^2 
    \\
    \le\ 
    &\left(\sum_{\N \kd_1, \ldots, \N \kd_k \le N}
    \frac{|\lambda^{\pi_0}_\pi(\kd_1) \cdots \lambda^{\pi_0}_\pi(\kd_k)|^2}{\N\kd_1 \cdots \N\kd_k}
    \right)
    \left( 
    \sum_{\kd_1, \ldots, \kd_k} \N\kd_1 \cdots \N\kd_k
    \Big\vert 
    \sum_{\N\kc \le N^k}
    \lambda^{\pi_0}_\pi(\kc) 
    a_{\kd_1,\ldots,\kd_k}(\kc) \Big\vert^2 \right).
\end{aligned}
\end{equation}

Let us analyze the first sum.
By \eqref{eqn:pointwise-RS-bound} (with $\lambda^{\circ}=\lambda^+=\lambda$, $\pi_1=\pi$, and $\pi_2=\pi_0$), we have $|\lambda_\pi^{\pi_0}(\kd)|^2 \le \lambda_{\pi \times \tilde\pi}(\kd)$. Using this and \cref{lem:mertens}, we obtain 
\[
    \sum_{\N \kd_1, \ldots, \N \kd_k \le N}
    \frac{|\lambda^{\pi_0}_\pi(\kd_1) \cdots \lambda^{\pi_0}_\pi(\kd_k)|^2}{\N\kd_1 \cdots \N\kd_k}
    =
    \Bigg( \sum_{\N\kd \le N} \frac{|\lambda_\pi^{\pi_0}(\kd)|^2}{\N\kd} \Bigg)^k
    \le
    \Bigg( \sum_{\N\kd \le N} \frac{\lambda_{\pi \times \tilde\pi}(\kd)}{\N\kd} \Bigg)^k
    \ll_{k,\epsilon} (NQ)^{\epsilon}.
\]
Combining this with \eqref{eqn:holder-step} and \eqref{eqn:squared-holder-sum}, we reach
\[
\sum_{\pi \in \mS} \Big\vert \sum_\kn \lambda_\pi^{\pi_0}(\kn) a(\kn) \Big\vert^2
    \ll_{k,\epsilon} (NQ)^{\epsilon}
    |\mS|^{\frac{k-1}{k}}
    \Bigg(
    \sum_{\pi \in \mS}
    \sum_{\kd_1, \ldots, \kd_k} \N\kd_1 \cdots \N\kd_k
    \Big\vert 
    \sum_{\N\kc \le N^k}
    \lambda^{\pi_0}_\pi(\kc) 
    a_{\kd_1,\ldots,\kd_k}(\kc) \Big\vert^2
    \Bigg)^{\frac{1}{k}}.
\]
Recalling \eqref{eqn:twist-pi0} and then applying \cref{thm:large_sieve2} for the sequence $a_{\kd_1,\ldots,\kd_k}(\kc)$, we conclude that
\[
\begin{aligned}
    &\sum_{\pi \in \mS} \Big\vert \sum_\kn \lambda_\pi^{\pi_0}(\kn) a(\kn) \Big\vert^2\\
    &\ll_{k,\epsilon} (NQ)^{\epsilon} |\mS|^{\frac{k-1}{k}} \Bigg(\sum_{\kd_1,\ldots,\kd_k} \N\kd_1 \cdots \N\kd_k \sum_{\pi \in \mS} \Big\vert \sum_{\substack{\N\kc \le N^k \\ \lambda_{\pi_0 \times \tilde\pi_0}(\kc) \neq 0}} \lambda_{\pi \times \pi_0}(\kc) \frac{a_{\kd_1,\ldots,\kd_k}(\kc)}{\sqrt{\lambda_{\pi_0 \times \tilde\pi_0}(\kc)}} \Big\vert^2
    \Bigg)^{\frac{1}{k}}
    \\
    &\ll_{k,\epsilon} (NQ)^{\epsilon} |\mS|^{\frac{k-1}{k}} \Bigg(\sum_{\kd_1, \ldots, \kd_k} \N\kd_1 \cdots \N\kd_k (NQ)^{\epsilon} (N^k + Q^n|\mS|) \sum_{\substack{\N\kc\leq N^k \\ \lambda_{\pi_0\times\tilde{\pi}_0}(\kc)\neq 0}}\lambda_{\pi_0\times\tilde{\pi}_0}(\kc)\Bigg|\frac{a_{\kd_1,\ldots,\kd_k}(\kc)}{\sqrt{\lambda_{\pi_0 \times \tilde\pi_0}(\kc)}}\Bigg|^2 \Bigg)^{\frac{1}{k}}
    \\
    &\ll_{k,\epsilon} (NQ)^{\epsilon}
    (N|\mS|^{\frac{k-1}{k}} + Q^{\frac{n}{k}}|\mS|)
    \Bigg(
    \sum_{\kd_1, \ldots, \kd_k} \N\kd_1 \cdots \N\kd_k
    \|a_{\kd_1,\ldots,\kd_k}\|_2^2
    \Bigg)^{\frac{1}{k}}.
\end{aligned}
\]
To establish \eqref{eqn:holder-cor-to-show}, it remains to bound the sum over $\kd_1, \ldots, \kd_k$ by $N^{\epsilon}$. From \eqref{eqn:ak-def}, Cauchy--Schwarz, the divisor bound $\sum_{\km_1\cdots \km_k=\kc} 1 \ll_{k,\epsilon} \N\kc^{\epsilon}$ (see \cite{Weiss}), we have
\[
\begin{aligned}
    \|a_{\kd_1,\ldots,\kd_k}\|_2^2
    &=
    \sum_{\N \kc \le N^k} \Bigg\vert \sum_{\substack{\km_1 \cdots \km_k = \kc \\ \mathscr{M}(\kd_1\km_1, \ldots, \kd_k\km_k) = (\km_1, \ldots, \km_k)}} a(\kd_1\km_1) \cdots a(\kd_k\km_k) \Bigg\vert^2
    \\
    &\ll_{k,\epsilon} (NQ)^{\epsilon}
    \sum_{\N \kc \le N^k} \sum_{\substack{\km_1 \cdots \km_k = \kc \\ \mathscr{M}(\kd_1\km_1, \ldots, \kd_k\km_k) = (\km_1, \ldots, \km_k)}} |a(\kd_1\km_1) \cdots a(\kd_k\km_k)|^2.
\end{aligned}
\]
Summing over $\kd_1, \ldots, \kd_k$ in a choice of dyadic ranges of sizes $D_1, \ldots, D_k \le N$, and using that $a(\kn)$ is supported on $\N\kn \le N$ and uniformly bounded by $N^{-1/2}$, we deduce that
\[
\begin{aligned}
    \sum_{\substack{D_1 < \N\kd_1 \le 2D_1 \\ \vdots \\ D_k < \N \kd_k \le 2D_k}} \N\kd_1 \cdots \N\kd_k \|a_{\kd_1,\ldots,\kd_k}\|_2^2
    &\ll_{k,\epsilon} 
    \frac{D_1\cdots D_k}{N^{k-\epsilon}}
    \sum_{\N\m_1 \le \frac{N}{D_1}, \ldots, \N\m_k \le \frac{N}{D_k}} \sum_{\substack{D_1 < \N\kd_1 \le 2D_1 \\ \vdots \\ D_k < \N \kd_k \le 2D_k \\ \mathscr{M}(\kd_1\km_1, \ldots, \kd_k\km_k) = (\km_1, \ldots, \km_k)}} 1.
\end{aligned}
\]
We recall from \cref{lem:decompose-ideals} that if $\mathscr{M}(\kd_1\km_1, \ldots, \kd_k\km_k) = (\km_1, \ldots, \km_k)$, then $\lcm(\kd_1\km_1, \ldots, \kd_k\km_k) = \km_1 \cdots \km_k$. This implies that $\kd_j \mid \km_1 \cdots \km_k$ for all $1 \le j \le k$. Using the divisor bound once again, as well as the bound $\sum_{\N\km \le M} 1 \ll_{\epsilon} M^{1+\epsilon}$ (see \cite{Weiss}), we reach
\[
    \sum_{\substack{D_1 < \N\kd_1 \le 2D_1 \\ \vdots \\ D_k < \N \kd_k \le 2D_k}} \N\kd_1 \cdots \N\kd_k \|a_{\kd_1,\ldots,\kd_k}\|_2^2
    \ll_{k,\epsilon} \frac{D_1\cdots D_k}{N^{k-\epsilon}}
    \sum_{\N\m_1 \le \frac{N}{D_1}, \ldots, \N\m_k \le \frac{N}{D_k}} 1
    \ll_{k,\epsilon} N^{\epsilon}.
\]
Summing over dyadic ranges, this establishes the desired bound \eqref{eqn:holder-cor-to-show}.
\end{proof}

\begin{proof}[Sketch of proof of \cref{thm:2ndMoment}]
Let $\mathcal{S}\subseteq\mathfrak{F}_n$ be a finite set, $Q=\max_{\pi\in\mathcal{S}}\mathfrak{C}_{\pi}$, and $t\in\R$.  Write $\mathcal{Q}=Q(|t|+3)^{n[F:\Q]}$.  It follows from the approximate functional equation \cite{HarcosAFE1,HarcosAFE2} that for all $\epsilon>0$ and $\pi \in \F_n$, there is an effectively computable constant $\Cl[abcon]{AFE}=\Cr{AFE}(n,[F:\Q],\epsilon)>0$ such that
\[
|L(\tfrac{1}{2}+it,\pi)|^2 \ll_{\epsilon}\mathcal{Q}^{\epsilon}\int_{-\Cr{AFE}\mathcal{Q}^{\epsilon}}^{\Cr{AFE}\mathcal{Q}^{\epsilon}}\Bigg|\sum_{\N\kn\leq \Cr{AFE} \mathcal{Q}^{1/2+\epsilon}}\frac{\lambda_{\pi}(\kn)}{\N\kn^{1/2+\epsilon+it+i\tau}}\Bigg|^2 d\tau+\mathcal{Q}^{-1}.
\]
See \cite[Lemma 3.4]{jana2020applications} for details when $t=0$; a very similar approach also handles $t\neq 0$. Thus
\[
\sum_{\pi\in\mathcal{S}}|L(\tfrac{1}{2}+it,\pi)|^2\ll_{\epsilon} \mathcal{Q}^{3\epsilon}|\mathcal{S}|\sup_{\substack{R\leq \Cr{AFE}\mathcal{Q}^{1/2+\epsilon} \\ \tau\in[-\Cr{AFE}\mathcal{Q}^{\epsilon},\Cr{AFE}\mathcal{Q}^{\epsilon}]}}\frac{1}{|\mathcal{S}|}\sum_{\pi\in\mathcal{S}}\Bigg|\sum_{R<\N\kn\leq 2R}\frac{\lambda_{\pi}(\kn)}{\N\kn^{1/2+\epsilon+it+i\tau}}\Bigg|^2+\mathcal{Q}^{-1}.
\]
Let $\delta=\delta_{\mathcal{S}}\geq 0$ be such that $|\mathcal{S}|\gg_{F}Q^{\delta}$.  We apply \cref{cor:large_sieve2} in the form 
\[
    \sum_{\pi \in \mS} \left\vert \sum_{N < \N\kn \le 2N} \lambda_\pi(\kn) a(\kn) \right\vert^2 \ll_{k,\epsilon} (NQ)^{\epsilon} (N|\mS|^{\frac{k-1}{k}} + Q^{\frac{n}{k}} |\mS|)\, N\max_{N<\N\kn \le 2N} |a(\kn)|^2,
\]
with $N=R\leq \mathcal{Q}^{1/2+\epsilon}$, $a(\kn)=\N\kn^{-1/2-\epsilon-it-i\tau}$, and $k=\lceil 2(n+\delta)\rceil$, thus obtaining
\[
\sum_{\pi\in\mathcal{S}}|L(\tfrac{1}{2}+it,\pi)|^2 \ll_{k,\epsilon} 
\mathcal{Q}^{6\epsilon}|\mathcal{S}|(\mathcal{Q}^{\frac{1}{2}}|\mathcal{S}|^{-\frac{1}{k}}+Q^{\frac{n}{k}})\ll_F |\mathcal{S}| Q^{\frac{1}{2}-\frac{\delta}{k}+6\epsilon}(|t|+3)^{\frac{n[F:\Q]}{2}},
\]
as desired.
\end{proof}

\section{Proof of \cref{thm:LFZDE}}
\label{sec:LFZDE}

\subsection{Bounds for the logarithmic derivative}
The key new ingredient in our proof of \cref{thm:LFZDE} is the use of \eqref{eqn:large-sieve-general-sifted} to deduce a large-sieve-type bound for the Dirichlet coefficients $\Lambda_{\pi\times\pi_0}(\kn)$ as $\kn$ ranges over all ideals of $\mathcal{O}_F$ (instead of restricting to prime ideals).  This crucially uses \cref{prop:inequality-RS-covers}, which enables us to prove a large sieve inequality for the Dirichlet coefficients of $\log L(s,\pi\times\pi_0)$ as $\pi\in\mathcal{S}$ varies.  Our use of Selberg's sieve in \eqref{eqn:large-sieve-general-sifted} saves us a crucial logarithm.

\begin{corollary}
\label{cor:ls-log-derivative}
Recall the notation and hypotheses of  \cref{thm:large-sieve-general}.  Define $\tilde n = \max(n, n_0)$, $\tilde Q = \max(Q, \mathfrak{C}_{\pi_0})$, and $\|a\|_\infty = \max_{\N \n \in (x, e^{1/T} x]} |a(\n)|$.  There exists an effectively computable constant $\Cl[abcon]{BT_average}=\Cr{BT_average}(n,n_0,[F:\Q],\epsilon)\geq 1$ such that if $Q\geq D_F^{\Cr{family}}$ and $x \geq \Cr{BT_average} (\tilde{Q}T^{[F:\Q]})^{2\tilde{n}^3 + 6\tilde{n}^2}$, then
\[
\begin{aligned}
    \sum_{\pi \in \mS}
    \Bigg| \sum_{\N \n \in (x, e^{1/T} x]} a(\n) \Lambda_{\pi \times \pi_0}(\n) \Bigg|^2
    \ll 
    \frac{x^2}{T^2} \|a\|_\infty^2.
\end{aligned}
\]
\end{corollary}

\begin{proof} 
In this proof, we denote by $\lambda^\circ_{\pi \times \tilde\pi'}(\n)$ the Dirichlet coefficients of $\log L(s, \pi \times \tilde\pi')$, which have $\lambda_{\pi \times \tilde\pi'}(\n)$ as a positive semi-definite cover by \cref{prop:inequality-RS-covers}(i). Both $\lambda^\circ_{\pi \times \pi_0}$ and $\Lambda_{\pi \times \pi_0}$ are supported on prime powers $\n = \p^k$, and they are related by $\Lambda_{\pi \times \pi_0}(\n) = \lambda^\circ_{\pi \times \pi_0}(\n) \log \N \n$.  Let $k_0 > 0$ and $\eps > 0$ be parameters to be chosen shortly (where $k_0$ depends only on $n$ and $n_0$, and $\eps$ is absolute).  By the Cauchy--Schwarz inequality, it suffices to separately estimate the two sums
{\small\[
    \mathfrak{L}_{\text{small}} := \sum_{\pi \in \mS}
    \Bigg| \sum_{\substack{\N \p^k \in (x, e^{1/T} x] \\ k \le k_0}} a(\p^k) \lambda^\circ_{\pi \times \pi_0}(\p^k) k \log \N \p \Bigg|^2,
    \ \ \ 
    \mathfrak{L}_{\text{large}} := \sum_{\pi \in \mS}
    \Bigg| \sum_{\substack{\N \p^k \in (x, e^{1/T} x] \\ k > k_0}} a(\p^k) \lambda^\circ_{\pi \times \pi_0}(\p^k) k \log \N \p \Bigg|^2.
\]}%

We bound $\mathfrak{L}_{\text{small}}$ using \cref{thm:large-sieve-general}, with $z := x^{1/k_0}$.  Instead of using \cref{thm:large-sieve-general} with $a(\kn)$, we will use \cref{thm:large-sieve-general} with $b(\kn)=
    a(\kn)\log\N\kn$ when $\kn=\kp^k$ with $k\leq k_0$ and $b(\kn)=0$ otherwise.  Note that $\N \p^k \in (x, e^{1/T} x]$ and $k \le k_0$ force $\N \p > z$. Therefore, if
\begin{equation}
\label{eqn:required-lower-bound-1}
    x \geq \max\Big\{\Cr{large_sieve_z}(n,[F:\Q],\epsilon)D_F^{-\frac{n^2}{2}}Q^{n+\eps},\Cr{large_sieve_z}(n_0,[F:\Q],\epsilon)D_F^{-\frac{n^2}{2}}\mathfrak{C}_{\pi_0}^{n_0+\eps}\Big\}^{k_0},
\end{equation}
then \eqref{eqn:large-sieve-general-sifted} and \cref{cor:sifted-bound-single-rep} imply that $\mathfrak{L}_{\text{small}}$ equals

\begin{equation}
\label{eqn:pre-sifted}
\begin{aligned}
&\sum_{\pi \in \mS}
    \Bigg| \sum_{\substack{\N \n \in (x, e^{1/T} x] \\ \p \mid \n \Rightarrow \N \p > z}} b(\n) \lambda^\circ_{\pi \times \pi_0}(\n) \Bigg|^2
    \\
    &\ll_{\epsilon}
 \Bigg(\frac{x}{T \log x} + D_F^{-\frac{n^2}{2}}Q^{n+\eps} x^{\frac{2n^2+2+\eps}{k_0}} T^{\frac{[F:\Q]n^2}{2}+\epsilon} 
     |\mS|\Bigg) \sum_{\substack{\N \n \in (x, e^{1/T} x] \\ \p \mid \n \Rightarrow \N \p > z}} |b(\n)|^2 \lambda_{\pi_0 \times \tilde\pi_0}(\n)
    \\
    &\ll_{\epsilon} 
  \Big(\frac{x}{T} + D_F^{-\frac{n^2}{2}}Q^{n+\eps} x^{\frac{2n^2+2+2\eps}{k_0}} T^{\frac{[F:\Q]n^2}{2}+\epsilon} 
     |\mS|\Big) \Big(\frac{x}{T} + D_F^{-\frac{n^2}{2}}\mathfrak{C}_{\pi_0}^{n_0+\eps} x^{\frac{2n_0^2+2+2\eps}{k_0}} T^{\frac{[F:\Q]n_0^2}{2}+\epsilon}\Big) \|a\|_\infty^2,
\end{aligned}
\end{equation}
where we used $\|b\|_\infty \ll \|a\|_\infty \log x$. Since $|\mS| \le |\F_n(Q)| \ll_{\epsilon} D_F^{-n^2}Q^{2n+\epsilon}$ per \cref{cor:family}, the final line of \eqref{eqn:pre-sifted} is $\ll_{\epsilon} x^2 T^{-2} \|a\|_\infty^2$, provided that
\begin{equation}
\label{eqn:required-lower-bound-2}
    \begin{cases} 
    x^{1-\frac{2n^2+2+2\eps}{k_0}} \gg_{\eps} Q^{3n+2\eps} T^{\frac{[F:\Q]n^2}{2} + 1+\epsilon},
    \\
    x^{1-\frac{2n_0^2+2+2\eps}{k_0}} \gg_{\eps} \mathfrak{C}_{\pi_0}^{n_0+\eps} T^{\frac{[F:\Q]n_0^2}{2} + 1+\epsilon}.
    \end{cases}
\end{equation}
For $\mathfrak{L}_{\text{large}}$, if $\N\p^k \in (x, e^{1/T}x]$, then by \eqref{eqn:LRS_2}, the bound $\lambda^\circ_{\pi \times \pi_0}(\p^k) \ll x^{\theta_n+\theta_{n_0}}$ holds.  Thus, if
   \begin{equation} \label{eqn:required-lower-bound-3}
    x^{2(1-\theta_n-\theta_{n_0} - \frac{1+\eps}{k_0})} \gg_{\eps} |\mathcal{S}|,
\end{equation}
then
\[
\begin{aligned}
    \mathfrak{L}_{\text{large}} 
    &\ll (\log x)^2 \sum_{\pi \in \mS}
    \Bigg( \sum_{\substack{\N \p^k \in (x, e^{1/T} x] \\ k > k_0}} \lambda^\circ_{\pi \times \pi_0}(\p^k) \Bigg)^2 \|a\|_\infty^2 
    \\
    &\ll  x^{2(\theta_n+\theta_{n_0})} (\log x)^2 |\mS|
    \Bigg(\sum_{k_0 < k \le \frac{\log(ex)}{\log 2}}~\sum_{\N \p \in (x^{1/k}, e^{1/T} x^{1/k}]} 1 \Bigg)^2 \|a\|_\infty^2 
    \\
    &\ll  x^{2(\theta_n+\theta_{n_0} + \frac{1}{k_0})} (\log x)^4 T^{-2} \|a\|_\infty^2 |\mS|\ll x^2 T^{-2} \|a\|_\infty^2.
\end{aligned}
\]
A small calculation using \eqref{eqn:ramanujan_progress} shows that if $k_0 = 2\tilde{n}^2 + 5.5\tilde{n}$ and $\eps \leq 10^{-3}$, then the constraints in \eqref{eqn:required-lower-bound-1}, \eqref{eqn:required-lower-bound-2}, and \eqref{eqn:required-lower-bound-3} hold simultaneously when $x$ satisfies the hypotheses of the corollary.
\end{proof}

\subsection{Notation and conventions}
\label{subsec:notation_LFZDE}

In keeping with the convention in \cref{cor:ls-log-derivative}, we let
\[
\mathcal{S}\subseteq\mathfrak{F}_n,\quad\pi\in\mathcal{S},\quad Q=\max_{\pi\in\mathcal{S}}\mathfrak{C}_{\pi},\quad\pi_0\in\mathfrak{F}_{n_0},\quad\tilde{n}=\max\{n,n_0\},\quad\tilde{Q}=\max\{Q,\mathfrak{C}_{\pi_0}\},\quad T\geq 2.
\]
For the rest of the paper, we use the following additional notation:
\begin{enumerate}[label=(\roman*)]
	\item $\mathcal{L}=\log(\tilde{Q}^{8\tilde{n}^3}T^{4[F:\Q]\tilde{n}^3})$,
	\item $\tau\in\R$ with $|\tau|\leq T$,
	\item $\alpha = 7.257\,570\,591\ldots$ and $A=3.893\,444\,953\ldots$,
	\item $R=\sqrt{A^2+1}$ and $\frac{1}{R\mathcal{L}}\leq \eta\leq \frac{1}{n_0 n R}$,
	\item $s_0=1+\eta+i\tau$,
	\item $\delta_{\pi\times\pi_0}$ the indicator function of the trivial primitive character,
	\item $\mathbf{1}_{I}(t)$ the indicator function of a subinterval $I\subseteq\R$,
	\item $\mathcal{N}_{\eta}=8A\eta\mathcal{L}+\Cl[abcon]{Linnik_err}$, where $\Cr{Linnik_err}=\Cr{Linnik_err}(n,n_0,[F:\Q])>0$ is sufficiently large.
	\item $M_{\eta}=(\alpha-1)\mathcal{N}_{\eta}$, which is at least $146$
	\item $k\in[M_{\eta},M_{\eta}+\mathcal{N}_{\eta}-1]\cap\Z\subseteq [M_{\eta},\frac{\alpha}{\alpha-1}M_{\eta}]\cap\Z$, which implies that $k\geq 146$,
	\item $V=2(4e\alpha)^{1/(\alpha-1)}+0.38=4.399\,815\,114\ldots$,
	\item $A_0 = 1/(eV)=0.083\,612\,477\ldots$,
	\item $A_{1} >2$ satisfies $V^{-1}=A_{1}  e^{1-\frac{A_{1} (\alpha-1)}{2\alpha}}$, so $A_{1}  = 11.401\,638\,518\,0\ldots$,
	\item $N_{\eta} = \exp(A_{0} M_{\eta}/\eta)$	and $N_{\eta}^{*}  = \exp(A_{1}  M_{\eta}/\eta)$,
	\item $\xi=1+10^{-7}$.
\end{enumerate}
The numerical values of $\alpha$ and $A$ are chosen to
\begin{equation}
\label{eqn:minimization_problem}
\textup{minimize $(4e\alpha 2^{\alpha-1})^A$ subject to $\alpha > 1$, $A>1$, and $4e\alpha(2/\sqrt{A^2+1})^{\alpha-1}=1-10^{-8}$,}
\end{equation}
resulting in the bound $2(4e\alpha)^{1/(\alpha-1)}<4.019\,815\,114\ldots$.

\subsection{Preliminaries}

We record some results that are already in the literature (or minor modifications thereof).  Throughout, we denote by $\rho=\beta+i\gamma$ a nontrivial zero of $L(s,\pi\times\pi_0)$.

\begin{lemma}
\label{lem:log_deriv}
If $k\geq 1$ is an integer and $s\in\mathbb{C}$ is not a zero of $L(s,\pi\times\pi_0)$, then
\begin{equation}
\label{eqn:Hadamard_0}
\sum_{\rho}\re\Big(\frac{1}{s-\rho}\Big)=\re\Big(\frac{L'}{L}(s,\pi\times\pi_0)+\frac{L_{\infty}'}{L_{\infty}}(s,\pi\times\pi_0)+\frac{\delta_{\pi\times\pi_0}}{s-1}+\frac{\delta_{\pi\times\pi_0}}{s}\Big)
\end{equation}
and
{\small\begin{equation}
\label{eqn:Hadamard_k}
\frac{(-1)^k}{k!}\Big(\frac{L'}{L}(s,\pi\times\pi_0)\Big)^{(k)}=\sum_{\rho}\frac{1}{(s-\rho)^{k+1}}-\frac{\delta_{\pi\times\pi_0}}{(s-1)^{k+1}}-\frac{\delta_{\pi\times\pi_0}}{s^{k+1}}-\sum_{j=1}^{n_0 n [F:\Q]}\sum_{m=0}^{\infty}\frac{1}{(2m+s+\mu_{\pi\times\pi_0}(j))^{k+1}}.\hspace{-1mm}
\end{equation}
}%
\end{lemma}
\begin{proof}
See \cite[(3.3) and Section 4]{soundararajan2019weak}, which are based on Hadamard factorizations.
\end{proof}

\begin{lemma}
\label{lem:Mertens}
If $(\pi,\pi_0)\in\mathfrak{F}_n\times\mathfrak{F}_{n_0}$ and $\eta>0$, then
\[
\sum_{\kn}\frac{|\Lambda_{\pi\times\pi_0}(\kn)|}{\N\kn^{1+\eta}}\leq \frac{1}{\eta}+\tilde{n}\log \tilde{Q}+O(1)\leq \frac{1}{\eta}+\frac{\mathcal{L}}{2}+O(1).
\]
\end{lemma}
\begin{proof}
When $\pi_0=\tilde{\pi}$, this follows from \eqref{eqn:BH} and the proof of \cite[Lemma~2.3]{soundararajan2019weak}.  When $\pi_0\neq\tilde{\pi}$, this follows from the preceding case and $|\Lambda_{\pi\times\pi'}(\kn)|^2\leq \Lambda_{\pi\times\tilde{\pi}}(\kn)\Lambda_{\pi'\times\tilde{\pi}'}(\kn)$ \cite[Proposition A.1]{soundararajan2019weak}.
\end{proof}

\begin{lemma}
\label{lem:Linnik_lemma}
If $(\pi,\pi_0)\in\mathfrak{F}_n\times\mathfrak{F}_{n_0}$, $0<\eta\leq 1$, and $t\in\R$, then
\begin{equation}
\label{eqn:Linnik_intermediate}
\sum_{\rho}\frac{1+\eta-\beta}{|1+\eta+it-\rho|^2}\leq \frac{1}{\eta}+2\tilde{n}\log\tilde{Q}+\frac{n_0 n [F:\Q]}{2}\log(|t|+3)+O(1)\leq \frac{1}{\eta}+\mathcal{L}+O(1).
\end{equation}
Also, if $\sigma\geq 1$ and $n_{\pi\times\pi_0}(\eta,\sigma+it)=\#\{\rho\colon |\sigma+it-\rho|\leq\eta\}$, then
\[
n_{\pi\times\pi_0}(\eta,\sigma+it) \leq n_{\pi\times\pi_0}(\eta,1+it)\leq n_{\pi\times\pi_0}(2\eta,1+\eta+it)\leq 4\eta\mathcal{L}+O(1).
\]
\end{lemma}
\begin{proof}
If $t\in\R$, $\sigma\geq 1$, and $0<\eta\leq 1$, then it follows from the geometry of complex numbers and the fact that $L(s,\pi\times\pi_0)\neq 0$ for $\re(s)\geq 1$ that
\[
n_{\pi\times\pi_0}(\eta,\sigma+it)\leq n_{\pi\times\pi_0}(\eta,1+it)\leq n_{\pi\times\pi_0}(2\eta,1+\eta+it)\leq 4\eta\sum_{\rho}\re\Big(\frac{1}{1+\eta+it-\rho}\Big).	
\]
Per \eqref{eqn:Hadamard_0}, the final sum (without the leading factor of $4\eta$) is
\[
\leq \frac{\log(D_F^{n_0 n}\N\kq_{\pi\times\pi_0})}{2}+\sum_{j=1}^{n_0 n [F:\Q]}\re\frac{\Gamma_{\R}'}{\Gamma_{\R}}\Big(\frac{1+\eta+it+\mu_{\pi\times\pi_0}(j)}{2}\Big)+\sum_{\kn}\frac{|\Lambda_{\pi\times\pi_0}(\kn)|}{\N\kn^{1+\eta}}+\delta_{\pi\times\pi_0}\Big(\frac{1}{\eta}+1\Big).
\]
The bound
\[
\frac{\log(D_F^{n_0 n}\N\kq_{\pi\times\pi_0})}{2}+\sum_{j=1}^{n_0 n [F:\Q]}\re\frac{\Gamma_{\R}'}{\Gamma_{\R}}\Big(\frac{1+\eta+it+\mu_{\pi\times\pi_0}(j)}{2}\Big)\leq \tilde{n}\log \tilde{Q}+\frac{n_0 n [F:\Q]}{2}\log T+O(1)
\]
follows from \cite[Lemma, p. 1203]{HIJTS} and \eqref{eqn:BH}.  The lemma now follows from \cref{lem:Mertens}.
\end{proof}

\subsection{Proof of \cref{thm:LFZDE}}

We proceed to detect the zeros of $L(s,\pi\times\pi_0)$ that satisfy the following hypothesis (which violates GRH).

\begin{hypothesis}
	\label{hyp}
There exists a nontrivial zero $\rho_0$ of $L(s,\pi\times\pi_0)$ such that $|1+i\tau-\rho_0|\leq \eta$.
\end{hypothesis}

Our main result en route to \cref{thm:LFZDE} is the following zero detection result.

\begin{theorem}
\label{thm:detect}
If \cref{hyp} is true, then
\[
1\ll e^{2.782471821 M_{\eta}}\Bigg(M_{\eta}\eta^3\int_{N_{\eta}}^{N_{\eta}^*}\Bigg|\sum_{N_{\eta}<\N\kn\leq u}\frac{\Lambda_{\pi\times\pi_0}(\kn)}{\N\kn^{1+i\tau}}\Bigg|^2\frac{du}{u}+\delta_{\pi\times\pi_0}\mathbf{1}_{[-A\eta,A\eta]}(\tau)\Bigg).
\]
\end{theorem}

\begin{proof}[Proof of \cref{thm:LFZDE} assuming \cref{thm:detect}]

In order to account for the zeros that lie on the circle $|1+i\tau-\rho|=\eta$, we replace each occurrence of $\eta$ with $\xi\eta$.  Therefore, if \cref{hyp} is true, then
\[
1\ll e^{2.782\,473 M_{\eta}}\Bigg(M_{\eta}\eta^3\int_{N_{\xi\eta}}^{N_{\xi\eta}^*}\Bigg|\sum_{N_{\xi\eta}<\N\kn\leq u}\frac{\Lambda_{\pi\times\pi_0}(\kn)}{\N\kn^{1+i\tau}}\Bigg|^2\frac{du}{u}+\delta_{\pi\times\pi_0}\mathbf{1}_{[-A\xi\eta,A\xi\eta]}(\tau)\Bigg).
\]
Let $\Phi_{\rho,w}(\tau)=1$ if $|1+i\tau-\rho|\leq w$, and $\Phi_{\rho,w}(\tau)=0$ otherwise.  We have that
\[
\Phi_{\rho_0,\xi\eta}(\tau)\ll e^{2.782\,473 M_{\eta}}\Bigg(M_{\eta}\eta^3\int_{N_{\xi\eta}}^{N_{\xi\eta}^*}\Bigg|\sum_{N_{\xi\eta}<\N\kn\leq u}\frac{\Lambda_{\pi\times\pi_0}(\kn)}{\N\kn^{1+i\tau}}\Bigg|^2\frac{du}{u}+\delta_{\pi\times\pi_0}\mathbf{1}_{[-A\xi\eta,A\xi\eta]}(\tau)\Bigg)\Phi_{\rho_0,\xi\eta}(\tau).
\]
Since \cref{hyp} implies that $\int_{-T}^T \Phi_{\rho_0,\xi\eta}(\tau)d\tau\geq \eta\sqrt{\xi^2-1}$, we can integrate both sides over $\tau\in[-T,T]$ and conclude that
\[
1\ll  e^{2.782\,473 M_{\eta}}\Bigg(M_{\eta}\eta^2\int_{-T}^{T}\int_{N_{\xi\eta}}^{N_{\xi\eta}^*}\Bigg|\sum_{N_{\xi\eta}<\N\kn\leq u}\frac{\Lambda_{\pi\times\pi_0}(\kn)}{\N\kn^{1+i\tau}}\Bigg|^2 \Phi_{\rho_0,\xi\eta}(\tau)\frac{du}{u}d\tau+\frac{\delta_{\pi\times\pi_0}}{\eta}\int_{-A\xi\eta}^{A\xi\eta}\Phi_{\rho_0,\xi\eta}(\tau)d\tau\Bigg).
\]

We now sum both sides over the zeros $\rho_0=\beta+i\gamma$ such that $1-\eta\leq \beta\leq 1$ and $|\gamma|\leq T$.  Since
\[
\sum_{\Lambda(\rho,\pi\times\pi_0)}\Phi_{\rho,\xi\eta}(\tau)=n_{\pi\times\pi_0}(\xi\eta,1+i\tau)\ll M_{\eta},
\]
we conclude that
\begin{align*}
N_{\pi\times\pi_0}(1-\eta,T)&\ll e^{2.782\,473 M_{\eta}}\Bigg(M_{\eta}^2 \eta^2\int_{-T}^{T}\int_{N_{\xi\eta}}^{N_{\xi\eta}^*}\Bigg|\sum_{N_{\xi\eta}<\N\kn\leq u}\frac{\Lambda_{\pi\times\pi_0}(\kn)}{\N\kn^{1+i\tau}}\Bigg|^2 \frac{du}{u}d\tau+\frac{\delta_{\pi\times\pi_0}}{\eta}M_{\eta}\int_{-A\xi\eta}^{A\xi\eta} d\tau\Bigg)\\
&\ll e^{2.782\,473 M_{\eta}}\Bigg(M_{\eta}^2 \eta^2\int_{-T}^{T}\int_{N_{\xi\eta}}^{N_{\xi\eta}^*}\Bigg|\sum_{N_{\xi\eta}<\N\kn\leq u}\frac{\Lambda_{\pi\times\pi_0}(\kn)}{\N\kn^{1+i\tau}}\Bigg|^2 \frac{du}{u}d\tau+\delta_{\pi\times\pi_0} M_{\eta}\Bigg).
\end{align*}
Summing over $\pi\in\mathcal{S}$ and swapping the order of integration, we arrive at
\begin{align*}
\sum_{\pi\in\mathcal{S}}N_{\pi\times\pi_0}(1-\eta,T)\ll e^{2.782\,473 M_{\eta}}\Bigg(M_{\eta}^2 \eta^2\int_{N_{\xi\eta}}^{N_{\xi\eta}^*}\sum_{\pi\in\mathcal{S}}\int_{-T}^{T}\Bigg|\sum_{N_{\xi\eta}<\N\kn\leq u}\frac{\Lambda_{\pi\times\pi_0}(\kn)}{\N\kn^{1+i\tau}}\Bigg|^2 d\tau\frac{du}{u}+M_{\eta}\Bigg).
\end{align*}
An application of Plancherel's theorem (as in the work of Gallagher \cite{Gallagher}) yields
\[
\sum_{\pi\in\mathcal{S}}\int_{-T}^{T}\Bigg|\sum_{N_{\xi\eta}<\N\kn\leq u}\frac{\Lambda_{\pi\times\pi_0}(\kn)}{\N\kn^{1+i\tau}}\Bigg|^2 d\tau\ll T^2 \int_0^{\infty}\sum_{\pi\in\mathcal{S}}\Bigg|\sum_{\substack{N_{\xi\eta}<\N\kn\leq u \\ x<\N\kn\leq xe^{1/T}}}\frac{\Lambda_{\pi\times\pi_0}(\kn)}{\N\kn}\Bigg|^2\frac{dx}{x}.
\]
An application of \cref{cor:ls-log-derivative} (with $a(\kn) = 1/\N\kn$ if $N_{\xi\eta}<\N\kn\leq u$ and $a(\kn)=0$ otherwise) yields
\[
T^2 \int_0^{\infty}\sum_{\pi\in\mathcal{S}}\Bigg|\sum_{\substack{N_{\xi\eta}<\N\kn\leq u \\ x<\N\kn\leq xe^{1/T}}}\frac{\Lambda_{\pi\times\pi_0}(\kn)}{\N\kn}\Bigg|^2\frac{dx}{x}\ll \int_{e^{-1/T}N_{\xi\eta}}^{u}\frac{dx}{x}\ll \log(e^{1/T}u/N_{\xi\eta}).
\]
Therefore, if $1/(R\mathcal{L})\leq\eta\leq 1/(R n_0 n)$, then
\begin{align*}
\sum_{\pi\in\mathcal{S}}N_{\pi\times\pi_0}(1-\eta,T)&\ll e^{2.782\,473 M_{\eta}}\Bigg(M_{\eta}^2 \eta^2\int_{N_{\xi\eta}}^{N_{\xi\eta}^*}\log(e^{1/T}u/N_{\xi\eta})\frac{du}{u}+M_{\eta}\Bigg)\\
&\ll e^{2.782\,473 M_{\eta}}(M_{\eta}^2 \eta^2(\log N_{\xi\eta})^2+M_{\eta})\ll e^{2.782\,473 M_{\eta}}M_{\eta}^4\ll (\tilde{Q}T^{[F:\Q]})^{543\tilde{n}^3 \eta}.
\end{align*}
Writing $\sigma=1-\eta$, we conclude that
\[
\sum_{\pi\in\mathcal{S}}N_{\pi\times\pi_0}(\sigma,T)\ll (\tilde{Q}T^{[F:\Q]})^{543\tilde{n}^3(1-\sigma)},\qquad 1-\frac{1}{Rn_0 n}\leq\sigma\leq 1-\frac{1}{R\mathcal{L}}.
\]
If $\sigma>1-1/(R\mathcal{L})$, then $\sum_{\pi\in\mathcal{S}}N_{\pi\times\pi_0}(\sigma,T)\leq \sum_{\pi\in\mathcal{S}}N_{\pi\times\pi_0}(1-\tfrac{1}{R\mathcal{L}},T)\ll 1$.  If $\sigma<1-1/(Rn_0 n)$, then the estimate is trivial compared to \cref{thm:ZDE}.
\end{proof}

\subsection{Lower bounds on high derivatives}

We apply  \eqref{eqn:Hadamard_k} to $s = s_0$.  Using \eqref{eqn:LRS_2}, we find that
\[
\frac{\delta_{\pi\times\pi_0}}{s_0^{k+1}}+\sum_{j=1}^{n_0 n [F:\Q]}\sum_{m=0}^{\infty}\frac{1}{(2m+s_0+\mu_{\pi\times\pi_0}(j))^{k+1}}\ll (n_0 n)^{k+1}\ll \frac{1}{(R\eta)^{k+1}}.
\]
Note that if $|1+i\tau-\rho|>A\eta$, then $|s_0-\rho|\geq R\eta$.  Therefore, by \eqref{eqn:Linnik_intermediate}, we have that
\[
\sum_{|1+i\tau-\rho|>A\eta}\frac{1}{|s_0-\rho|^{k+1}}\leq \frac{1}{(R\eta)^{k-1}}\sum_{\rho}\frac{1}{|s_0-\rho|^2}\leq\frac{R}{(R\eta)^{k}}\sum_{\rho}\re\Big(\frac{1}{s_0-\rho}\Big)\ll\frac{\eta\mathcal{L}}{(R\eta)^{k+1}}\ll \frac{\mathcal{N}_{\eta}}{(R\eta)^{k+1}}.
\]
If $\pi_0=\tilde{\pi}$, then we trivially estimate the contribution from the pole at $s=1$ by
\[
\Big|\frac{1}{(s_0-1)^{k+1}}\Big|\leq \begin{cases}
	\eta^{-k-1}&\mbox{if $|\tau|\leq A\eta$,}\\
	(R\eta)^{-k-1}&\mbox{if $|\tau|>A\eta$.}
\end{cases}
\]
Since $\mathcal{N}_{\eta}\gg 1$, we infer (upon multiplying through by $\eta^{k+1}$) that there exists an effectively computable constant $\Cl[abcon]{far_zeros}=\Cr{far_zeros}(n,n_0,[F:\Q])\geq 1.007$ such that
\begin{equation}
\label{eqn:lower_1}
\Big|\frac{\eta^{k+1}}{k!}\Big(\frac{L'}{L}\Big)^{(k)}(s_0,\pi\times\pi_0)\Big|+\delta_{\pi\times\pi_0}\mathbf{1}_{[-A\eta,A\eta]}(\tau)\geq \Bigg|\sum_{|1+i\tau-\rho|\leq A\eta}\frac{\eta^{k+1}}{(s_0-\rho)^{k+1}}\Bigg|-\Cr{far_zeros}\frac{\mathcal{N}_{\eta}}{R^{k+1}}.
\end{equation}

By \cref{hyp}, the sum over zeros in \eqref{eqn:lower_1} contains $\rho_0$.  Now, we may apply \cref{prop:Turan} to the sum over zeros in \eqref{eqn:lower_1} with $z_1 = \eta/(s_0-\rho)$, which, by \cref{hyp}, satisfies $|z_1|\geq\frac{1}{2}$.  By \cref{lem:Linnik_lemma}, there are $\leq \mathcal{N}_{\eta}$ zeros in this sum.  We choose $M=M_{\eta}$, in which case there exists an integer $k\in[M_{\eta}, M_{\eta}+\mathcal{N}_{\eta}-1]\subseteq [(\alpha-1)\mathcal{N}_{\eta},\alpha\mathcal{N}_{\eta}]$ such that the right-hand side of \eqref{eqn:lower_1} is
\begin{align*}
\geq \frac{1.007}{2^{k+1}(4e\alpha)^{\mathcal{N}_{\eta}}}-\Cr{far_zeros}\frac{\mathcal{N}_{\eta}}{R^{k+1}}&\geq 1.007\Big(\frac{1}{4e\alpha 2^{\alpha-1}}\Big)^{\mathcal{N}_{\eta}}\Big(1-\frac{\Cr{far_zeros}}{1.007}\Big(4e\alpha\Big(\frac{2}{\sqrt{A^2+1}}\Big)^{\alpha-1}\Big)^{\mathcal{N}_{\eta}}\mathcal{N}_{\eta}\Big)\\
&\geq 1.007\Big(\frac{1}{2(4e\alpha)^{1/(\alpha-1)}}\Big)^{M_{\eta}}\Big(1-\frac{\Cr{far_zeros}}{1.007}\mathcal{N}_{\eta}\Big(4e\alpha\Big(\frac{2}{\sqrt{A^2+1}}\Big)^{\alpha-1}\Big)^{\frac{\Cr{far_zeros}}{1.007}\mathcal{N}_{\eta}}\Big).
\end{align*}
Our choices of $\alpha$ and $A$ (see \eqref{eqn:minimization_problem} and the line that follows it) ensure that
\begin{align*}
\Big|\sum_{|1+i\tau-\rho|\leq A\eta}\frac{\eta^{k+1}}{(s_0-\rho)^{k+1}}\Big|-\Cr{far_zeros}\frac{\mathcal{N}_{\eta}}{R^{k+1}}
&\geq 1.007\Big(\frac{1}{4.019\,815\,115}\Big)^{M_{\eta}}\Big(1-\frac{\Cr{far_zeros}}{1.007}\mathcal{N}_{\eta}(1-10^{-8})^{\frac{\Cr{far_zeros}}{1.007}\mathcal{N}_{\eta}}\Big).
\end{align*}
Without loss of generality, we may assume that $\Cr{Linnik_err}\geq 10^{10}\Cr{far_zeros}$, which ensures that
\begin{equation}
\label{eqn:lower_bound_high_deriv}
\Bigg|\frac{\eta^{k+1}}{k!}\Big(\frac{L'}{L}\Big)^{(k)}(s_0,\pi\times\pi_0)\Bigg|+\delta_{\pi\times\pi_0}\mathbf{1}_{[-A\eta,A\eta]}(\tau)\geq \Big(\frac{1}{4.019\,815\,115}\Big)^{M_{\eta}}.
\end{equation}

\subsection{Upper bounds on high derivatives}

Let $k\in[M_{\eta},\frac{\alpha}{\alpha-1}M_{\eta}]$.  If $j_k(u)=e^{-u}u^k/k!$, then
\begin{equation}
\label{eqn:DirichletExpansion}
\Bigg|\frac{\eta^{k+1}}{k!}\Big(\frac{L'}{L}\Big)^{(k)}(s_0,\pi\times\pi_0)\Bigg|=\eta\Bigg|\sum_{\kn}\frac{\Lambda_{\pi\times\pi_0}(\kn)}{\N\kn^{1+\eta+i\tau}}\cdot\frac{(\eta\log\N\kn)^k}{k!}\Bigg|=\eta\Bigg|\sum_{\kn}\frac{\Lambda_{\pi\times\pi_0}(\kn)}{\N\kn^{1+i\tau}}j_k(\eta\log \N\kn)\Bigg|.
\end{equation}
We split the sum over $\kn$ into the ranges $\N\kn\in[N_{\eta},N_{\eta}^*]$ and $\N\kn\notin[N_{\eta},N_{\eta}^*]$.  We use partial summation in the former range and trivially handle the latter range, thus showing that \eqref{eqn:DirichletExpansion} is
\begin{equation}
\begin{aligned}
&\leq \eta\Big(\sum_{\N\kn\notin[N_{\eta},N_{\eta}^*]}\frac{|\Lambda_{\pi\times\pi_0}(\kn)|}{\N\kn}j_k(\eta\log\N\kn)+\sum_{\N\kn\in [N_{\eta},N_{\eta}^*]}\frac{|\Lambda_{\pi\times\pi_0}(\kn)|}{\N\kn}j_k(\eta\log N_{\eta}^*)\\
&+\int_{N_{\eta}}^{N_{\eta}^*}\Bigg|\frac{d}{du}j_k(\eta\log u)\Bigg|\cdot\Bigg|\sum_{\N\kn\in[N_{\eta},u]}\frac{\Lambda_{\pi\times\pi_0}(\kn)}{\N\kn^{1+i\tau}}\Bigg|du\Big).
\end{aligned}
\end{equation}

If $\N\kn\leq \mathcal{N}_{\eta}$, then $\eta\log\N\kn\leq A_0 M_{\eta}$.  Since $k\in[M_{\eta},\frac{\alpha}{\alpha-1}M_{\eta}]$ and $k!\geq (k/e)^k$, we observe that
\[
j_k(\eta\log\N\kn)=\frac{\N\kn^{-\eta}(\eta\log\N\kn)^k}{k!}\leq\N\kn^{-\eta}\Big(\frac{e\eta\log\N\kn}{k}\Big)^k\leq \N\kn^{-\eta}\Big(\frac{A_0 e M_{\eta}}{k}\Big)^k \leq \N\kn^{-\eta}V^{-k}.
\]
If $\N\kn\geq N_{\eta}^*$, then $\eta\log\N\kn\geq A_1 M_{\eta}$.  Since $e^{-u/2}u^k/k!$ is decreasing in the range $u>2k$, we find that if $\N\kn\geq N_1$, then
\[
j_k(\eta\log\N\kn)=\N\kn^{-\frac{\eta}{2}}\frac{e^{-\frac{1}{2}\eta\log\N\kn}(\eta\log\N\kn)^k}{k!}\leq \N\kn^{-\frac{\eta}{2}}\frac{e^{-\frac{A_1}{2}M_{\eta}}(A_1 M_{\eta})^k}{k!}\leq \N\kn^{-\frac{\eta}{2}}V^{-k}.
\]
Using these two bounds on $j_k(\eta\log\N\kn)$, we conclude via \cref{lem:Mertens} that
\[
\eta\Bigg|\sum_{\N\kn\notin[N_{\eta},N_{\eta}^*]}\frac{\Lambda_{\pi\times\pi_0}}{\N\kn^{1+i\tau}}j_k(\eta\log\N\kn)\Bigg|\leq \frac{\eta}{V^k}\sum_{\kn}|\Lambda_{\pi\times\pi_0}(\kn)|\Big(\frac{1}{\N\kn^{1+\eta}}+\frac{1}{\N\kn^{1+\eta/2}}\Big)\ll\frac{k}{V^k}.
\]
Now, using $|\frac{d}{du}(j_k(\eta\log u))|=|j_{k-1}(\eta\log u)-j_k(\eta\log u)|\leq \eta/u$, we infer that
\[
\eta\int_{N_{\eta}}^{N_{\eta}^*}\Big|\frac{d}{du}j_k(\eta\log u)\Big|\cdot\Bigg|\sum_{\N\kn\in[N_{\eta},u]}\frac{\Lambda_{\pi\times\pi_0}(\kn)}{\N\kn^{1+i\tau}}\Bigg|du\leq \eta^2 \int_{N_{\eta}}^{N_{\eta}^*}\Bigg|\sum_{\N\kn\in[N_{\eta},u]}\frac{\Lambda_{\pi\times\pi_0}(\kn)}{\N\kn^{1+i\tau}}\Bigg|\frac{du}{u}.
\]
In conclusion, there exists an effectively computable constant $\Cl[abcon]{Dirichlet_upper}=\Cr{Dirichlet_upper}(n,n_0,[F:\Q])\geq 2$ such that
\begin{equation}
\label{eqn:upper_bound_high_deriv}
\Bigg|\frac{\eta^{k+1}}{k!}\Big(\frac{L'}{L}\Big)^{(k)}(s_0,\pi\times\pi_0)\Bigg|\leq \eta^2\int_{N_{\eta}}^{N_{\eta}^*}\Bigg|\sum_{N_{\eta}<\N\kn\leq u}\frac{\Lambda_{\pi\times\pi_0}(\kn)}{\N\kn^{1+i\tau}}\Bigg|\frac{du}{u}+\Cr{Dirichlet_upper}\frac{k}{V^k}.
\end{equation}

\subsection{Proof of \cref{thm:detect}}

It follows from \eqref{eqn:lower_bound_high_deriv} and \eqref{eqn:upper_bound_high_deriv} that
\[
\Big(\frac{1}{4.019\,815\,115}\Big)^{M_{\eta}}\leq \eta^2\int_{N_{\eta}}^{N_{\eta}^*}\Bigg|\sum_{N_{\eta}<\N\kn\leq u}\frac{\Lambda_{\pi\times\pi_0}(\kn)}{\N\kn^{1+i\tau}}\Bigg|\frac{du}{u}+\Cr{Dirichlet_upper}\frac{k}{V^k}+\delta_{\pi\times\pi_0}\mathbf{1}_{[-A\eta,A\eta]}(\tau)
\]
Since $k\geq M_{\eta}$, it follows (upon inflating $\Cr{Linnik_err}$ to satisfy $\Cr{Linnik_err}\geq 10^{10}\Cr{Dirichlet_upper}$, if needed) that
\[
\frac{1}{2}\Big(\frac{1}{4.019\,815\,115}\Big)^{M_{\eta}}\leq \eta^2\int_{N_{\eta}}^{N_{\eta}^*}\Bigg|\sum_{N_{\eta}<\N\kn\leq u}\frac{\Lambda_{\pi\times\pi_0}(\kn)}{\N\kn^{1+i\tau}}\Bigg|\frac{du}{u}+\delta_{\pi\times\pi_0}\mathbf{1}_{[-A\eta,A\eta]}(\tau),
\]
hence
\[
1\ll e^{1.391235911 M_{\eta}}\Bigg(\eta^2\int_{N_{\eta}}^{N_{\eta}^*}\Bigg|\sum_{N_{\eta}<\N\kn\leq u}\frac{\Lambda_{\pi\times\pi_0}(\kn)}{\N\kn^{1+i\tau}}\Bigg|\frac{du}{u}+\delta_{\pi\times\pi_0}\mathbf{1}_{[-A\eta,A\eta]}(\tau)\Bigg).
\]
We square both sides and apply the Cauchy--Schwarz inequality, thus obtaining
\begin{align*}
1&\ll e^{2.782471821 M_{\eta}}\Bigg(\eta^4\Bigg(\int_{N_{\eta}}^{N_{\eta}^*}\Bigg|\sum_{N_{\eta}<\N\kn\leq u}\frac{\Lambda_{\pi\times\pi_0}(\kn)}{\N\kn^{1+i\tau}}\Bigg|\frac{du}{u}\Bigg)^2+\delta_{\pi\times\pi_0}\mathbf{1}_{[-A\eta,A\eta]}(\tau)\Bigg)\\
&\ll e^{2.782471821 M_{\eta}}\Bigg(\eta^4\Bigg(\int_{N_{\eta}}^{N_{\eta}^*}\frac{du}{u}\Bigg)\int_{N_{\eta}}^{N_{\eta}^*}\Bigg|\sum_{N_{\eta}<\N\kn\leq u}\frac{\Lambda_{\pi\times\pi_0}(\kn)}{\N\kn^{1+i\tau}}\Bigg|^2\frac{du}{u}+\delta_{\pi\times\pi_0}\mathbf{1}_{[-A\eta,A\eta]}(\tau)\Bigg)\\
&\ll e^{2.782471821 M_{\eta}}\Bigg(M_{\eta}\eta^3\int_{N_{\eta}}^{N_{\eta}^*}\Bigg|\sum_{N_{\eta}<\N\kn\leq u}\frac{\Lambda_{\pi\times\pi_0}(\kn)}{\N\kn^{1+i\tau}}\Bigg|^2\frac{du}{u}+\delta_{\pi\times\pi_0}\mathbf{1}_{[-A\eta,A\eta]}(\tau)\Bigg).
\end{align*}

\bibliographystyle{abbrv}
\bibliography{PascadiThorner_LargeSieve}
\end{document}